\documentclass[11pt]{amsart}
\usepackage[utf8]{inputenc}

\usepackage{amssymb}

\usepackage{bm}
\usepackage{graphicx}
\usepackage[centertags]{amsmath}
\usepackage{amsfonts}
\usepackage{amsthm}
\usepackage{enumerate}
\usepackage{tocvsec2}
\usepackage{xcolor}
\usepackage[margin=1in]{geometry}

\newtheorem{theorem}{Theorem}[section]
\newtheorem*{theorem*}{Theorem}

\newtheorem{corollary}[theorem]{Corollary}
\newtheorem{lemma}[theorem]{Lemma}

\newtheorem{proposition}[theorem]{Proposition}
\newtheorem{claim}[theorem]{Claim}
\newtheorem{fact}{Fact}[section]

\newtheorem{example}{Example}[section]
\newtheorem{question}{Question}[section]
\theoremstyle{definition}

\newenvironment{remark}[1][Remark]{\begin{trivlist}
\item[\hskip \labelsep {\bfseries #1}]}{\end{trivlist}}

\newcommand{\cts}{C(2^\nn)}

\newcommand{\infin}{[\mathbb{N}]}

\newcommand{\rr}{\mathbb{R}}
\newcommand{\nn}{\mathbb{N}}

\newcommand{\ee}{\varepsilon}

\newcommand{\aaa}{\mathcal{A}}
\newcommand{\bbb}{\mathcal{B}}

\newcommand{\fff}{\mathcal{F}}
\newcommand{\gggg}{\mathcal{G}}

\newcommand{\sss}{\mathcal{S}}

\newcommand{\nnn}{\mathcal{N}}

\newcommand{\supp}{\mathrm{supp}}

\DeclareMathOperator{\sgn}{sgn}

\newcommand{\uuu}{\mathcal{U}}
\newcommand{\ord}{\mathbf{Ord}}
\newcommand{\Ban}{\mathbf{Ban}}
\newcommand{\rank}{\text{rank}}

\newcommand{\mt}{\mathcal{MT}}
\newcommand{\ttt}{\mathcal{T}}
\newcommand{\qn}{(\mathbb{Q}\times \mathbb{N})^{<\nn}}

\begin{document}

\begin{abstract} We introduce and study the Bourgain index of an operator between two Banach spaces. In particular, we study the Bourgain $\ell_p$ and $c_0$ indices of an operator.  Several estimates for finite and infinite direct sums are established.  We define classes determined by these indices and show that some of these classes form operator ideals.  We characterize the ordinals which occur as the index of an operator and establish exactly when the defined classes are closed.  We study associated indices for non-preservation of $\ell_p^\xi$ and $c_0^\xi$ spreading models and indices characterizing weak compactness of operators between separable Banach spaces.  We also show that some of these classes are operator ideals and discuss closedness and distinctness of these classes.

\end{abstract}

\title[Classes determined by ordinal indices]{Classes of operators \\ determined by ordinal indices}

\author{Kevin Beanland}
\address{Department of Mathematics, Washington and Lee University, Lexington, VA 24450}
\email{beanlandk@wlu.edu}

\author{Ryan Causey}
\address{Department of Mathematics, University of South Carolina, Columbia, SC 29208}
\email{causeyrm@mailbox.sc.edu}

\author{Daniel Freeman}
\address{Department of Mathematics and Computer Science, Saint Louis University, St. Louis, MO 63103}
\email{dfreema7@slu.edu}

\author{Ben Wallis}
\address{Department of Mathematical Sciences, Northern Illinois University, DeKalb, IL
60115}
\email{wallis@math.niu.edu}

\thanks{2010 \textit{Mathematics Subject Classification}. Primary: 46B28; Secondary: 03E15.}
\thanks{\textit{Key words}: operator ideals, ordinal ranks.}

\maketitle

\tableofcontents

\addtocontents{toc}{\setcounter{tocdepth}{1}}

\section{Introduction}

The Bourgain $\ell_1$ index uses trees and ordinal numbers as a way
of quantifying the representation of the unit vector basis for
$\ell_1$ in a Banach space \cite{Bo}.  The larger the Bourgain
$\ell_1$ index of a Banach space, the better represented the unit
vector basis for $\ell_1$ is in that space. In particular, a
separable Banach space has countable Bourgain $\ell_1$ index if and
only if $\ell_1$ does not embed into the space.  It was quickly
realized that the analagous index for other bases could provide
useful results as well. For instance, Bourgain  used the
corresponding index for a basis of $C(2^\nn)$ to prove that if $X$
is a separable Banach space such that every separable reflexive
Banach space embeds into $X$ then every separable Banach space
embeds into $X$ as well \cite{Bo1}. Given a basic sequence
$(e_i)_{i=1}^\infty$, we introduce in Section 3 an ordinal index of
operators between Banach spaces which quantifies the property of an
operator not preserving the basis $(e_i)_{i=1}^\infty$. We call this
the non-preservation $(e_i)_{i=1}^\infty$ index, and it is a natural
generalization of the Bourgain $(e_i)_{i=1}^\infty$ index of a
Banach space in the sense that the Bourgain index of a Banach space
is the non-preservation index of the identity operator on that
space.

The classification of operator ideals is a fundamental area of
research in the study of operators on Banach spaces, and thus when a
new operator property is introduced, it is natural to consider its
connection to operator ideals. In \cite{ADST}, an ordinal index is
constructed which quantifies the property of an operator being
strictly singular.  It was hoped that this index could be used to
define new operator ideals, but  an example was later given of two
$\sss_1$-strictly singular operators whose sum was not $\sss_1$
strictly singular \cite{OT}.  It is unknown if the
$\sss_\zeta$-strictly singular operators form an ideal for any
countable ordinal $\zeta$, but  it follows from Proposition 2.4 in
\cite{ADST} that for each countable ordinal $\zeta$ the set of
operators whose strictly singular index is less than $\omega^\zeta$
forms an ideal.  For each $1\leqslant p \leqslant\infty$ and ordinal
$\zeta$, we let $\mathfrak{NP}_p^{\zeta}$ denote the set of
operators whose non-preserving $(e_i)$ index is at most $\zeta$
where $(e_i)$ is the unit vector basis for $\ell_p$ (or $c_0$ in the
case that $p=\infty$).  For each ordinal $\zeta$, we prove that
$\mathfrak{NP}_1^{\omega^\zeta}$ is a closed operator ideal and that
for each infinite ordinal $\zeta$, $\mathfrak{NP}_\infty^{\zeta}$ is
a closed operator ideal. For the other cases of $1\leqslant
p\leqslant \infty$, we have that $\cup_{\xi<\omega^\zeta}
\mathfrak{NP}_p^{\xi}$ is an operator ideal.

 The higher order spreading models, use higher order Schreier sets to measure the
asymptotic structure of a sequence in a Banach space.  Given
$1\leqslant p\leqslant \infty$ and a countable ordinal $\xi$, the
existence of an $\ell_p^\xi$ spreading model in a Banach space $X$
is a strong measurement of the representation of $\ell_p$ in $X$. In
particular, if $X$ contains  an $\ell_p^\xi$ spreading model then
the Bourgain index of $X$ is at least $\omega^\xi$ (the order of the
Schreier-$\xi$ family), but we show in Section 7 that there exist
Banach spaces whose Bourgain index is at least $\omega^\xi$ and
which do not contain even an $\ell_p^1$ spreading model. We let
$\mathfrak{SM}_p^\xi$ denote the set of bounded operators which
don't preserve any $\ell_p^\xi$ spreading model.  We prove that for
all countable ordinals $\xi$, both $\mathfrak{SM}_1^\xi$ and
$\mathfrak{SM}_\infty^\xi$ are closed operator ideals, and that for
$1\leqslant p\leqslant \infty$ we have that $\cup_{\xi<\omega^\zeta}
\mathfrak{SM}_p^{\xi}$ is an operator ideal.  In \cite{BF}, an
ordinal index is constructed to measure the weak compactness of an
operator in an analogous way to how strictly singularity is measured
in \cite{ADST}.  We prove that for every countable ordinal $\xi$, an
operator $A$ is $\sss_\xi$ weakly compact if and only if $A$ is
weakly compact and $A\in\mathfrak{SM}_1^{\xi}$.  Thus, we have that
the set of $\sss_\xi$ weakly compact operators forms a closed ideal.

So far, all of the ideals we have considered are constructed using
the unit vector basis $(e_i)$ for $\ell_p$ or $c_0$. In these cases,
for any bounded operators $A$ and $B$, we are able to obtain
explicit bounds for the non-preservation $(e_i)$ index of $A+B$ in
terms of the individual indexes of $A$ and $B$.  It is natural to
ask what can be proven for other bases. Unfortunately, our proofs
implicitly make use of the fact that the unit vector basis for
$\ell_p$ or $c_0$ is equivalent to all its normalized block bases.
Thus, our proofs cannot be generalized to any other basic sequences.
However, given a basic sequence $(e_i)$, we may not be able to
explicitly calculate a bound for the non-preservation $(e_i)$ index
of $A+B$ in terms of the individual
 indexes of $A$ and $B$, but we would like to know if such a bound
 exists.  In section 8 we introduce a property $(S')$ analogous to
 Dodos' property $(S)$ \cite{D} and use descriptive set theory
 techniques to prove that if $(e_i)$ is a Schauder basis with
 property $(S')$ then there is a function
 $\psi_{(e_i)}:[1,\omega_1)\rightarrow[1,\omega_1)$ so that for every countable ordinal $\xi$, if $X$ and
 $Y$ are separable Banach spaces and $A$ and $B$ are bounded
 operators from $X$ to $Y$ whose non-preserving $(e_i)$ index is at
 most $\xi$ then the non-preserving $(e_i)$ index of $A+B$ is at
 most $\psi(\xi)$.

\section{Trees, orders, and combinatorial lemmas}

\subsection{Minimal trees and Schreier families}

Throughout, we let $\Ban$ denote the class of Banach spaces, $\textbf{SB}$ the class of separable Banach spaces, $\ord$ the class of ordinal numbers. For $X,Y\in \Ban$, we will let $\mathfrak{L}(X,Y)$ denote the bounded, linear operators, in the sequel referred to simply as operators, from $X$ to $Y$.   We let $\omega$ (resp. $\omega_1$) denote the first infinite (resp. uncountable) ordinal.

If $X$ is a Banach space, we let $S_X$, $B_X$ denote the unit sphere and unit ball of $X$, respectively.  For a subset $S$ of a Banach space, we let $[S]$ denote the closed span of $S$.  For $K\geqslant 1$ and a (finite or infinite) sequence $(x_i)$ in a Banach space, we say $(x_i)$ is $K$-basic if for all scalar sequences $(a_i)$ and all $m\leqslant n$, $n$ not exceeding the length of $(x_i)$, $$\|\sum_{i=1}^m a_ix_i\|\leqslant K\|\sum_{i=1}^n a_ix_i\|.$$ The basis constant of a basic sequence is the smallest $K$ so that the sequence is $K$-basic.  

If $(e_i), (f_i)$ are sequences of the same length in (possibly different) Banach spaces, we say $(e_i)$ is $K$-dominated by $(f_i)$ if for all scalar sequences $(a_i)$ (finitely non-zero in the case that $(e_i)$ and $(f_i)$ are infinite), $$\|\sum a_ie_i\|\leqslant K\|\sum a_if_i\|.$$  In this case, we will write $(e_i)\lesssim_K (f_i)$.  We write $(e_i)\approx_K (f_i)$ to mean that there exist $a, b>0$ with $ab\leqslant K$ so that $(e_i)\lesssim_a (f_i)$ and $(f_i)\lesssim_b (e_i)$.  

If $E\in \Ban$, by an \emph{unconditional basis for} $E$, we shall mean an unordered, not necessarily countable set of vectors $(e_i)_{i\in I}\subset E$ so that each $x\in E$ has a unique representation $x=\sum a_ie_i$, with $\{i\in I: a_i\neq 0\}$ countable and $\sum a_ie_i$ unconditionally converging to $x$. We recall the definition of the coordinate functionals $(e_i^*)_{i\in I}\subset E^*$ corresponding to $(e_i)_{i\in I}$.  If $x=\sum_{i\in I}a_ie_i\in E$, and if $j\in I$, $e_j^*(x)=a_j$.  We recall that if $(e_i)_{i\in I}$ is an unconditional (resp. $1$-unconditional) basis for $E$, $(e_i^*)_{i\in I}$ is an unconditional (resp. $1$-unconditional) basis for its closed span in $E^*$.  We say $(e_i)_{i\in I}$ is \emph{shrinking} if $(e_i^*)_{i\in I}$ is a basis for $E^*$. This is equivalent to $(e_i^*)_{i\in I}$ having dense span in $E^*$, and equivalent to $E$ not containing an isomorphic copy of $\ell_1$.   We also recall the definition of the $p$-\emph{convexification} of a Banach space with $1$-unconditional basis.  If $E$ is a Banach space and $(e_i)_{i\in I}$ is a $1$-unconditional basis for $E$, for $1\leqslant p<\infty$, the $p$-convexification $E^p$ of $E$ is given by $$E^p=\Bigl\{\sum a_ie_i: \sum|a_i|^pe_i\in E\Bigr\}.$$  This is a Banach space when endowed with the norm $$\bigl\|\sum a_ie_i\bigr\|_{E^p}= \bigl\|\sum |a_i|^pe_i\bigr\|_E^{1/p}.$$ Often we will refer to the $p$-convexification of a Banach space $E$ having an unconditional basis $(e_i)$ without assuming the basis is $1$-unconditional.  In these instances, we will mean the $p$-convexification of $E$ with its equivalent norm $\|\cdot\|_0$ defined by $\|\sum a_ie_i\|_0=\sup_{|\ee_i|=1} \|\sum \ee_i a_ie_i\|$.  

If $\Lambda$ is a set, we let $\Lambda^\nn$ (resp. $\Lambda^{<\nn}$) denote the infinite (resp. finite) sequences in $\Lambda$, including the empty sequence.  If $s=(x_i)_{i=1}^n\in \Lambda$, we let $|s|=n$ and let $s|_k=(x_i)_{i=1}^k$ for any $0\leqslant k\leqslant n$.  We order $\Lambda^{<\nn}$ by letting $s\preceq t$ if $s=t|_{|s|}$, and in this case we say $s$ is an \emph{initial segment} of $t$, and that $t$ is an \emph{extension} of $s$.  For $s, t\in \Lambda^{<\nn}$, we let $s\verb!^!t$ denote the concatenation of $s$ with $t$ listing the members of $s$ first.  If $T\subset \Lambda^{<\nn}$ is downward closed with respect to the order $\preceq$, we say $T$ is a \emph{tree on} $\Lambda$ or, if the set $\Lambda$ is understood, simply a \emph{tree}.  If $T$ is a tree on $\Lambda$ and $t\in \Lambda^{<\nn}$, we let $$T(t)=\{s\in \Lambda^{<\nn}: t\verb!^!s\in T\}.$$  Note that $T(t)$ is a (possibly empty) tree on $\Lambda$.  We refer to the non-empty, linearly ordered subsets of a tree as \emph{chains} of the tree.  

If $T$ is a tree, we let $T'= T\setminus MAX(T)$, where $MAX(T)$ is the set of members of $T$ which are maximal with respect to $\preceq$.  By transfinite induction, we define the higher order derived trees $T^\xi$ of $T$ for each $\xi\in \ord$.  We let $$T^0=T,$$ $$T^{\xi+1}=(T^\xi)',$$ and if $T^\zeta$ has been defined for each $\zeta<\xi$, $\xi$ a limit ordinal, $$T^\xi= \bigcap_{\zeta<\xi}T^\zeta.$$  Note that for any $\xi\in \ord$ and $t\in \Lambda^{<\nn}$, $(T^\xi)(t)= (T(t))^\xi$, which can be shown by a standard induction argument.  Another fact easily verified by induction is that for any tree $T$ and any $\xi, \zeta\in \ord$, $(T^\zeta)^\xi=T^{\zeta+\xi}$.  

Of course, if $\zeta<\xi$, $T^\xi\subset T^\zeta$, and there must exist some $\xi\in \ord$ so that $T^\xi=T^{\xi+1}$.  If there exists $\xi\in \ord$ so that $T^\xi=\varnothing$, we say $T$ is \emph{well-founded}, and let $o(T)=\min \{\xi\in \ord: T^\xi=\varnothing\}$.  Otherwise, there exists $\xi\in \ord$ so that $T^\xi=T^{\xi+1}\neq \varnothing$, and in this case we say $T$ is \emph{ill-founded}, and we write $o(T)=\infty$.  By convention, we will declare that for any $\xi\in \ord \cup\{\infty\}$, $\xi\infty = \infty\xi=\infty$, and $\xi+\infty= \infty+\xi=\infty$. We also declare that $\xi<\infty$ for any $\xi\in \ord$. Note that $T$ is ill-founded if and only if there exists $(x_i)\in \Lambda^\nn$ so that $(x_i)_{i=1}^n \in T$ for all $n\in \nn$.  

For $1\leqslant p\leqslant \infty$, if $(x_i)_{i=1}^n$ is a sequence in a Banach space, we say $(y_i)_{i=1}^m$ is a $p$-\emph{absolutely convex block} of $(x_i)_{i=1}^n$ provided there exist $0=k_0<k_1<\ldots <k_m\leqslant n$ and scalars $(a_i)_{i=1}^n$ so that for each $1\leqslant j\leqslant m$, $(a_i)_{i= k_{j-1}+1}^{k_j}$ has norm $1$ in $\ell_p^{k_j-k_{j-1}}$ and $y_j=\sum_{i=k_{j-1}+1}^{k_j} a_ix_i$.  If $\Lambda$ is a subset of a Banach space, and if $T$ is a tree on $\Lambda$, we say $T$ is $p$-\emph{absolutely convex} if any $p$-absolutely convex block of a member of $T$ is also a member of $T$.  We will call $T$ \emph{block closed} if every normalized block of a member of $T$ is also a member of $T$.  

If $\Lambda$ is a set and $T\subset \Lambda^{<\nn}\setminus \{\varnothing\}$, we say $T$ is a $B$-tree if $T\cup \{\varnothing\}$ is a tree. Each of the notions for trees above can also be applied to $B$-trees.  The presentation of the main results of this work is significantly improved by including the empty sequence in the considerations, but the presentation of the proofs is much improved by only considering $B$-trees.  For this reason, we will readily use both.  We will define a collection of tree $\mt_\xi$, $\xi\in \ord$, and associated $B$-trees which will be useful in our considerations for witnessing the orders of given trees, in a sense which will be made apparent in the following proposition.

Let $$\mt_0=\{\varnothing\},$$ $$\mt_{\xi+1}=\{\varnothing\}\cup \{(\xi+1)\verb!^!t: t\in \mt_\xi\},$$ and if $\mt_\zeta$ has been defined for every ordinal $\zeta$ less than a limit ordinal $\xi$, we let $$\mt_\xi=\bigcup_{\zeta<\xi} \mt_{\zeta+1}.$$  Let $\ttt_\xi=\mt_\xi\setminus \{\varnothing\}$.  The following items are easily checked.  

\begin{proposition}\cite{C2}Fix $\xi\in \ord$.  \begin{enumerate}[(i)]\item $\ttt_\xi$ is a $B$-tree on $[1,\xi]$ with $o(\ttt_\xi)=\xi$.  \item If $\Lambda$ is any set and $T$ is a tree on $\Lambda$, then $o(T)>\xi$ if and only if there exists a function $f:\ttt_\xi\to \Lambda$ so that for each $t\in \ttt_\xi$, $(f(t|_i))_{i=1}^{|t|}\in T$.

\end{enumerate}
\label{indexing}

\end{proposition}

For $\zeta, \xi\in \ord$, we will say a function $g:\ttt_\zeta\to \ttt_\xi$ is \emph{monotone} if for each $s, t\in \ttt_\zeta$ with $s\prec t$, $g(s)\prec g(t)$.  If $h$ is a function mapping $\ttt_\zeta$ into the chains of $\ttt_\xi$, we will call $h$ a \emph{block map} if for each $s,t\in \ttt_\zeta$ with $s\prec t$, and for all $s'\in h(s)$, $t'\in h(t)$, $s' \prec t'$.  That is, if $h$ is a block map, each branch $(t|_i)_{i=1}^{|t|}$ of $\ttt_\zeta$ will be mapped to successive chains lying along the same branch of $\ttt_\xi$.

In addition to these trees, which we will use to measure local $\ell_p$ structure, we will be interested in computing the complexity of sequences which exhibit $\ell_p$ behavior.  For this, we will use the Schreier families.  We let $[\nn]^{<\nn}$ denote the finite subsets of $\nn$, which we identify with strictly increasing sequences in $\nn$ in the natural way.   With this identification, the order $\preceq$ described above can be applied to $[\nn]^{<\nn}$. That is, $E\preceq F$ if $E$ is an initial segment of $F$ when the two sets are listed as sequences in increasing order. We similarly identify $[\nn]$, the infinite subsets of $\nn$, with the infinite, strictly increasing sequences in $\nn$. In the sequel, we will assume all sequences in $\nn$ are written in strictly increasing order. Furthermore, for any $M\in [\nn]$, we let $[M]^{<\nn}$ (resp. $[M]$) denote the finite (resp. infinite) subsets of $M$.  

For $E, F\in [\nn]^{<\nn}$, we write $E<F$ to mean $\max E<\min F$.  We write $n<E$ (resp. $n\leqslant E$) to mean $n<\min E$ (resp. $n\leqslant \min E$).  For $E\in [\nn]^{<\nn}$ and $(m_n)=M\in \infin$, $M(E)=(m_n:n\in E)$.  For $\fff\subset [\nn]^{<\nn}$, we let $\fff(M)=\{M(E):E\in \fff\}$.

If $\fff, \gggg$ are regular, we define $\fff[\gggg]=\Bigl\{\bigcup_{i=1}^n E_i: E_1<\ldots <E_n, (\min E_i)_{i=1}^n\in \fff, E_i\in \gggg\Bigr\},$ noting that $\fff[\gggg]$ is also regular.  We let $\sss=\{E\in [\nn]^{<\nn}: |E|\leqslant E\}$.  For $k\in \nn$, we let $$\aaa_k=\{E\in [\nn]^{<\nn}: |E|\leqslant k\}.$$  Recall the Schreier families from \cite{AA}.  We let $$\sss_0=\{\varnothing\}\cup \bigl\{(n):n\in \nn\bigr\},$$ $$\sss_{\xi+1}= \sss[\sss_\xi],$$ and if $\xi<\omega_1$ is a limit ordinal, we fix a sequence of successors $\xi_n\uparrow \xi$ and let $$\sss_\xi = \{E\in [\nn]^{<\nn}: \exists n\leqslant E\in \sss_{\xi_n}\}.$$ It is known that in this case, $\xi_n\uparrow \xi$ can be chosen so that $\sss_{\xi_n}\subset \sss_{\xi_{n+1}}$ for all $n\in \nn$.  For convenience, we let $\sss_{\omega_1}=[\nn]^{<\nn}$.  We note that each family $\sss_\xi$ is \emph{spreading}, meaning that if $(m_i)_{i=1}^k\in \sss_\xi$ and if $n_i\geqslant m_i$ for each $1\leqslant i\leqslant k$, $(n_i)_{i=1}^k\in \sss_\xi$.  We also note that since $\sss_\xi$ is spreading, the derived tree $\sss_\xi^\zeta$ as defined above coincide with the $\zeta^{th}$ Cantor-Bendixson derivative, where $[\nn]^{<\nn}$ is topologized by identifying $E\leftrightarrow 1_E\in 2^\nn$ and endowing $2^\nn$ with the product topology.  It is well known that the Cantor-Bendixson index of $\sss_\xi$ is $\omega^\xi+1$.  For regular families, however, it is usually more convenient to consider the index $\iota(\fff)=\min \{\xi: \fff^\xi\subset \{\varnothing\}\}$.  For $\fff\neq \varnothing$, $\iota(\fff)+1$ is the Cantor-Bendixson index.  This index is somewhat more natural than the Cantor-Bendixson for our purposes, since $\iota(\fff[\gggg])=\iota(\gggg)\iota(\fff)$ for regular families $\fff, \gggg$.  

We recall the following result.  

\begin{proposition}\cite{C1} For regular families $\fff, \gggg$, there exists $M\in \infin$ so that $\fff(M)\subset \gggg$ if and only if the Cantor-Bendixson index of $\fff$ does not exceed the Cantor-Bendixson index of $\gggg$.    Moreover, if such an $M$ exists, then for any $N\in \infin$, there exists $L\in [N]$ so that $\fff(L)\subset \gggg$.  

\label{Gasparis analogue}

\end{proposition}

In particular, for any $k\in \nn$ and $\xi, \zeta<\omega_1$ with $0<\zeta$, the Cantor-Bendixson index of $\sss_\zeta[\aaa_k[\sss_\xi]]$ is $\omega^\xi k \omega^\zeta+1 = \omega^{\xi+\zeta}+1$, which is the Cantor-Bendixson index of $\sss_{\xi+\zeta}$.  Thus there exists $M\in \infin$ so that $\sss_\zeta[\aaa_k[\sss_\xi]](M)\subset\sss_{\xi+\zeta}$.

\subsection{Coloring lemma}

Throughout this work we will make use of a dichotomy which was introduced in \cite{C2}.  For readability, we do not include in this work all of the formalities involved in the statement and use of this dichotomy.  We will discuss here an interpretation of that dichotomy which is applicable to this work.  The most basic example will involve an operator $A:X\to Y$ between Banach spaces.  Suppose we have a collection $(x_t)_{t\in \ttt_{\xi\zeta}}\subset B_X$.  Suppose also that we have a decreasing collection of real-valued functions $(f_t)_{t\in \mt_{\xi\zeta}}$ defined on the chains in $\ttt_{\xi\zeta}$.  Here, decreasing means that for each non-empty chain $S$ of $\ttt_{\xi\zeta}$ and each $s,t\in \mt_{\xi\zeta}$ with $s\prec t$, $f_s(S)\leqslant f_t(S)$.  

\begin{lemma}\cite{C2} With the definitions above, either there exist a monotone function $g:\ttt_\xi\to \ttt_{\xi\zeta}$, $\delta>0$, and $t_0\in \mt_{\xi\zeta}$ so that for each $t\in \ttt_\zeta$ and for each chain $S$ in $\ttt_\zeta$, $f_{t_0}(\{g(t): t\in S\})\geqslant \delta$, or for any $\delta_n\downarrow 0$, there exists a block map $h$ taking $\ttt_\zeta$ into the chains of $\ttt_{\xi\zeta}$ so that with $h(\varnothing)=\{\varnothing\}$, for each $s,t\in \mt_\zeta$ with $s\prec t$, and for each $s'\in h(s)$, $f_{s'}(h(t))<\delta_{|t|}$.  

\label{dichotomy}

\end{lemma}

Often we will apply a simpler version of this lemma in which $f_\varnothing=f_t$ for all $t\in \mt_{\xi\zeta}$.  The idea is a refinement of ideas appearing in \cite{JO}.  We view the tree $\ttt_{\xi\zeta}$ as a tree of order $\zeta$ consisting of trees of order $\xi$.  Either one of the functions $f_t$ can be bounded away from zero on all chains of one of the ``interior'' trees of order $\xi$, which is the first alternative, or we can choose in a ``compatible'' manner one  chain from each of the interior trees so that what remains is ordered so as to resemble $\ttt_\zeta$ and, moreover, the chains can have a small value under a prescribed function, where both the value and the function depend upon the choices of chains which lie above the current segment in the tree resembling $\ttt_\zeta$.  

We remark here that if $h$ is a block map from $T_\zeta$ to the chains of $\ttt_\xi$, then for each $t\in \ttt_\zeta$ and $t'\in h(t)$, $|t|\leqslant |t'|$.  

\section{The Bourgain index of an operator}

Fix a normalized Schauder basis $(e_i)$.  For Banach spaces $X, Y$ and $A:X\to Y$ and $K\geqslant 1$, let $$T_{(e_i)}(A, X, Y, K)=\Bigl\{(x_i)_{i=1}^n\in B_X^{<\nn}:  (x_i)_{i=1}^n\lesssim_1 (e_i)_{i=1}^n, (e_i)_{i=1}^n\lesssim_K (Ax_i)_{i=1}^n\Bigr\}.$$  We define the $K$-$(e_i)$ \emph{non-preservation indices} of $A$ by  $$\textbf{NP}_{(e_i)}(A, X, Y, K)=o(T_{(e_i)}(A,X,Y,K)),$$ and the $(e_i)$ \emph{non-preservation index} of $A$ by $$\textbf{NP}_{(e_i)}(A,X,Y)=\sup_{K\geqslant 1} \textbf{NP}_{(e_i)}(A,X,Y,K).$$  

Note that there exists a subspace $Z$ of $X$ isomorphic to $[e_i]$ so that $A|_Z$ is an isomorphic embedding if and only if there exists $K\geqslant 1$ so that $T_{(e_i)}(A,X,Y,K)$ is ill-founded, so that $A$ fails to preserve a copy of $(e_i)$ if and only if $\textbf{NP}_{(e_i)}(A,X,Y)<\infty$.  We let $\mathfrak{NP}_{(e_i)}(X,Y)$ denote the operators from $X$ to $Y$ not preserving a copy of $(e_i)$.  We let $\mathfrak{NP}_{(e_i)}$ be the class consisting of all components $\mathfrak{NP}_{(e_i)}(X,Y)$, $X,Y\in \Ban$.  We write $T_p$ in place of $T_{(e_i)}$, $\textbf{NP}_p$ in place of $\textbf{NP}_{(e_i)}$, etc., in the case that $(e_i)$ is the canonical $\ell_p$ (resp. $c_0$ if $p=\infty$) basis.  Observe that $T_p(A,X,Y,K)$ and all of its derived trees are $p$-absolutely convex.   

For $X\in \Ban$ and $K\geqslant 1$, we write $T_p(X,K)$ in place of $T_p(I_X,X,X,K)$, $\textbf{I}_p(X,K)$ in place of $\textbf{NP}_p(I_X,X,X,K)$ and $\textbf{I}_p(X)$ in place of $\textbf{NP}_p(I_X,X,X)$.  We note that $\textbf{I}_p$ is the Bourgain $\ell_p$ (resp. $c_0$) index of $X$.  We recall that $\textbf{I}_p(X)>\omega$ if and only if $\ell_p$ (resp. $c_0$) is finitely representable in $X$.  

We make the following easy observations about these indices.    

\begin{proposition} Let $X, Y\in \Ban$. Fix a normalized basis $(e_i)$.  \begin{enumerate}[(i)]\item If $A:X\to Y$ is finite rank, $\emph{\textbf{NP}}_{(e_i)}(A,X,Y)=1+\emph{\rank}(A).$  \item For any $\xi\in \ord$, $\{A\in \mathfrak{L}(X,Y): \emph{\textbf{NP}}_{(e_i)}(A,X,Y)\leqslant \xi\}$ is closed with respect to the norm topology on $\mathfrak{L}(X,Y)$.  \item For any $W,Z\in \Ban$, $A\in \mathfrak{L}(Y,Z)$, $C\in \mathfrak{L}(W,X)$, $\emph{\textbf{NP}}_{(e_i)}(ABC, W,Z)\leqslant \emph{\textbf{NP}}_{(e_i)}(B, X, Y)$. \item If $X$ is separable and $A\in \mathfrak{L}(X,Y)$, $A\in \mathfrak{NP}_{(e_i)}(X,Y)$ if and only if $\emph{\textbf{NP}}_{(e_i)}(A,X,Y)<\omega_1$.  \end{enumerate}

\label{tedious}
\end{proposition}

\begin{proof}$(i)$ Let $r=\rank(A)$.  Fix $(x_i)_{i=1}^r$ so that $(Ax_i)_{i=1}^r$ is a basis of $A(X)$.  Then there exist $a,b>0$ so that $(x_i)_{i=1}^r\lesssim_a (e_i)_{i=1}^r$ and $(e_i)_{i=1}^r\lesssim_b (Ax_i)_{i=1}^r$.  Thus $(a^{-1}x_i)_{i=1}^r\in T_{(e_i)}(A,X,Y,ab)$, and $\textbf{NP}_{(e_i)}(A,X,Y,ab)>r$, since $\varnothing\in T_{(e_i)}(A,X,Y, ab)^r$.  But for any $(u_i)_{i=1}^{r+1}\subset X$, there exist scalars $(a_i)_{i=1}^{r+1}$ not all zero so that $\sum_{i=1}^{r+1} a_iAu_i=0$.  Therefore $(Au_i)_{i=1}^{r+1}$ does not $K$-dominate $(e_i)_{i=1}^{r+1}$ for any $K$.  Therefore $(u_i)_{i=1}^{r+1}\notin T_{(e_i)}(A,X,Y,K)$ for any $K$. 

$(ii)$ Assume $A\in \mathfrak{L}(X,Y)$ is such that $\textbf{NP}_{(e_i)}(A,X,Y)>\xi$.  There exists $K\geqslant 1$ so that $\textbf{NP}_{(e_i)}(A,X,Y,K)>\xi$.   By Proposition \ref{indexing}, there exists $(x_t)_{t\in \ttt_\xi}\subset B_X$ so that for each $t\in \ttt_\xi$, $(x_{t|_i})_{i=1}^{|t|}\in T_{(e_i)}(A,X,Y,K)$.  We claim that $(x_t)_{t\in \ttt_\xi}$ can be used to show that $\textbf{NP}_{(e_i)}(B,X,Y,2K)>\xi$ for any $B:X\to Y$ with $\|A-B\|<1/2K$, which will give that the complement of the indicated set is open.  By Proposition \ref{indexing}, it suffices to show that $(e_i)_{i=1}^{|t|}\lesssim_{2K}(Bx_{t|_i})_{i=1}^{|t|}$ for each $t\in \ttt_\xi$.  Fix scalars $(a_i)_{i=1}^{|t|}$ with $\|\sum_{i=1}^{|t|} a_ie_i\|=1$.  Then $\|\sum_{i=1}^{|t|} a_i x_{t|_i}\|\leqslant 1$, and $$\Bigl\|\sum_{i=1}^{|t|} a_i Bx_{t|_i}\Bigr\|\geqslant \Bigl\|\sum_{i=1}^{|t|} a_iAx_{t|_i}\Bigr\|- \|A-B\|\Bigl\|\sum_{i=1}^{|t|} a_ix_{t|_i}\Bigr\| \geqslant 1/K-1/2K = 1/2K.$$  

$(iii)$ Assume $\textbf{NP}_{(e_i)}(ABC,X,Y)>\xi$.   We will show $\textbf{NP}_{(e_i)}(B,X,Y)>\xi$. Note that if $A=0$ or $C=0$, $\textbf{NP}_{(e_i)}(ABC, X, Y)=1$, and $\xi=0$.  Then $\textbf{NP}_{(e_i)}(B,X,Y)>\xi=0$, since every tree contains the empty sequence. Therefore we must consider the case that neither $A$ nor $C$ is the zero operator.   Fix $K\geqslant 1$ so that $\textbf{NP}_{(e_i)}(ABC,X,Y,K)>\xi$.  Choose $(w_t)_{t\in \ttt_\xi}$ so that $(w_{t|_i})_{i=1}^{|t|}\in T_{(e_i)}(ABC,X,Y,K)$ for each $t\in \ttt_\xi$.  Choose a number $0<c<\|C\|^{-1}$ and let $x_t=cCw_t$.  Then since $\|cC\|<1$, for any  $t\in\ttt_\xi$, $$(x_{t|_i})_{i=1}^{|t|}\lesssim_1 (w_{t|_i})_{i=1}^{|t|} \lesssim_1 (e_i)_{i=1}^{|t|}.$$  Moreover, $$(e_i)_{i=1}^{|t|}\lesssim_K (ABCw_{t|_i})_{i=1}^{|t|} \lesssim_{\|A\|} (BCw_{t|_i})_{i=1}^{|t|} \lesssim_{c^{-1}} (cBCw_{t|_i})_{i=1}^{|t|} = (Bx_{t|_i})_{i=1}^{|t|}.$$  Thus $(x_t)_{t\in \ttt_\xi}$ witnesses the fact that $\textbf{NP}_{(e_i)}(B, X,Y, \|A\|Kc^{-1})>\xi$.

$(iv)$ This follows from the fact that $\textbf{NP}_{(e_i)}(A,X,Y)<\omega_1$ if and only if $\textbf{NP}_{(e_i)}(A,X,Y,K)<\omega_1$ for all $K\geqslant 1$.  Since $T_{(e_i)}(A,X,Y,K)$ is clearly seen to be a closed tree on the Polish space $X$, Bourgain's version of the Kunen-Martin theorem \cite{Bo} guarantees that $T_{(e_i)}(A,X,Y,K)$ is well-founded if and only if its order is countable.  

\end{proof}

Of particular interest to us will be the cases when $(e_i)$ is the $\ell_p$ or $c_0$ basis.  The following facts are known for computing the Bourgain $\ell_p$ index of a Banach space.  The modifications for operators are inessential, so we only sketch the proof.   

\begin{proposition} Fix $1\leqslant p\leqslant \infty$, $X,Y\in \Ban$, and $A\in \mathfrak{L}(X,Y)$ not finite rank.  \begin{enumerate}[(i)]\item For $K\geqslant 1$, if $W\leqslant X$, $Z\leqslant Y$ have finite codimension in $X$, $Y$, respectively, and if $\eta$ is a limit ordinal, $\emph{\textbf{NP}}_p(A, X, Y,K)>\eta $ if and only if $ o(\{(x_i)_{i=1}^n\in T_p(A,X,Y,K): x_i\in W, Ax_i\in Z\})>\eta.$ \item Either $A$ preserves a copy of $\ell_p$ (or $c_0$ if $p=\infty$) or there exists $0<\xi\in \ord$ so that $\emph{\textbf{NP}}_p(A,X,Y)=\omega^\xi$. \item If $0<\xi\in \ord$ and $\emph{\textbf{NP}}_p(A, X, Y)\leqslant \omega^\xi$, then $\emph{\textbf{NP}}_p(A,X,Y,K)<\omega^\xi$ for every $K\geqslant 1$.   

\end{enumerate}

\label{special form}

\end{proposition}

\begin{proof}[Sketch] $(i)$ One direction is obvious.  Let $E\subset X^*$ and $F\subset Y^*$ be finite sets so that $W=\cap_{x^*\in E} \ker(x^*)$ and $Z=\cap_{y^*\in F} \ker(y^*)$.  Choose $k\in \nn$ so that $k>|E|+|F|$. Assume that $\textbf{NP}_p(A,X,Y,K)>\eta$.  Note that $\eta=k\eta$.  Choose $(x_t)_{t\in \ttt_{k\eta}}$ so that $(x_{t|_i})_{i=1}^{|t|}\in T_p(A,X,Y,K)$ for every $t\in \ttt_{k\eta}$.  For a chain $S$ of $\ttt_{k\omega^\xi}$, let $f(S)=1$ provided $$0<\min \Bigl\{\sum_{x^*\in E} |x^*(x)|+\sum_{y^*\in F} |y^*(Ax)|:x\text{\ is a\ }p\text{-absolutely convex combination of\ }(x_t)_{t\in S}\Bigr\},$$ and $f(S)=0$ otherwise. Then by Lemma \ref{dichotomy}, either there exists a monotone $g:\ttt_k\to \ttt_{k\eta}$ so that for each segment $S$ of $\ttt_k$, $f(\{g(s): s\in S\})=1$, or there exists $(w_t)_{t\in \ttt_\eta}$ each branch of which consists of a $p$-absolutely convex block of a branch of $(x_t)_{t\in \ttt_{k\eta}}$ and so that for each $t\in \ttt_\eta$, $\sum_{x^*\in E}|x^*(w_t)|+\sum_{y^*\in F}|y^*(w_t)|=0$.  A dimension argument implies that the first alternative fails.  But the properties of $(w_t)_{t\in \ttt_\eta}$ and choices of $E,F$ witness the fact that $o(\{(x_i)_{i=1}^n\in T_p(A,X,Y,K): x_i\in W, Ax_i\in Z\})>\eta$.

$(ii)$ Note that, since we have assumed $A$ is not finite rank, $\textbf{NP}_p(A,X,Y)\geqslant \omega$.  This means that if $\textbf{NP}_p(A,X,Y)\in \ord$, it is an infinite ordinal.  Therefore in order to prove the existence of the desired $\xi$, it is sufficient to prove that for any limit ordinal $\eta< \textbf{NP}_p(A,X,Y)$, $\eta \cdot 2 <\textbf{NP}_p(A,X,Y)$ \cite{M}.  If there are no limit ordinals less than $\textbf{NP}_p(A,X,Y)$, then $\textbf{NP}_p(A,X,Y)\leqslant \omega$, and this inequality must be equality.  So assume $\omega<\textbf{NP}_p(A,X,Y)$ and fix a limit ordinal $\eta<\textbf{NP}_p(A,X,Y)$.  Choose $K\geqslant 1$ so that $\eta<\textbf{NP}_p(A,X,Y,K)$.  Fix $(x_i)_{i=1}^m\in T_p(A,X,Y,K)$ and choose $F\subset Y^*$ finite and $2$-norming for $[Ax_i]_{i=1}^m$.  If $(u_i)_{i=1}^n\in T_p(A,X,Y,K)$ is such that $Au_i\in \cap_{y^*\in F}\ker(y^*)$ for each $1\leqslant i\leqslant n$, then for any scalars $(a_i)_{i=1}^m$ and $(b_j)_{j=1}^n$, $$\Bigl\|\sum_{i=1}^m a_i x_i+\sum_{j=1}^n b_ju_j\Bigr\| \leqslant 2\Bigl[\Bigl\|\sum_{i=1}^m a_ix_i\Bigr\|\vee \Bigl\|\sum_{j=1}^n b_ju_j\Bigr\|\Bigr]\leqslant 2\Bigl(\sum_{i=1}^m |a_i|^p + \sum_{j=1}^n |b_j|^p\Bigr)^{1/p}$$ and $$\frac{1}{2}\Bigl(\sum_{i=1}^m |a_i|^p +\sum_{j=1}^n |b_j|^p\Bigr)^{1/p} \leqslant \Bigl\|\sum_{i=1}^m a_iA x_i\Bigr\| \vee \Bigl\|\sum_{j=1}^n b_jA u_j\Bigr\|\leqslant 3\Bigl\|\sum_{i=1}^n a_iAx_i+ \sum_{j=1}^n b_j A u_j\Bigr\|.$$   From this we deduce that $$\frac{1}{2}(x_1, \ldots, x_m, u_1, \ldots, u_n)\in T_p(A,X,Y,12K).$$  By $(i)$, since $\textbf{NP}_p(A,X,Y,K)>\eta$, we can choose $(u_t)_{t\in \ttt_\eta}$ so that for each $t\in \ttt_\eta$, $(u_{t|_i})_{i=1}^{|t|}\in T_p(A,X,Y,K)$ and $Au_t\in \cap_{y^*\in F}\ker(y^*)$.  Then with $T=T_p(A,X,Y, 12K)$, for each $t\in \ttt_\eta$, $$(\frac{1}{2}u_{t|_i})_{i=1}^{|t|}\in T(\frac{1}{2}x_1, \ldots, \frac{1}{2}x_m).$$ This means $$o\Bigl(T(\frac{1}{2}x_1, \ldots, \frac{1}{2}x_m)\Bigr)>\eta,$$ which is equivalent to $$(\frac{1}{2}x_1, \ldots, \frac{1}{2}x_m)\in T^\eta = T_p(A,X,Y, 12K)^\eta.$$  Since $(x_i)_{i=1}^m\in T_p(A,X,Y,K)$ was arbitrary, $$\Bigl\{\frac{1}{2}(x_i)_{i=1}^m: (x_i)_{i=1}^m\in T_p(A,X,Y,K)\Bigr\}\subset T_p(A,X,Y,12K)^\eta.$$  This means $\textbf{NP}_p(A,X,Y,12K)>\eta \cdot 2$.

$(iii)$ This follows from the fact that $\textbf{NP}_p(A,X,Y,K)$ is always a successor, since we include the empty sequence in $T_p(A,X,Y,K)$, while $\omega^\xi$ is a limit ordinal.

\end{proof}

We make the following definition: For $X,Y\in \Ban$, $\xi\in \ord$, $1\leqslant p\leqslant \infty$, we let $$\mathfrak{NP}_p^\xi(X,Y)= \{A\in \mathfrak{L}(X,Y): \textbf{NP}_p(A,X,Y)\leqslant \omega^\xi\}.$$  We let $\mathfrak{NP}_p^\xi$ be the class of all operators $A:X\to Y$ so that  $A\in \mathfrak{NP}_p^\xi(X,Y)$ for some $X,Y\in \Ban$.  

We have already noted that $$\mathfrak{NP}_p=\cup_{\xi\in \ord}\mathfrak{NP}_p^\xi,$$ and if we only consider operators on a separable domain, we only need to include all countable ordinals in this union.  It is not difficult to construct examples to show that neither of these unions can be replaced with a smaller union.  That is, for any $\xi \in \ord$ and $1\leqslant p\leqslant \infty$, there exist $X,Y\in \Ban$ and $A:X\to Y$ with $A\in \mathfrak{NP}_p\setminus \mathfrak{NP}_p^\xi$.  Moreover, if $\xi<\omega_1$, $X$ can be taken to be separable.  In fact, we will show later that in all cases one can take $X=Y$ and $A=I_X$.

We wish to determine when the classes $\mathfrak{NP}_p^\xi$ give ideals, or can be used to determine ideals.  For this we will use the following estimates.

\begin{lemma} Fix $1\leqslant p\leqslant \infty$, $X,Y\in \Ban$, $A, B\in \mathfrak{L}(X,Y)$, $K\geqslant 1$.  \begin{enumerate}[(i)]\item For any $\ee>0$, $\emph{\textbf{NP}}_p(A+B, X, Y, K)\leqslant \emph{\textbf{NP}}_p(A, X, Y)\emph{\textbf{NP}}_p(B, X, Y, K+\ee).$  \item $\emph{\textbf{NP}}_1(A+B, X, Y, K)\leqslant \emph{\textbf{NP}}_1(A,X,Y, 2K)\emph{\textbf{NP}}_1(B,X,Y,2K).$   \end{enumerate}

\label{product}

\end{lemma}

\begin{proof}$(i)$ We treat the $p<\infty$ case, with the $p=\infty$ case requiring only notational changes.  Fix $\ee>0$.  If either $\textbf{NP}_p(A, X, Y)=\infty$ or $\textbf{NP}_p(B, X, Y, K+\ee)=\infty$, there is nothing to prove.  So assume $\xi=\textbf{NP}_p(A,X,Y)\in \ord$ and $\zeta=\textbf{NP}_p(B,X,Y, K+\ee)\in \ord$.  To obtain a contradiction, assume $\textbf{NP}_p(A+B,X,Y, K)>\xi\zeta$.  Fix $(x_t)_{t\in \ttt_{\xi\zeta}}\subset B_X$ so that for each $t\in \ttt_{\xi\zeta}$, $(x_{t|_i})_{i=1}^{|t|}\in T_p(A+B, X, Y, K)$.  Fix $\delta_n\downarrow 0$ so that $K^{-1}- (K+\ee)^{-1}>\sum \delta_n$.  Define the function $f$ on the chains of $\ttt_{\xi\zeta}$ by $$f(S)=\min \{\|Ax\|: x\text{\ is a\ }p\text{-absolutely convex combination of\ }(x_t)_{t\in S}\}.$$  By Lemma \ref{dichotomy}, either there exists a monotone $g:\ttt_\xi\to \ttt_{\xi\zeta}$ and $\delta>0$ so that for each segment $S$ in $\ttt_\xi$, $f(\{g(t):t\in S\})\geqslant \delta$, or there exists a block map $h$ taking $\ttt_\zeta$ into the chains of $\ttt_{\xi\zeta}$ so that $f(h(t))<\delta_{|t|}$ for each $t\in \ttt_\zeta$.   In the first case, $(x_{g(t)})_{t\in \ttt_\xi}$ gives that $\textbf{NP}_p(A,X,Y)\geqslant \textbf{NP}_p(A,X,Y,\delta^{-1})>\xi$, a contradiction.  In the second case, for each $t\in \ttt_\zeta$, choose $u_t$ to be a $p$-absoblutely convex combination of $(x_s)_{s\in h(t)}$ so that $\|Au_t\| = f(h(t))<\delta_{|t|}$. We claim $(u_t)_{t\in \ttt_\zeta}$ implies that $\textbf{NP}_p(B,X,Y, K+\ee)>\zeta$.  To see this, we need to show that $(Bu_{t|_i})_{i=1}^{|t|}$ $(K+\ee)$-dominates the $\ell_p$ basis for each $t\in \ttt_\zeta$.  Fix scalars $(a_i)_{i=1}^{|t|}$ with $p$-norm equal to $1$.  Then \begin{align*} \Bigl\|\sum_{i=1}^{|t|} a_i Bu_{t|_i}\Bigr\| & \geqslant \Bigl\|\sum_{i=1}^{|t|} a_i(A+B)u_{t|_i}\Bigr\| - \Bigl\|\sum_{i=1}^{|t|} a_iAu_{t|_i}\Bigr\| \\ & \geqslant K^{-1} - \sum_{i=1}^{|t|} |a_i|\delta_i > (K+\ee)^{-1}. \end{align*} Of course, in both cases we have used that $T_p(A+B, X, Y, K)$ is $p$-absolutely convex and that $(u_{t|_i})_{i=1}^{|t|}$ was a $p$-absolutely convex block of a branch of $T_p(A+B, X, Y, K)$.

$(ii)$ The proof is similar to $(i)$.  Assume $\textbf{NP}_1(A, X, Y, 2K)=\xi\in \ord$, $\textbf{NP}_1(B, X, Y, 2K)=\zeta\in \ord$.  Again, assume $(x_t)_{t\in \ttt_{\xi\zeta}}\subset B_X$ is such that $(x_{t|_i})_{i=1}^{|t|}\in T_1(A+B, X, Y, K)$ for each $t\in \ttt_{\xi\zeta}$.  We define the function $f$ on the chains of $\ttt_{\xi\zeta}$ by letting $f(S)=1$ if $$1/2K\leqslant \min \{\|Ax\|: x\text{\ is a\ }1\text{-absolutely convex combination of\ }(x_t)_{t\in S}\}$$ and $f(S)=0$ otherwise.  By Lemma \ref{dichotomy}, either there exists $(z_t)_{t\in \ttt_\xi}$ each branch of which consists of a subsequence of a branch of $(x_t)_{t\in \ttt_{\xi\zeta}}$ so that $(Az_{t|_i})_{i=1}^{|t|}$ $2K$-dominates the $\ell_1^{|t|}$ basis for each $t\in \ttt_\xi$, in which case we reach the contradiction $\textbf{NP}_1(A, X, Y, 2K)>\xi$, or there exists a tree $(u_t)_{t\in \ttt_\zeta}$ consisting of $1$-absolutely convex blocks of branches of $(x_t)_{t\in \ttt_{\xi\zeta}}$, so that $\|Au_t\|<1/2K$ for each $t\in \ttt_\zeta$.  In the second case, $(u_t)_{t\in \ttt_\zeta}$ gives that $\textbf{NP}_1(B,X,Y,2K)>\zeta$, another contradiction.  To see the last statement, fix $t\in \ttt_\zeta$ and $(a_i)_{i\in \ttt_\zeta}$ with $1$-norm equal to $1$.  Then \begin{align*} \Bigl\|\sum_{i=1}^{|t|} a_iBu_{t|_i}\Bigr\| & \geqslant\Bigl\|\sum_{i=1}^{|t|} a_i(A+B)u_{t|_i}\Bigr\| - \Bigl\|\sum_{i=1}^{|t|} a_iAu_{t|_i}\Bigr\| \\ & \geqslant   \Bigl\|\sum_{i=1}^{|t|} a_i (A+B)u_{t|_i}\Bigr\| - \sum_{i=1}^{|t|}|a_i|/2K \\ & \geqslant 1/K - 1/2K = 1/2K.\end{align*}

\end{proof}

We note that since $c_0$ lower estimates are often easy to satisfy, we have the following improvement of the above estimates in the $p=\infty$ case.  

\begin{lemma}
Let $X, Y$ be Banach spaces and $A,B:X\to Y$ operators between $X$ and $Y$.  Then $\emph{\textbf{NP}}_\infty(A+B, X, Y)\leqslant \omega(\emph{\textbf{NP}}_\infty(A,X,Y) \vee \emph{\textbf{NP}}_\infty(B,X,Y))$. \label{c_0}
\end{lemma}

\begin{proof}

If $A$ and $B$ are both finite rank, or if either preserves a copy of $c_0$, the result is trivial.  Therefore we may assume $\textbf{NP}_\infty(A,X,Y)\vee \textbf{NP}_\infty(B,X,Y)=\omega^\xi$ for some ordinal $0<\xi$.  Assume $\textbf{NP}_\infty(A+B, X, Y)>\omega \omega^\xi$.  This means there exists $K\geqslant 1$ so that $\textbf{NP}_\infty(A+B, X, Y, K)>\omega \omega^\xi$.  Fix $(u_t)_{t\in \ttt_{\omega\omega^\xi}}$ so that $(u_{t|_i})_{i=1}^{|t|}\in T_\infty(A+B,X,Y,K)$. Fix positive numbers $\ee_n\downarrow 0$ so that $\prod_{i=1}^\infty (1-\ee_i)^{-1}\leqslant 2$.  For each $t\in\ttt_{\omega\omega^\xi}$, fix finite sets $E_t, F_t\subset B_{Y^*}$ so that $E_t$ is $(1-\ee_{|t|})$-norming for $[Au_s: s\preceq t]$ and $F_t$ is $(1-\ee_{|t|})$-norming for $[Bu_s: s\preceq t]$.  We may of course assume that $E_{t|_1}\subset E_{t|_2}\subset \ldots \subset E_t$ and $F_{t|_1}\subset \ldots \subset F_t$ for each $t\in \ttt_{\omega \omega^\xi}$.  Let $E_\varnothing=F_\varnothing = \{0\}$.  

For each $t\in \mt_{\omega\omega^\xi}$, define $f_t$ on the chains of $\ttt_{\omega\omega^\xi}$ by letting $f_t(S)=0$ if there exists an $\infty$-absolutely convex combination $x$ of $(u_s: s\in S)$ so that $y^*(Ax)=0$ for all $y^*\in E_t$ and $y^*(Bx)=0$ for all $y^*\in F_t$, and $f_t(S)=1$ otherwise.  By a dimension argument, for any monotone $\theta:\ttt_\omega\to \ttt_{\omega\omega^\xi}$ and $t\in \ttt_{\omega\omega^\xi}$, there exists a chain $S$ of $\ttt_\omega$ so that $f_t(\{\theta(s):s\in S\})=0$.  By Lemma \ref{dichotomy}, there exists a block map $h$ from $\ttt_{\omega^\xi}$ to the chains of $\ttt_{\omega\omega^\xi}$ so that for all $s,t\in \ttt_{\omega^\xi}$ with $s\prec t$, and for all $s'\in h(s)$, $f_{s'}(h(t))=0$.  This means that for each $t\in \ttt_{\omega^\xi}$, there exists an $\infty$-absolutely convex combination $x_t$ of $(u_s:s\in h(t))$ so that for any $s\in \ttt_{\omega^\xi}$ with $s\prec t$, $y^*(Ax_t)=0$ for all $y^*\in E_{\max h(s)}$ and $y^*(Bx_t)=0$ for all $y^*\in F_{\max h(s)}$.

Define $c:\ttt_{\omega^\xi}\to \{0,1\}$ by letting $c(t)=0$ if $\|Ax_t\|\geqslant 1/2K$, and $c(t)=1$ otherwise.  Note that if $c(t)=1$, $$\|Bx_t\| \geqslant \|(A+B)x_t\|- \|Ax_t\| \geqslant 1/K- 1/2K=1/2K.$$  By \cite{C2}, there exists a monotone map $\theta:\ttt_{\omega^\xi}\to \ttt_{\omega^\xi}$ so that $c\circ \theta$ is constant.  Without loss of generality, we assume $c\circ \theta\equiv 0$, so that $\|Ax_{\theta(t)}\|\geqslant 1/2K$ for all $t\in \ttt_{\omega^\xi}$.  

Fix $t\in \ttt_{\omega^\xi}$ with $|t|>1$ and scalars $(a_i)_{i=1}^{|t|}$.  Let $t'$ be the immediate predecessor of $t$ in $\ttt_{\omega^\xi}$.  Since $\sum_{i=1}^{|t'|} a_i Ax_{\theta(t|_i)}\in [Au_s: s\preceq \max h(\theta(t'))]$, there exists $y^*\in E_{\max h(\theta(t'))}$ so that $$y^*\Bigl(\sum_{i=1}^{|t'|} a_i A x_{\theta(t|_i)}\Bigr) \geqslant (1-\ee_{|\max h(\theta(t'))|}) \|\sum_{i=1}^{|t'|} a_i Ax_{\theta(t|_i)}\| \geqslant (1-\ee_{|t'|}) \|\sum_{i=1}^{|t'|} a_i Ax_{\theta(t|_i)}\|.$$ Here we have used the fact that $|t'|\leqslant |\theta(t')|\leqslant |\max h(\theta(t'))|$ and $\ee_n\downarrow 0$.  Since $y^*(Ax_t)=0$, $$\|\sum_{i=1}^{|t|} a_i A x_{\theta(t|_i)} \| \geqslant y^*\Bigl(\sum_{i=1}^{|t'|} a_i A x_{\theta(t|_i)}\Bigr)\geqslant (1-\ee_{|t'|}) \|\sum_{i=1}^{|t'|} a_i Ax_{\theta(t|_i)}\|.$$   

Applying this inequality iteratively yields that for all $t\in \ttt_{\omega^\xi}$, the sequence $(Ax_{\theta(t|_i)})_{i=1}^{|t|}$ is $2$-basic.  Since $\|Ax_{\theta(t)}\|\geqslant 1/2K$, we deduce that $(Ax_{\theta{(t|_i)}})_{i=1}^{|t|}$ $8K$-dominates the $\ell_\infty^{|t|}$ basis for each $t\in \ttt_{\omega^\xi}$, and we deduce that $\textbf{NP}_\infty(A,X,Y)>\omega^\xi$, a contradiction.

\end{proof}

\begin{remark} Note that essentially the same proof above with $2K$ replaced by $nK$ allows us to deduce that for $A_i:X\to Y$, $1\leqslant i\leqslant n$, $\textbf{NP}_\infty(\sum_{i=1}^n A_i, X, Y)\leqslant \omega \vee_{i=1}^n \textbf{NP}_\infty(A_i, X,Y)$, which is better than the estimate $\textbf{NP}_\infty(\sum_{i=1}^n A_i, X, Y)\leqslant \omega^{n-1}\vee_{i=1}^n \textbf{NP}_\infty(A_i, X, Y)$ which can be deduced by iterating the previous lemma.

\end{remark}

We remark here that the proof of Lemma \ref{c_0} essentially contains a proof of the following result.  

\begin{proposition} Suppose $0<\xi\in \ord$, $A:X\to Y$ is an operator, and $(u_t)_{t\in \ttt_{\omega\xi}}\subset B_X$ is such that $(u_{t|_i})_{i=1}^{|t|}\in T_p(A,X,Y, K)$ (resp. $SS(A,X,Y,K))$ for all $t\in \ttt_{\omega \xi}$.  Then for any $\ee>0$, there exists a $p$-absolutely convex block tree (resp. normalized block tree) $(x_t)_{t\in \ttt_\xi}$ of $(u_t)_{t\in \ttt_{\omega\xi}}$ so that for all $t\in \ttt_\xi$, both $(x_{t|_i})_{i=1}^{|t|}$ and $(Ax_{t|_i})_{i=1}^{|t|}$ are $(1+\ee)$-basic.

\end{proposition}

We also note that if $\xi\geqslant \omega$, $\omega\omega^\xi=\omega^\xi$, so that if $A,B\in \mathfrak{NP}_\infty^\xi$, $A+B\in \textbf{NP}_\infty^\xi$.  Thus Lemma \ref{c_0} implies that $\mathfrak{NP}_\infty^\xi$ is an ideal whenever $\xi\geqslant \omega$.

Recall \cite{M} that for $\xi\in \ord$, $\alpha\beta<\xi$ for each $\alpha, \beta<\xi$ if and only if $\xi=0$, $\xi=1$, or $\xi=\omega^{\omega^\zeta}$ for some $\zeta\in \ord$.  Moreover, for $0<\alpha<\omega^{\omega^\zeta}$, $\alpha \omega^{\omega^\zeta}= \omega^{\omega^\zeta}$.  This means that if $A,B:X\to Y$ are such that $\textbf{NP}_p(A,X,Y), \textbf{NP}_p(B,X,Y)<\omega^{\omega^\zeta}$, then $\textbf{NP}_p(A+B,X,Y)<\omega^{\omega^\zeta}$.  Moreover, if $\textbf{NP}_p(A,X,Y), \textbf{NP}_p(B,X,Y)\leqslant \omega^{\omega^\zeta}$ and at least one of these inequalities is strict, then $\textbf{NP}_p(A+B, X, Y)\leqslant \omega^{\omega^\xi}$. This uses Proposition \ref{special form}$(iii)$.   However, the appearance of $\textbf{NP}_p(A,X,Y)$ in the product estimates above does not allow us to deduce that if $A,B\in \mathfrak{NP}_p^{\omega^\zeta}$, $A+B\in \mathfrak{NP}_p^{\omega^\zeta}$ except in the case that $\zeta=0$.  However, the improvement of the product estimate for $p=1$ does allow this conclusion, again using Proposition \ref{special form}$(iii)$. 

Note that the difference between the $p=1$ and $1<p$ cases is that small, uniform perturbations of sequences exhibiting $\ell_1$ behavior also exhibit $\ell_1$ behavior, which is false for each $1<p$ without a uniform bound on the length of the sequences.  The positive result for sequences of uniformly bounded length follows from a more general result.  In analogy to \cite{D}, we say a basis has property $(S')$ provided $\mathfrak{NP}_{(e_i)}$ is an ideal.  Since by standard techniques it is easy to see that if $A+B:X\to Y$ is an isomorphic embedding of an infinite dimensional subspace $Z$ of $X$ into $Y$, then either $A$ or $B$ is an isomorphic embedding of an infinite dimensional subspace of $Z$ into $Y$.  From this we deduce, for instance, that every Schauder basis of a minimal Banach space has property $(S')$, and therefore the any bases of $\ell_p$ and $c_0$ have property $(S')$. Recall that for any operator $A:X\to Y$ and any ultrafilter $\uuu$ over any set, there is an induced operator $A_\uuu:X_\uuu\to Y_\uuu$.  Following a general method for building new operator ideals from given operator ideals, if $(e_i)$ has property $(S')$, we say the operator $A:X\to Y$ is super-$\mathfrak{NP}_{(e_i)}$ provided $A_\uuu\in \mathfrak{NP}_{(e_i)}(X_\uuu, Y_\uuu)$ for any ultrafilter $\uuu$.  Since $(e_i)$ has property $(S')$, $\mathfrak{NP}_{(e_i)}$ is an ideal, easily seen to be closed, and we deduce that the class of super-$\mathfrak{NP}_{(e_i)}$ operators is also a closed ideal.  By standard ultrafilter techniques, we obtain the following.  

\begin{proposition} Let $(e_i)$ be a Schauder basis.  Fix $X,Y\in \Ban$.  Then $\emph{\textbf{NP}}_{(e_i)}(A,X,Y)\leqslant \omega$ if and only if for any ultrafilter $\uuu$, $A_\uuu\in \mathfrak{NP}_{(e_i)}(X_\uuu, Y_\uuu)$. If $(e_i)$ has property $(S')$, the class of operators $A:X\to Y$ with $\emph{\textbf{NP}}_{(e_i)}(A,X,Y)\leqslant \omega$ is a closed operator ideal.    

\label{ultrafilter}
\end{proposition}

\begin{proof}[Sketch] 
If $\textbf{NP}_{(e_i)}(A,X,Y)>\omega$, then there exists $K$ so that $\textbf{NP}_{(e_i)}(A,X,Y,K)>\omega$.  This means that for any $n\in \nn$, there exists $(x_i^n)_{i=1}^n\subset B_X$ which is $1$-dominated by $(e_i)_{i=1}^n$ and so that $(Ax_i^n)_{i=1}^n$ $K$-dominates $(e_i)_{i=1}^n$. Fix a free ultrafilter $\uuu$ on $\nn$. For each $i\in \nn$, if $\chi_i\in X_\uuu$ is the equivalence class of $$(0, \ldots, 0, x^n_i, x^{n+1}_i, x^{n+2}_i, \ldots),$$ with $i-1$ zeros, it is straightforward to check that $(\chi_i)\subset B_{X_\uuu}$ is $1$-dominated by $(e_i)$ and $(A_\uuu \chi_i)$ $K$-dominates $(e_i)$.  Therefore $A_\uuu\notin \mathfrak{NP}_{(e_i)}(X_\uuu, Y_\uuu)$.   

Suppose $\uuu$ is an ultrafilter, $(\chi_i)\subset B_{X_\uuu}$ is $1/2$-dominated by $(e_i)$, and $(A_\uuu \chi_i)$ $K/2$-dominates $(e_i)$.  For any $n\in \nn$, there exist isomorphisms $P:[\chi_i:1\leqslant i\leqslant n]:=E\to F\subset X$ and $Q:A_\uuu(E)\to A(F)$ so that $QA_\uuu=AP$ $\|P\|\leqslant 2$ and $\|Q^{-1}\|\leqslant 2$.  We then deduce that $(P\chi_i)_{i=1}^n$ is $1$-dominated by $(e_i)_{i=1}^n$ and $(AP\chi_i)_{i=1}^n$ $K$-dominates $(e_i)_{i=1}^n$.  Then $(P\chi_i)_{i=1}^n\in T_{(e_i)}(A,X,Y, K)$, and since $n\in\nn$ was arbitrary, $o(T_{(e_i)}(A,X,Y, K))\geqslant \omega$.  Since the order of a tree is always a successor, $o(T_{(e_i)}(A,X,Y, K))>\omega$.  

The second statement follows from the first statement and the discussion preceding the proposition.   

\end{proof}

\begin{theorem} Fix $0<\zeta\in \ord$. \begin{enumerate}[(i)]\item  $\cup_{\xi<\omega^\zeta} \mathfrak{NP}_p^\xi$ is an operator ideal.  \item $\mathfrak{NP}_1^{\omega^\zeta}$ is a closed operator ideal.  \item For each $1\leqslant p\leqslant \infty$, $\mathfrak{NP}^1_p$ is a closed operator ideal. \item If $\zeta$ is infinite, $\mathfrak{NP}_\infty^\xi$ is a closed ideal.  \end{enumerate}

\label{ideals}

\end{theorem}

We will see later that $\cup_{\xi<\omega^\zeta}\mathfrak{NP}_p^\xi$ is not closed unless $\zeta$ has uncountable cofinality.  

\begin{proof} We have already discussed why each statement is true.  Because it demonstrates a simple and highly elucidative case of our coloring lemma, we offer an alternative proof of the last statement of Theorem \ref{ideals}, which is a consequence of Proposition \ref{ultrafilter}.  Assume $X,Y\in \Ban$ and $A,B\in \mathfrak{NP}_p^1(X,Y)$. Note that by Proposition \ref{special form}$(iii)$, this simply means that for any $K\geqslant 1$, $\textbf{NP}_p(A,X,Y,K), \textbf{NP}_p(B,X,Y,K)<\omega$.  Let $m=\textbf{NP}_p(A,X,Y,2K)$ and $n=\textbf{NP}_p(B,X,Y, 2Km)$.  Assume $(x_i)_{i=1}^{mn}\subset B_X$ is $1$-dominated by the $\ell_p$ basis.  Then for each $1\leqslant j\leqslant m$, we can find $u_j$ a $p$-absolutely convex block of $(x_i)_{i=(j-1)n+1}^{jn}$ so that $\|Bu_j\|<1/2Km$.  If this statement were false for a given $j$, $(x_i)_{i=(j-1)n+1}^{jn}$ would imply that $\textbf{NP}_p(B, X, Y, 2Km)>n$, a contradiction.  Since $(u_j)_{j=1}^m$ is also $1$-dominated by the $\ell_p$ basis, if $(Au_j)_{j=1}^m$ were to $2K$-dominate the $\ell_p$ basis, $(u_j)_{j=1}^m$ would imply that $\textbf{NP}_p(A,X,Y,2K)>m$, another contradiction.  Thus there exists a $p$-absolutely convex combination $u$ of $(u_j)_{j=1}^m$, and therefore of $(x_i)_{i=1}^{mn}$, so that $\|Au\|<1/2K$. Then with $u=\sum_{j=1}^m a_ju_j$, $$\|(A+B)u\|\leqslant \|Au\|+ \sum_{j=1}^m |a_j|\|Bu_j\| <1/2K+m(1/2Km) = 1/K.$$  This shows $\textbf{NP}_p(A+B, X, Y, K)\leqslant mn<\omega$.  Since $K$ was arbitrary, we are done.

\end{proof}

\subsection{Local strictly singular indices}

We recall that for $X,Y\in\Ban$ and $A\in \mathfrak{L}(X,Y)$, $A$ is \emph{strictly singular} if for each infinite dimensional $Z\leqslant X$, $A|_Z$ is not an isomorphism.  Moreover, $A$ is said to be \emph{finitely strictly singular} if for any $\ee>0$, there exists $n=n(\ee)\in \nn$ so that for any $E\leqslant X$ with $\dim E=n$, there exists $x\in E$ with $\|Ax\|<\ee \|x\|$.  In \cite{ADST}, the notion of a $\xi$-\emph{strictly singular operator} was defined.  An operator $A:X\to Y$ is called $\xi$-strictly singular if for any basic sequence $(x_n)\subset X$ and any $K\geqslant 1$, there exists $E\in \sss_\xi$ and $x\in [x_i:i\in E]$ so that $\|Ax\|<\ee \|x\|$. We let $\sss\sss_\xi(X,Y)$ denote the $\xi$-strictly singular operators from $X$ to $Y$, $\sss\sss_\xi$ the collection of all components $\sss\sss_\xi(X,Y)$.   If $X$ is separable, then for any $Y\in \Ban$, the strictly singular operators in $\mathfrak{L}(X,Y)$ coincide with the operators in $\mathfrak{L}(X,Y)$ which are $\xi$-strictly singular for some $\xi<\omega_1$.    We define the following trees for $X,Y\in \Ban$, $A:X\to Y$, and $K\geqslant 1$.

$$SS(A,X,Y,K)=\Bigl\{(x_i)_{i=1}^n\in S_X^{<\nn}:(x_i)_{i=1}^n \text{\ is\ }K\text{-basic}, (x_i)_{i=1}^n\lesssim_K (Ax_i)_{i=1}^n\Bigr\}.$$  Note that our blocking arguments for the Bourgain $\ell_p$ index of an operator relied on the fact that the trees $T_p(A,X,Y,K)$ are $p$-absolutely convex. All derived trees of the tree $SS(A,X,Y,K)$ are block closed, which we recall means that normalized blocks of a member of a derived tree of $SS(A,X,Y,K)$ are members of the same derived tree. The arguments above with $p$-absolutely convex blocks replaced by normalized blocks yield many similar results below with only minor modifications.    We define $\textbf{SS}(A,X,Y,K)=o(SS(A,X,Y,K))$ and $\textbf{SS}(A,X,Y)=\sup_{K\geqslant 1} \textbf{SS}(A,X,Y,K)$. We let $$\mathfrak{SS}^\xi(X,Y)=\{A\in \mathfrak{L}(X,Y): \textbf{SS}(A,X,Y)\leqslant \omega^\xi\}.$$  We let $\mathfrak{SS}(X,Y)$ denote the strictly singular operators from $X$ into $Y$.  

\begin{theorem} Fix $X,Y\in \Ban$, $B\in \mathfrak{L}(X,Y)$, $K\geqslant 1$.  \begin{enumerate}[(i)]\item $\mathfrak{SS}=\cup_{\xi\in \ord} \mathfrak{SS}^\xi$. \item For any Schauder basis $(e_i)$, $X,Y\in \Ban$, and $A\in \mathfrak{L}(X,Y)$, $\emph{\textbf{NP}}_{(e_i)}(A,X,Y)\leqslant \emph{\textbf{SS}}(A,X,Y),$ and consequently $\mathfrak{SS}^\xi\subset \mathfrak{NP}_p^\xi$ for all $\xi\in \ord$, $1\leqslant p\leqslant \infty$.  \item $\mathfrak{SS}^1=\mathfrak{NP}_2^1$ consists of all finitely strictly singular operators.  \item If $B$ is finite rank, $\emph{\textbf{SS}}(B,X,Y)=1+\emph{\rank}(B)$. \item For any $\xi\in \ord$, $\{A\in \mathfrak{L}(X,Y): \emph{\textbf{SS}}(A,X,Y)\leqslant \xi\}$ is norm closed in $\mathfrak{L}(X,Y)$.  \item For any $W,Z\in \Ban$, $A\in \mathfrak{L}(Y,Z)$, $C\in \mathfrak{L}(W,X)$, $\emph{\textbf{SS}}(ABC, W,Z)\leqslant \emph{\textbf{SS}}(B, X,Y)$.  \item If $X$ is separable, $\mathfrak{SS}(X,Y)=\cup_{\xi<\omega_1}\mathfrak{SS}^\xi(X,Y)$.  \item If $B\in \mathfrak{SS}(X,Y)$ is not finite rank, then there exists $0<\xi\in \ord$ so that $\emph{\textbf{SS}}(B,X,Y)=\omega^\xi$.  \item For any $\ee>0$ and $A\in \mathfrak{L}(X,Y)$, $$\emph{\textbf{SS}}(A+B, X,Y, K)\leqslant \emph{\textbf{SS}}(A, X, Y)\emph{\textbf{SS}}(B, X, Y, K+\ee).$$  \item For any $0<\xi\in \ord$, $\cup_{\zeta<\omega^\xi} \mathfrak{SS}^\zeta$ is an operator ideal, closed if $\xi$ has uncountable cofinality.

\end{enumerate}

\end{theorem}

\begin{proof}[Sketch] $(i)$ is clear.  

$(ii)$ Note that if $(x_i)_{i=1}^n\subset B_X$ is such that $(x_i)_{i=1}^n\lesssim_1 (e_i)_{i=1}^n$ and $(e_i)_{i=1}^n\lesssim_K (Ax_i)_{i=1}^n$, then since $(Ax_i)_{i=1}^n \lesssim_{\|A\|} (x_i)_{i=1}^n$, we deduce that $(x_i)_{i=1}^n$ is $b\|A\|K$-basic and $(x_i)_{i=1}^n\lesssim_K (Ax_i)_{i=1}^n$, where $b$ denotes the basis constant of $(e_i)$.  This means that $(x_i/\|x_i\|)_{i=1}^n$ is $b\|A\|K$-basic and $(x_i/\|x_i\|)_{i=1}^n\lesssim_K (Ax_i/\|x_i\|)_{i=1}^n$.  This means that $$\{(x_i/\|x_i\|)_{i=1}^n: (x_i)_{i=1}^n\in T_{(e_i)}(A,X,Y,K)\}\subset SS(A, X, Y, b\|A\|K \vee K),$$ and by induction, $$\{(x_i/\|x_i\|)_{i=1}^n: (x_i)_{i=1}^n\in T_{(e_i)}(A,X,Y,K)^\xi\}\subset SS(A, X, Y, b\|A\|K\vee K)^\xi $$ for each $\xi\in \ord$.  This gives the first statement, and the second follows immediately.  

For $(iii)$, note that for $A:X\to Y$, if $\textbf{SS}(A,X,Y)>\omega$ then there exists a sequence $(E_n)$ of finite dimensional subspaces of $X$ so that $\dim E_n\to \infty$ and so that $T|_{E_n}$ is a $K$-isomorphism of $E_n$ and its image for all $n\in \nn$.  By Dvoretsky's theorem, by passing to a subsequence of the spaces $(E_n)$, we may assume without loss of generality that for each $n\in \nn$ there exists a subspace $F_n$ of $E_n$ so that $\dim F_n=n$ and $F_n$ is $2$-isomorphic to $\ell_2^n$.  If $(x_i^n)_{i=1}^n$ is a basis for $F_n$ which is $1$-dominated and $2$-dominating the $\ell_2^n$ basis, these sequences are $K$-dominated by their images under $A$, whence the $\ell_2^n$ basis is $2K$-dominated by $(Ax^n_i)_{i=1}^n$.  These sequences witness the fact that $\textbf{NP}_2(A,X,Y,2K)>\omega$.  This implies that $\mathfrak{NP}_2^1\subset \mathfrak{SS}^1$.  The reverse inclusion follows from $(ii)$.  To see that $\mathfrak{SS}^1$ consists of finitely strictly singular operators, note that every finitely strictly singular operator $A:X\to Y$ necessarily satisfies $\textbf{SS}(A,X,Y)\leqslant \omega$.  This is because for any $\ee>0$, there exists $n=n(\ee)\in \nn$ so that if $E\leqslant X$ with $\dim E=n$, there exists $x\in E$ with $\|Ax\|<\ee \|x\|$.  Thus if $(x_i)_{i=1}^n\subset S_X$ is $K$-basic, there exists $x=\sum_{i=1}^n a_ix_i$ so that $\|Ax\|<\ee \|x\|$, which means $\textbf{SS}(A,X,Y,\ee^{-1})\leqslant n$.  This shows that every finitely strictly singular operator lies in $\mathfrak{SS}^1$.  By arguing as above, if $A:X\to Y$ is not finitely strictly singular, then there must exist $K> 1$ and a sequence of subspaces $(E_n)$ of $X$ so that $\dim E_n\to \infty$ and $T|_{E_n}$ is $K$-isomorhpic to its image.  By passing to subspaces of a subsequence of $(E_n)$, we may assume that $E_n$ is closely isomorphic to $\ell_2^n$ and is spanned by a sequence $(x_i^n)_{i=1}^n$ which is normalized and $K$-basic.  These sequences witness that $\textbf{SS}(A,X,Y, K)>\omega$.    

$(iv), (v)$ are trivial modifications of Proposition \ref{tedious}.  

$(vi)$ Suppose $(w_i)_{i=1}^n\in SS(ABC,W,Z,K)$.  Then $$(w_i)_{i=1}^n\lesssim_K (ABCw_i)_{i=1}^n\lesssim_{\|A\|\|B\|}(Cw_i)_{i=1}^n\lesssim_{\|C\|} (w_i)_{i=1}^n$$ implies that $(Cw_i)_{i=1}^n$ is $\|A\|\|B\|\|C\|K$-basic and $(Cw_i)_{i=1}^n\lesssim_{\|A\|\|C\|K}(BCw_i)_{i=1}^n$.  This means $(Cw_i/\|Cw_i\|)_{i=1}^n$ is $\|A\|\|B\|\|C\|K$-basic and $\|A\|\|C\|K$-dominated by $(BCw_i/\|Cw_i\|)_{i=1}^n$.  Therefore $$\{(Cw_i/\|Cw_i\|)_{i=1}^n: (w_i)_{i=1}^n\in SS(ABC, W,Z,K)\}\subset SS(B,X,Y,\|A\|\|B\|\|C\|K\vee \|A\|\|C\|K),$$ and by induction, $$\{(Cw_i/\|Cw_i\|)_{i=1}^n: (w_i)_{i=1}^n\in SS(ABC, W,Z,K)^\xi\}\subset SS(B,X,Y,\|A\|\|B\|\|C\|K\vee \|A\|\|C\|K)^\xi $$ for each $\xi\in \ord$, which gives the result.  

$(vii)$ This is another application of Bourgain's version of the Kunen-Martin theorem, noting that $SS(A,X,Y,K)$ is closed for each $K\geqslant 1$.  

$(viii)$  This proceeds as in Proposition \ref{special form}$(ii)$.  We only note that if $(x_i)_{i=1}^n\in SS(A, X, Y, K)$, $F\subset S_{Y^*}$ is $2$-norming for $[x_i]_{i=1}^n$, and $(u_j)_{j=1}^m\subset \cap_{y^*\in F}\ker(y^*)$, $(u_j)_{j=1}^m\in SS(A, X, Y, K)$, then $(x_1, \ldots, x_n, u_1, \ldots, u_m)\in SS(A, X, Y, 6K)$.  

$(ix)$ This follows as in Lemma \ref{product}$(i)$ with the assumption that $\|Au_t\|<\delta_{|t|}/2K$.  The factor of $2K$ is required since we can only guarantee in this case that $\max_{1\leqslant i\leqslant |t|}|a_i|\leqslant 2K\|\sum_{i=1}^{|t|} a_iu_{t|_i}\|$ for scalar sequences $(a_i)_{i=1}^{|t|}$.    

$(x)$ This follows again from $(ix)$ and the fact that if $\textbf{SS}(A,X,Y), \textbf{SS}(B,X,Y)<\omega^{\omega^\xi}$, $\textbf{SS}(A+B, X, Y)\leqslant \textbf{SS}(A,X,Y)\textbf{SS}(B,X,Y)<\omega^{\omega^\xi}$.  If $\xi$ has uncountable cofinality and $A_n, A:X\to Y$ are such that $A_n\to A$, and $\textbf{SS}(A_n, X, Y)<\omega^{\omega^\xi}$, then $\textbf{SS}(A,X,Y)\leqslant \sup_n \textbf{SS}(A_n, X, Y)<\omega^{\omega^\xi}$.  If $\textbf{SS}(A,X,Y)=\omega^\zeta$, we deduce $\zeta<\omega^\xi$, and $A\in \mathfrak{SS}^\zeta\subset \cup_{\eta<\omega^\xi}\mathfrak{SS}^\eta$.

\end{proof}

\section{Sequential indices} \subsection{Operators preserving no $\ell_p^\xi$ spreading model} For $0<\xi\leqslant \omega_1$, and a sequence $(x_i)$ in a Banach space, we say $(x_i)$ is an $\ell_p^\xi$ spreading model provided there exist $a,b>0$ so that for each $E\in\sss_\xi$, $(e_i)_{i\in E}\lesssim_a (x_i)_{i\in E}$ and $(x_i)_{i\in E}\lesssim_b (e_i)_{i\in E}$, where $(e_i)$ is the canonical $\ell_p$ basis.  Note that if $\xi=\omega_1$, our convention that $\sss_{\omega_1}=\nn^{<\nn}$ simply means that $(x_i)$ is equivalent to the $\ell_p$ basis. Since $\sss_\xi$ is spreading for each $\xi$, any subsequence of an $\ell_p^\xi$ spreading model is one as well with the same constants.   The notion of $c_0^\xi$ spreading model is defined similarly.  For $X,Y\in \Ban$, $1\leqslant p< \infty$, and $0<\xi\leqslant \omega_1$, we let $\mathfrak{SM}_p^\xi(X,Y)$ consist of all operators $A\in \mathfrak{L}(X,Y)$ so that if $(x_n)\subset X$ is an $\ell_p^\xi$ spreading model, then $(Ax_n)$ is not an $\ell_p^\xi$ spreading model.  The class $\mathfrak{SM}_\infty^\xi(X,Y)$ is defined similarly for $c_0^\xi$ spreading models. As usual, we let $\mathfrak{SM}_p^\xi$ consist of all operators lying in $\mathfrak{SM}_p^\xi(X,Y)$ for some $X,Y\in \Ban$.    Note that $\mathfrak{SM}_p^{\omega_1}= \mathfrak{NP}_p$, the operators not preserving a copy of $\ell_p$ (resp. $c_0$).  If $A\in \mathfrak{SM}_p^\xi$, we say $A$ preserves no $\ell_p^\xi$ (or $c_0^\xi$) spreading model.  We let $\textbf{SM}_p(A,X,Y)$ denote the smallest ordinal $\xi\in [1,\omega_1]$ so that $A$ preserves no $\ell_p^\xi$ (or $c_0^\xi$ if $p=\infty$) spreading model, provided such an ordinal exists, and $\textbf{SM}_p(A,X,Y)=\infty$ otherwise.  We obey a similar convention as with the local indices that $\textbf{SM}_p(X)=\textbf{SM}_p(I_X, X, X)$.  

The proof of the following proposition is similar to that of Proposition \ref{tedious}, so we omit it.  

\begin{proposition} Fix $X,Y\in \Ban$, $\xi\leqslant \omega_1$.  \begin{enumerate}[(i)]\item $\mathfrak{SM}_p^\xi$ is norm closed in $\mathfrak{L}(X,Y)$.  \item For $W, Z\in\Ban$, $ABC\in \mathfrak{SM}_p^\xi(W,Z)$ whenever $B\in \mathfrak{SM}_p^\xi(X,Y)$ and $C\in \mathfrak{L}(W, X)$, $A\in \mathfrak{L}(Y,Z)$. \item If $X$ is separable, $\mathfrak{NP}_p(X,Y)=\cup_{\zeta<\omega_1} \mathfrak{SM}_p^\zeta(X,Y)$.   \end{enumerate}
\end{proposition}

We remark at this point that there exist (necessarily non-separable) Banach spaces admitting no copy of $\ell_p$ (resp. $c_0$) but admitting for all $\xi<\omega_1$ an $\ell_p^\xi$ (resp. $c_0^\xi$) spreading model.  For example, the $\ell_2$ sum $\bigl(\oplus X_\xi)_{\ell_2[1, \omega_1)}$, $X_\xi$ the Schreier spaces of \cite{AA}, is such a space for $p=1$.  For $1<p$, the $p$-convexification of this space admits an $\ell_p^\xi$ spreading model for all $\xi<\omega_1$, and the dual of this space admits $c_0^\xi$ spreading models for all countable $\xi$.  Thus $\mathfrak{NP}_p= \mathfrak{SM}_p^{\omega_1}\neq \cup_{\zeta<\omega_1}\mathfrak{SM}_p^\zeta$.  As we will see later, the union $\cup_{\zeta<\omega_1}\mathfrak{SM}_p^\zeta$ is a closed ideal distinct from the ideal of operators not preserving a copy of $\ell_p$.  

We have the following analogue of Lemma \ref{product}.  The first part is similar to an argument concerning sums of $\xi$-and $\zeta$-strictly singular operators.  

\begin{lemma} Fix $0<\xi, \zeta<\omega_1$, $X,Y\in \Ban$, $1\leqslant p\leqslant \infty$. Assume $A\in \mathfrak{SM}_p^\xi(X,Y)$ and $B\in \mathfrak{SM}_p^\zeta(X,Y)$.   \begin{enumerate}[(i)]\item If $1<p<\infty$, $A+B\in \mathfrak{SM}_p^{\xi+\zeta}(X,Y)$. \item If $p\in \{1, \infty\}$, $A+B\in \mathfrak{SM}_p^{\xi\vee \zeta}(X,Y)$. \end{enumerate}

\label{product 2}

\end{lemma}

\begin{theorem} For each $1\leqslant p\leqslant \infty$, $\xi\leqslant \omega_1$, $\cup_{\zeta<\omega^\xi} \mathfrak{SM}_p^\zeta$ is an operator ideal, closed if $\xi=\omega_1$.  Moreover, $\mathfrak{SM}_p^\xi$ is a closed operator ideal if $p\in \{1,\infty\}$.

\end{theorem}

\begin{proof}[Proof of Lemma \ref{product 2}]  $(i)$ Fix $1<p<\infty$.  Fix $(x_n)\subset B_X$ and assume that for each $E\in \sss_{\zeta+\xi}$, $(x_n)_{n\in E}\lesssim_1 (e_i)_{i\in E}$, where $(e_i)$ is the $\ell_p$ basis.  If no such sequence exists, then $X$ admits no $\ell_p^{\zeta+\xi}$ spreading model, and obviously $A+B$ can preserve no $\ell_p^{\zeta+\xi}$ spreading model, and we reach the conclusion trivially.  Then since $A$ preserves no $\ell_p^\xi$ spreading model, for any $\ee>0$ and any subsequence $(x_n)_{n\in M}$ of $(x_n)$ and any $k\in \nn$, there exists $E\in \sss_\xi$ with $k\leqslant E$ and scalars $(a_i)_{i\in E}$ having $p$-norm equal to $1$ and so that $\|\sum_{i\in E} a_iA x_{m_i}\|<\ee$.  We choose $M\in \infin$ so that $\sss_\xi[\sss_\zeta](M)\subset \sss_{\zeta+\xi}$.  We then choose $E_1<E_2<\ldots$ and a $p$-absolutely convex block $(z_n)$ of $(x_n)$ so that $z_n=\sum_{i\in E_n} a_ix_{m_i}$ and $\|Az_n\|<\ee_n$, where $\ee_n\downarrow 0$ is chosen so that $\sum \ee_n<\infty$.  Then our choice of $M$ guarantees that $(z_n)_{n\in E}$ is $1$-dominated by the $\ell_p$ basis for each $E\in \sss_\zeta$.  Since $B$ preserves no $\ell_p^\zeta$ spreading model, for any $\ee>0$ and $k\in \nn$ there exist $E\in \sss_\zeta$ with $k\leqslant E$ and scalars $(b_i)_{i\in E}$ having $p$-norm equal to $1$ and so that $\|\sum_{i\in E}b_iBz_i\|< \ee$.  Then $$\Bigl\|\sum_{i\in E}b_i(A+B)z_i\Bigr\|\leqslant \sum_{i\in E}\|Az_i\|+ \ee \leqslant \sum_{i=k}^\infty \ee_i + \ee.$$  Since $k$ and $\ee>0$ were arbitrary, this quantity can be made arbitrarily small.  This means $((A+B)x_n)$ is not an $\ell_p^{\zeta+\xi}$ spreading model, since $\sum_{i\in E}b_i(A+B)z_i$ is a $p$-absolutely convex combination of $(x_{m_n})_{n\in \cup_{i\in E}E_i}$ and $\cup_{i\in E}E_i\in \sss_\zeta[\sss_\xi](M)\subset \sss_{\xi+\zeta}$.

$(ii)$ Assume $\xi=\xi\vee \zeta$.  First consider $p=\infty$.  If $X$ admits no $c_0^\xi$ spreading model, the result is trivial.  Assume $(x_n)\subset X$ is a $c_0^\xi$ spreading model.  Then if $\lim \sup \|Ax_n\|>\ee>0$, by passing to a subsequence we may assume $\|Ax_n\|>\ee$ for all $n\in \nn$.  Since any $c_0^\xi$ spreading model is weakly null, we may also assume $(Ax_n)$ is basic, in which case it dominates the $c_0$ basis, so we have the appropriate lower estimates.  The upper estimates to witness that $(Ax_n)$ is a $c_0^\xi$ spreading model come from comparison to $(x_n)$, and we reach a contradiction.  Thus $Ax_n\to 0$.  Next, note that since $\zeta\leqslant \xi$, the almost monotone property of the Schreier families gives that some subsequence of $(x_n)$ is a $c_0^\zeta$ spreading model, and $Bx_n\to 0$.  Therefore $(A+B)x_n\to 0$, and $((A+B)x_n)$ is not a $c_0^\xi$ spreading model.  

Next, consider $p=1$.  Suppose $(x_n)\subset B_X$ is such that $((A+B)x_n)$ is an $\ell_1^\xi$ spreading model.  Note that no subsequence of either $(Ax_n)$ or $(Bx_n)$ can be equivalent to the $\ell_1$ basis, and by Rosenthal's $\ell_1$ theorem we can assume $(Ax_n)$ and $(Bx_n)$ are both weakly Cauchy.  By passing to an appropriate subsequence and taking a difference sequence, we can asume $(Ax_n)$ and $(Bx_n)$ are both weakly null.  By \cite{AMT}, either some subsequence of $(Ax_n)$ is an $\ell_1^\xi$ spreading model, or there exists $N\in \infin$ so that for all $L\in [N]$, $\sum_{i\in \supp(\xi_n^L)}\xi^L_n(i)Ax_i\underset{n}{\to} 0$, where $(\xi_n^L)\subset c_{00}$ denotes the repeated averages hierarchy block corresponding to $L$ and $\xi$, and $\xi_n^L=(\xi_n^L(i))_i$. Of course, the second alternative must hold.  Using \cite{AMT} again, either there exists $M\in [N]$ so that $(Bx_i)_{i\in M}$ is an $\ell_1^\xi$ spreading model, or there exists $M\in [N]$ so that for all $L\in [M]$, $\sum_{i\in \supp(\xi_n^L)}\xi^L_n(i)Bx_i\underset{n}{\to} 0$, and again our hypothesis guarantees that the second alternative must hold.  Therefore $\sum_{i\in \supp(\xi_n^M)}\xi^M_n(i)(A+B)x_i\underset{n}{\to}0$.  Since $\supp(\xi^M_n)\in \sss_\xi$ and each $\xi^m_n$ is a convex combination of the $c_{00}$ basis, this shows that $((A+B)x_n)$ cannot be an $\ell_1^\xi$ spreading model.  

\end{proof}

\subsection{Weakly compact index} Let $\mathfrak{WC}$ denote the ideal of weakly compact operators.  We define $\mathfrak{WC}^\xi=\mathfrak{WC}\cap \mathfrak{SM}_1^\xi$.   Note that this is a closed ideal, being the intersection of two closed ideals. We let $\textbf{SWC}(A,X,Y)$ be the minimum ordinal in $[1, \omega_1]$ so that $A\in \mathfrak{WC}^\xi$, if such an ordinal exists, and $\textbf{SWC}(A,X,Y)=\infty$ otherwise. Again, we let $\textbf{SWC}(X)=\textbf{SWC}(I_X,X,X)$.  If $X_{\xi,2}$ denotes the completion of $c_{00}$ under the norm $$\|x\|_{X_{\xi,2}} = \sup\Bigl\{\Bigl(\sum_{i=1}^\infty \|E_i x\|_{\ell_1}^2\Bigr)^{1/2}: E_1<E_2<\ldots, E_i\in \sss_\xi\Bigr\},$$ it was shown in \cite{C3} that $X_{\xi,2}$ admits no $\ell_1^{\xi+1}$ spreading model for $0\leqslant \xi<\omega_1$.  Moreover, $X_{\xi,2}$ is reflexive and the basis is an $\ell_1^\xi$ spreading model.  Therefore $\textbf{SWC}(X_{p,2}) = \xi+1$ and $\mathfrak{WC}^{\xi+1}\setminus \mathfrak{WC}^\xi\neq \varnothing$ for each $0\leqslant \xi<\omega_1$.  We last observe that for a separable Banach space $X$ and any Banach space $Y$, $A:X\to Y$ is weakly compact if and only if there exists $\xi<\omega_1$ so that $A\in \mathfrak{WC}^\xi(X,Y)$.  This is because if $A:X\to Y$ is weakly compact, then if $A\notin \mathfrak{WC}^\xi$, $A\notin \mathfrak{SM}^\xi$.  But since $\mathfrak{NP}_1(X,Y)= \cup_{\xi<\omega_1}\mathfrak{SM}_1^\xi(X,Y)$, $A\in \mathfrak{WC}(X,Y)\setminus \cup_{\xi<\omega_1}\mathfrak{WC}^\xi(X,Y)$ implies that $A$ preserves a copy of $\ell_1$, contradicting the assumption of weak compactness.  

Thus we have arrived at \begin{theorem} For each $0< \xi< \omega_1$, $\mathfrak{WC}^\xi$ is a closed operator ideal and $\mathfrak{SM}^{\xi+1}\setminus \mathfrak{SM}^\xi$ is non-empty.  Moreover, if $X,Y\in \Ban$ and if $X$ is separable, then $\mathfrak{WC}(X,Y)=\cup_{\xi<\omega_1}\mathfrak{WC}^\xi(X,Y)$.

\end{theorem}

At first, this definition may seem somewhat artificial, but an equivalent, more apparently natural definition has appeared previously in the literature \cite{BF}.  Of course, $X$ is reflexive if and only if any bounded sequence $(x_n)$ has a weakly converging subsequence, which is equivalent to every bounded sequence in $X$ having a convex block which is norm convergent.  In \cite{AMT}, the Schreier families and repeated averages hierarchy were used to quantify the complexity of the blocking required to witness the convex block of a weakly converging subsequence which is norm convergent.  In complete analogy, the operator $A:X\to Y$ is weakly compact if and only if every for every sequence $(x_n)\subset B_X$, some subsequence of $(Ax_n)$ is weakly convergent or, equivalently, for every sequence $(x_n)\subset B_X$, some subsequence of $(Ax_n)$ has a convex block converging in norm in $Y$.  The stratification of $\mathfrak{WC}$ into the classes $\mathfrak{WC}^\xi$ also measures the complexity of a convex block of a subsequence of $(x_n)$ which has norm converging image sequence.  In \cite{AMT}, the authors defined $\xi$ and $(\xi, M)$ convergent.   For $\xi<\omega_1$ and $M\in [\nn]$, the sequence $(y_n)$ converging weakly to $y$ is $(\xi, M)$ convergent to $y$ if  $\|y-\sum_{i\in \supp(\xi^M_n)} \xi^M_n(i)y_i\|\underset{n\to \infty}{\to} 0$.  The sequence $(y_n)$ converging weakly to $y$ is $\xi$ convergent to $y$ if for any $N\in [\nn]$, there exists $M\in [N]$ so that $(y_n)$ is $(\xi, M)$ convergent to $y$. 

Negating the characterization of weak compactness above, one can deduce that the operator $A:X\to Y$ fails to be weakly compact if and only if there exists $(x_n)\subset B_X$ so that $(Ax_n)$ dominates the summing basis $(s_i)$, the norm of which is given by $\|\sum_{i=1}^n a_is_i\|= \max_{1\leqslant m\leqslant n}|\sum_{i=1}^m a_i|$.  In \cite{BF}, for $\xi<\omega_1$, an operator $A:X\to Y$ was called $\sss_\xi$-\emph{weakly compact} if for any seminormalized basic sequence $(x_n)\subset X$ and any $\ee>0$, there exist $E\in \sss_\xi$ and scalars $(a_i)_{i\in E}$ with $$\|\sum_{i\in E} a_i Ax_i\|< \ee \|\sum_{i\in E} a_is_i\|.$$ We note that these notions both lead to the same quantification.

\begin{proposition} Let $A:X\to Y$ be an operator, $\xi<\omega_1$. Then $A$ is $\sss_\xi$ weakly compact if and only if $A\in \mathfrak{WC}^\xi$.  
\end{proposition}

\begin{proof} Assume $A\notin \mathfrak{WC}^\xi$.  If $A$ fails to be weakly compact, then of course $A$ fails to be $\sss_\xi$ weakly compact.  If $A$ is weakly compact, then there exists $(x_n)\subset B_X$ so that $(Ax_n)$ is an $\ell_1^\xi$ spreading model. Then for some $K\geqslant 1$ and all $E\in \sss_\xi$, $(Ax_n)_{n\in E}$ $K$-dominates the $\ell_1^{|E|}$ basis, and therefore $K$-dominates $(s_i)_{i=1}^{|E|}$.  This implies that $A$ is not $\sss_\xi$-weakly compact, since  $(s_i)$ is isometrically equivalent to all of its subsequences.  

Next, assume $A\in \mathfrak{WC}^\xi$.  Fix $(x_n)\subset B_X$.  By passing to a subsequence, we may assume $(Ax_n)$ converges weakly to some $y\in \overline{AB_X}^w=\overline{AB_X}$. Then there exists $(u_n)\subset B_X$ so that $(Au_n)$ converges in norm to $y$. If a subsequence of $(A(x_n-u_n))$ is norm null, then the corresponding subsequence of $(Ax_n)$ converges in norm to $y$, and we are done.  Otherwise we can pass to a subsequence and assume $(A(x_n-u_n))$ is convexly unconditional \cite{AMT}.   Recall that for $M\in \infin$, $(\xi^M_n)_n$ denotes the repeated averages hierarchy blocking corresponding to $\xi$ and $M$.  By \cite{AMT}, either some subsequence of $(A(x_n-u_n))$ is an $\ell_1^\xi$ spreading model, which is impossible since $A\in \mathfrak{WC}^\xi$, or there exists $M\in \infin$ so that $\sum_{i\in \supp(\xi^M_n)} \xi^M_n(i)(x_i-u_i)\underset{n\to \infty}{\to}0$ in norm.  Note that $\|\xi^M_n\|_{c_0}\underset{n\to \infty}{\to}0$, so we may partition $\supp(\xi^M_n)$ into $A_n< B_n$ so that  $\sum_{i\in A_n}\xi^M_n(i),\sum_{i\in B_n}\xi^m_n(i)\underset{n\to \infty}{\to}1/2$.  By convex unconditionality, if $\ee_i=1$ for $i\in A_n$ and $-1$ for $i\in B_n$, $\|\sum_{i\in \supp(\xi_n^M)} \ee_i \xi^M_n(i)(x_i-u_n)\|\to 0$.  But \begin{align*} \|\sum_{i\in \supp(\xi^M_n)} \ee_i\xi^M_n(i)x_i\| & \leqslant \|\sum_{i\in \supp(\xi^M_n)} \ee_i \xi^M_n(i)(x_i-u_i)\| + \|\sum_{i\in A_n} \xi^M_n(i)u_i - \sum_{i\in B_n} \xi^M_n(i) u_i\| \\ & \to 0 + \|\frac{1}{2}y-\frac{1}{2}y\|=0.\end{align*}

Since $\lim \sup_n \|\sum_{i\in \supp(\xi^M_n)} \ee_i \xi^M_n(i)s_i\|\geqslant 1/2$, this proves that there does not exist $K$ so that for all $E\in \sss_\xi$, $(x_n)_{n\in E}$ $K$-dominates the summing basis $(s_i)_{i\in E}$.  This proves that $A$ is $\sss_\xi$ weakly compact.  

\end{proof}

\section{Dualization} 

Given a normalized, bimonotone Schauder basis $(e_i)$ with coordinate functionals $(e_i^*)$ and an operator $A:X\to Y$, a natural question to ask is how  $\textbf{NP}_{(e_i)}(A, X, Y)$ and $\textbf{NP}_{(e_i^*)}(A^*, Y^*, X^*)$ may compare.  From \cite{BD}, \cite{H}, and \cite{FOS}, we deduce that in general these indices may be drastically different.  It follows from \cite{FOS} that if $(e_i)$ is any shrinking basis, there exists a $\mathcal{L}_\infty$ Banach space $Z_{(e_i)}$ admitting a sequence equivalent to $(e_i)$ so that $Z_{(e_i)}^*\approx \ell_1$.  This means that if $(e_i)$ is the $\ell_p$ basis for $1<p<2$, $\textbf{NP}_p(I_{Z}, Z, Z)=\infty$, while $\textbf{NP}_q(I_{Z^*}, Z^*, Z^*)=\omega$, since $\ell_q$ is not finitely representable in $\ell_1$ for $2<q<\infty$. Additionally, one can take separable, reflexive spaces admitting large $\ell_1$ indices, for example the Schreier spaces, and embed these as well into Banach spaces having duals isomorphic to $\ell_1$, which has the smallest possible $c_0$ indices.  These examples show that it is impossible in general to deduce any connection between $\textbf{NP}_{(e_i)}(A, X, Y)$ and $\textbf{NP}_{(e_i^*)}(A^*, Y^*, X^*)$.  However, we do establish the following sharp relationship.  

\begin{theorem} Let $\xi\in \ord$, $X,Y\in \Ban$, and $A\in \mathfrak{L}(X,Y)$.   \begin{enumerate}[(i)]\item If $A\in \mathfrak{NP}_1^\xi(X,Y)$, then $A^*\in \mathfrak{NP}_\infty^\xi(Y^*,X^*)$.  \item If $A^*\in \mathfrak{NP}_1^\xi(Y^*, X^*)$, then $A\in \mathfrak{NP}_\infty^\xi(X, Y)$. \item If $0<\xi<\omega_1$ and $A\in \mathfrak{SM}_1^\xi(X,Y)$, then $A^*\in \mathfrak{SM}_\infty^\xi(Y^*, X^*)$.  \item If $0<\xi < \omega_1$, and $A^*\in \mathfrak{SM}_1^\xi(Y^*, X^*)$, then $A\in \mathfrak{SM}_\infty^\xi(X,Y)$.  \end{enumerate} 

\label{dualize}

\end{theorem}

Note that for $(i)$ and $(ii)$, the $\xi=0$ case reduces to the case that either $A$ or $A^*$ is the zero operator, and there is nothing to prove.  Therefore in the proof below, we consider only $0<\xi$.  The positive result here is due to the fact that $\ell_1$ structure requires only a one-sided estimate, and that this estimate can be found by norming vectors with functionals acting on them biorthogonally and exhibiting $c_0$ structure.   

Parts $(iii)$ and $(iv)$ of Theorem \ref{dualize} follow from standard techniques.  If $(x_i)\subset B_X$ is such that $(x_i)$ and $(Ax_i)$ are both $c_0^\xi$ spreading models, then $(x_i)$ and $(Ax_i)$ are both weakly null.  By standard arguments, if $0<\ee< \inf \|Ax_i\|$, then for any $\ee_i\downarrow 0$ we can find $(y_i^*)\subset B_{Y^*}$ so that $y^*_i(Ax_i)\geqslant \ee$ and, by passing to subsequences of $(x_i)$ and $(y_i^*)$, assume that $|y^*_i(Ax_j)|<\ee_j$ for all $1\leqslant i<j$.   By Rosenthal's $\ell_1$ dichotomy, either some subsequence of $(A^*y_i^*)$ is equivalent to the $\ell_1$ basis, in which case we are done, or we can pass to a difference sequence of a weakly Cauchy subsequence of $(A^*y_i^*)$ and, by another diagonalization, obtain a subsequence $(x_{n_i})$ and a difference sequence $(z_i^*)$ of a subsequence of $(y_i^*)$ so that $(A^*z_i^*)$ is weakly null.   By passing to a subsequence as before, we may assume $|z_i^*(Ax_{n_j})|<\ee_{\max\{i,j\}}$ for all $i\neq j$ and $z_i^*(Ax_{n_i})>\ee/2$.  Choosing $\ee_i\downarrow 0$ rapidly enough (depending on $\ee$) allows us to use $\infty$-convex combinations of $(x_{n_i})_{i\in E}$ to appropriately norm any linear combination of $(A^*z_i^*)_{i\in E}$, $E\in \sss_\xi$, to witness that $(A^*z_i^*)$, and therefore $(z_i^*)$, is an $\ell_1^\xi$ spreading model.  The argument is the same if $(y^*_i)$ and $(A^*y^*_i)$ are $c_0^\xi$ spreading models, except that we norm $A^*y_i^*$ by a member of $X$ rather than $X^{**}$.  

The method for proving $(i)$ and $(ii)$ will again require us to find functionals to biorthogonally norm the vectors witnessing $\ell_1$ structure.  The method will follow easily from the next technical lemma. The proof is an inessential modification of the non-operator version from \cite{C2}, so we omit it.

\begin{lemma} Fix $\zeta, \xi\in \ord$ with $0< \xi$.  Fix $n\in \nn$ and $K\geqslant 1$.  Let $F\subset Y^*$ be finite, $b$ a member of $T_\infty(A, X, Y, K)$ be such that $o(T_\infty(A, X, Y, K)^\zeta(b))>\omega^\xi n$.  Then for any $C>K$, there exist a $B$-tree $\ttt$ with $o(\ttt)=\omega^\xi n$, vectors $(x_t)_{t\in \ttt}\subset B_X$, and functionals $(y_t^*)_{t\in \ttt} \subset CB_{Y^*}$ so that the following hold for every $t\in \ttt$: \begin{enumerate}[(i)]\item $y^*_t(Ax_t)=1$, \item for $s\in \ttt$ comparable to $t$ and not equal to $t$, $y^*_s(Ax_t)=y_t^*(Ax_s)=0$, \item for any $y^*\in F$, $y^*(Ax_t)=0$, \item for any $u\in b$, $y^*_t(Au)=0$, \item $(x_{t|_i})_{i=1}^{|t|}\in T_\infty(A, X, Y, K)^\zeta(b)$.  \end{enumerate}

Moreover, if $X_0$, $Y_0$ are preduals of $X, Y$, respectively, such that there exists $B:Y_0\to X_0$ so that $A= B^*$, then $(y^*_t)_{t\in \ttt}$ can be taken to lie in $CB_{Y_0}$ rather than $CB_{Y^*}$.  

\label{dualization lemma}
\end{lemma}

\begin{proof}[Proof of Theorem \ref{dualize}]  If $A\notin \mathfrak{NP}_\infty^\xi(X,Y)$, there exists $K\geqslant 1$ so that  $\textbf{NP}_\infty(A, X, Y, K)>\omega^\xi$.  Then with $\zeta=0$, $n=1$, and $b=\varnothing$, we deduce the existence of a $B$-tree $\ttt$ with $o(\ttt)=\omega^\xi$ and vectors $(x_t)_{t\in \ttt}\subset B_X$ and $(y_t^*)_{t\in \ttt}\subset KB_{Y^*}$ so that for any $t\in \ttt$, $y_t^*(Ax_t)=1$, for $s\in \ttt$ comparable to $t$ and not equal to $t$, $y_s^*(Ax_t)=y_t^*(Ax_s)=0$, and so that $(x_{t|_i})_{i=1}^{|t|}$ is $1$-dominated by the $\ell_\infty^{|t|}$ basis.  Using $\infty$-absolute convex combinations of branches of $(x_{t|_i})_{i=1}^{|t|}$ to appropriately norm linear combinations of branches of $(A^*y^*_{t|_i})_{i=1}^{|t|}$ shows that the sequence $(A^*y^*_{t|_i})_{i=1}^{|t|}$ $1$-dominates the $\ell_1^{|t|}$ basis.  Thus $(K^{-1}y^*_t)_{t\in \ttt}\subset B_{Y^*}$ gives that $\textbf{NP}_1(A^*, Y^*, X^*)>\omega^\xi$, and $A^*\notin \mathfrak{N}_1^\xi(Y^*, X^*)$, which proves $(ii)$.  The proof of $(i)$ is similar, using the ``moreover'' statement of Lemma \ref{dualization lemma}.

\end{proof}

\section{Direct sums and $p$-convexifications}

In this section, we wish to discuss how combining operators behaves under finite and infinite direct sums, as well as under $p$-convexifications.

\subsection{Local indices}
 
Our first result is analogous to a result concerning the Bourgain $\ell_p$ block index. The non-operator version of the analogous result was first shown for $p=1$ in \cite{JO}, and for $1\leqslant p\leqslant \infty$ in \cite{C2}.    

\begin{proposition} Suppose $X, Y$ are Banach spaces having $1$-unconditional bases $(e_i)_{i\in I}$, $(f_j)_{j\in J}$, respectively, and $A\in \mathfrak{L}(X,Y)$.  Then for $0<\xi\in \ord$ and $K\geqslant 1$, if $\emph{\textbf{NP}}_p(A,X,Y,K)>\omega \xi$, for any $\ee>0$ and $\ee_n\downarrow 0$ there exist and $(x_t)_{t\in \ttt_\xi}$ and $(y_t)_{t\in \ttt_\xi}$ so that for each $t\in \ttt_\xi$, \begin{enumerate}[(i)]\item  $(x_{t|_i})_{i=1}^{|t|}\in T_p(A,X,Y,K+\ee)$, \item $\|Ax_t-y_t\|<\ee_{|t|}$,  \item $(x_{t|_i})_{i=1}^{|t|}$ have finite, disjoint supports with respect to $(e_i)_{i\in I}$. \item $(y_{t|_i})_{i=1}^{|t|}$ have finite, disjoint supports with respect to $(y_j)_{j\in J}$.     

\end{enumerate}

\label{block index}

\end{proposition}

Note that we do not need the bases to be $1$-unconditional.  It is simply a matter of improving the presentation of the proof.

\begin{proof}  For $M\subset I$, let $P^E_M$ denote the projection onto $[e_i:i\in M]$ in $X$, and similarly for $N\subset J$.   Fix $\delta_n\downarrow 0$ so that for each $n\in \nn$, $\sum_{m=n}^\infty \delta_m<\ee_n$.  Choose $(u_t)_{t\in \ttt_{\omega\xi}}$ so that $(u_{t|_i})_{i=1}^{|t|}\in T_p(A,X,Y,K)$ for each $t\in \ttt_{\omega\xi}$.  By replacing $K$ with any strictly larger number not exceeding $K+\ee$ and perturbing, we may assume that for each $t\in \ttt_{\omega\xi}$, $\supp(u_t)$ is finite.  For each $t\in \ttt_{\omega\xi}$, choose a finite set $N_t'\subset J$ so that $$\|Au_t - P^F_{N'_t}Au_t\|\leqslant \delta_{|t|}.$$  Let $M_\varnothing = N_\varnothing = \varnothing$ and, for each $t\in \ttt_{\omega\xi}$, let $M_t=\cup_{s\preceq t}\supp(u_s),$ and $N_t=\cup_{s\preceq t}N_s'.$  We apply Lemma \ref{dichotomy} with the functions $(f_t)_{t\in \mt_{\omega^\xi}}$ defined for a chain $S$ of $\ttt_{\omega\xi}$ by $f_t(S)=1$ if $$0<\min\Bigl\{\|P^E_{M_t}u\|+ \|P^F_{N_t}Au\|: u\text{\ is a\ }p\text{-absolutely convex combination of\ }(u_s)_{s\in S}\Bigr\},$$ and $f_t(S)=0$ otherwise.  By a dimension argument, if $g:\ttt_\omega\to \ttt_{\omega\xi}$ is monotone and $t\in \mt_{\omega\xi}$, there exists a chain $S$ in $\ttt_\omega$ so that $f_t(\{g(s):s\in S\})=0$.  Therefore Lemma \ref{dichotomy} implies that there exists a block map $h$ from $\mt_\xi$ to the chains of $\ttt_{\omega\xi}$ so that for each $s,t\in \mt_\xi$ with $s\prec t$, and for each $s'\in h(s)$, $f_{s'}(h(t))=0$.  

If $t\in \ttt_\xi$ is minimal in $\ttt_\xi$, let $x_t=u_s$ for some $s\in h(t)$.  If $t\in \ttt_\xi$ is not minimal in $\ttt_\xi$, let $s$ be the immediate predecessor of $t$ in $\ttt_\xi$ and let $s'=\max h(s)$.  Since $f_{s'}(h(t))=0$, there exists a $p$-absolutely convex combination $x_t=\sum_{t'\in h(t)} a_{t'}u_{t'}$ of $(u_{t'})_{t'\in h(t)}$ so that $P^E_{M_{s'}}x_t=0$ and $P^F_{N_{s'}}Ax_t=0$.  Let $y_t=P_{N_{\max h(t)}}Ax_t$.  Note that with $(x_v)_{v\in \ttt_\xi}$, $(y_v)_{v\in \ttt_\xi}$ defined in this way, for $t'\prec t\in\ttt_\xi$, $s$ still denoting the immediate predecessor of $t$ in $\ttt_\xi$ and $s'$ still denoting $\max h(s)$, $$\supp(x_{t'})\subset M_{\max h(t')} \subset M_{s'}, \hspace{5mm}\supp(x_t)\cap M_{s'}=\varnothing$$ and $$\supp(y_{t'})\subset N_{\max h(t')} \subset N_{s'}, \hspace{5mm} \supp(y_t)\cap N_{s'}=\varnothing,$$ whence $(iii)$ and $(iv)$ follow.  

Item $(i)$ follows from the fact that $(x_{t|_i})_{i=1}^{|t|}$ is a $p$-absolutely convex block of a branch of $(u_s)_{s\in \ttt_{\omega\xi}}$.  For $(ii)$, recall that for any $t'\in h(t)$, $|t|\leqslant |t'|$, so \begin{align*} \|Ax_t -y_t\| &\leqslant \sum_{t'\in h(t)} \|Au_{t'} -  P^F_{N_{\max h(t)}} Au_{t'}\||a_{t'}| \\ & \leqslant \sum_{t'\in h(t)} \|Au_{t'} - P^F_{N_{t'}}Au_{t'}\| \leqslant \sum_{t'\in h(t)} \delta_{|t'|} \\ & \leqslant \sum_{n=|t|}^\infty \delta_n <\ee_{|t|},\end{align*} where as above, $x_t=\sum_{t'\in h(t)} a_{t'}u_{t'}$.  Here we have used $1$-unconditionality of $(f_j)$ and the fact that $N_{t'}\subset N_{\max h(t)}$ for each $t'\in h(t)$.

\end{proof}

\begin{remark} It is easy to see that if we assume that either only $X$ or only $Y$ has an unconditional basis, we can omit either $(iii)$ or $(iv)$ and obtain the conclusion.  

Considering the identity operator on $c_0$, we deduce that the factor of $\omega$ in the preceding proof is sharp. 

Moreover, it is easy to see how to modify the proof to work for other coordinate systems such as a Schauder or Markushevich basis, and that if the coordinate system is sequentially ordered, the supports of the branches of $(x_t),(y_t)$ can be made successive rather than simply disjoint.  

\end{remark}

Again, minor modificaitons give the analogous result for the strictly singular index.  

\begin{proposition} Suppose $X, Y$ are Banach spaces having $1$-unconditional bases $(e_i)_{i\in I}$, $(f_j)_{j\in J}$, respectively, and $A\in \mathfrak{L}(X,Y)$.  Then for $0<\xi\in \ord$ and $K\geqslant 1$, if $\emph{\textbf{SS}}(A,X,Y,K)>\omega \xi$, for any $\ee>0$ and $\ee_n\downarrow 0$ there exist and $(x_t)_{t\in \ttt_\xi}$ and $(y_t)_{t\in \ttt_\xi}$ so that for each $t\in \ttt_\xi$, \begin{enumerate}[(i)]\item  $(x_{t|_i})_{i=1}^{|t|}\in SS(A,X,Y,K+\ee)$, \item $\|Ax_t-y_t\|<\ee_{|t|}$,  \item $(x_{t|_i})_{i=1}^{|t|}$ have finite, disjoint supports in $(e_i)_{i\in I}$. \item $(y_{t|_i})_{i=1}^{|t|}$ have finite, disjoint supports in $(f_j)_{j\in J}$.     

\end{enumerate}

\label{strictly singular block index}

\end{proposition}

The proof follows from replacing $p$-convex blocks with normalized blocks and replacing $\delta_n$ with $\delta_n/2K$.  The reason for the latter modification is because the cofficients of a $p$-absolutely convex block must have moduli bounded by $1$, whereas the moduli of the coefficients of a normalized block of a $K$-basic sequence need only be bounded by $2K$.  

\begin{corollary} For any set $\Gamma$, any $1\leqslant p<\infty$, any Banach space $X$, and any operators $A:X\to \ell_p(\Gamma)$, $B:\ell_p(\Gamma)\to X$, $$\emph{\textbf{SS}}(A,X,\ell_p(\Gamma))\leqslant \omega \emph{\textbf{NP}}_p(A, X, \ell_p(\Gamma)) $$ and $$\emph{\textbf{SS}}(B, \ell_p(\Gamma), X)\leqslant \omega \emph{\textbf{NP}}_p(B,\ell_p(\Gamma),X),$$ and the same is true for $c_0(\Gamma)$ when $p=\infty$.    

\label{mixing ss and ellp}

\end{corollary}

\begin{proof} For the statement concerning $A$, we can use Proposition \ref{strictly singular block index} to obtain a tree $(x_t)_{t\in \ttt_\xi}$ (without the assumptions on disjointness of supports of the branches of $(x_t)_{t\in \ttt_\xi}$) so that the branches of this tree are uniformly equivalent to their images under $A$ and so that the images under $A$ are a small perturbation of disjointly supported vectors in $\ell_p(\Gamma)$ (resp. $c_0(\Gamma)$).  Thus this tree witnesses the fact that $\textbf{NP}_p(A,X,Y)>\xi$.  For the statement concerning $B$, we omit the portion of Proposition \ref{strictly singular block index} concerning $(y_t)_{t\in \ttt_\xi}$ to obtain a tree the branches of which are disjointly supported in $\ell_p(\Gamma)$ with branches uniformly equivalent to their images, which is necessarily witnesses the fact that $\textbf{NP}_p(A,X,Y)>\xi$.

\end{proof}

We note the analogue of this for the sequential indices.  

\begin{proposition} Let $X$ be a Banach space. \begin{enumerate}[(i)]\item For any $\Gamma$, $1\leqslant p<\infty$, $0<\xi<\omega_1$, and operators $A: \ell_p(\Gamma)\to X$ and $B:X\to \ell_p(\Gamma)$, $A\in \sss\sss_\xi( \ell_p(\Gamma),X)$ if and only if $A\in \mathfrak{SM}_p^\xi(\ell_p(\Gamma),X)$ and $B\in \sss\sss_\xi(X,\ell_p(\Gamma))$ if and only if $B\in \mathfrak{SM}_p^\xi(X, \ell_p(\Gamma))$.  The analogous results hold for $c_0(\Gamma)$.   \item If $Y$ is any Banach space and $A:X\to Y$, $0<\xi<\omega_1$ are such that $A\notin \mathfrak{SM}^\xi_p(X,Y)$, then $A\notin \sss\sss_\xi(X,Y)$.  \end{enumerate} 

\label{strictly singular thing}
\end{proposition}

\begin{proof} $(i)$ If $(x_n)\subset \ell_p(\Gamma)$ (resp. $c_0(\Gamma)$) is normalized, $K$-basic, and $(x_n)_{n\in E}$ is $K$-equivalent to $(Ax_n)_{n\in E}$ for every $E\in \sss_\xi$, then we may assume $(x_n)$ is coordinate-wise convergent, and by passing to an appropriate difference sequence and normalizing, we may assume $(x_n)$ is coordinate-wise null.  By passing to a further subsequence and perturbing, we may assume $(x_n)$ is disjointly supported, and therefore $(x_n)$ and $(Ax_n)$ are both $\ell_p^\xi$ (resp. $c_0^\xi$) spreading models.  Thus if $A\in \mathfrak{SM}_p^\xi$, $A\in \sss\sss_\xi$.  For the analogous statement concerning $B$, the argument is similar, except we assume $(Bx_n)$ is essentially disjointly supported.  

The other direction of $(i)$ is a consequence of $(ii)$. 

$(ii)$  If $(x_n)$ is a  $K$-$\ell_p^\xi$ (resp. $K$-$c_0^\xi$) spreading model and so is its image under $A$ for some $K>1$, we may assume $(x_n)$ is $K$-basic and $(x_n)_{n\in E}$ and $(Bx_n)_{n\in E}$ are $K^2$-equivalent for each $E\in \sss_\xi$.  

\end{proof}

\begin{proposition} Suppose that for $i=1,, \ldots, k$, $X_i, Y_i\in \Ban$ are such that $Y_i$ has an unconditional basis.  Then if $A_i\in \mathfrak{L}(X_i, Y_i)$, $$\emph{\textbf{NP}}_1(\oplus_{i=1}^k A_i,\oplus_{i=1}^k X_i, \oplus_{i=1}^k Y_i)\leqslant \omega  \bigvee_{i=1}^k \emph{\textbf{NP}}_1(A_i, X_i, Y_i).$$

\label{finite direct sum}

\end{proposition}

\begin{proof} We may assume that each $Y_i$ has a $1$-unconditional basis and that the direct sums are $1$-sums.  We may also assume that at least one of the  operators $A_i$ is not finite rank and that none of the $A_i$ preserves a copy of $\ell_1$.   Suppose $0<\xi\in \ord$ is such that $\omega^\xi= \vee_{i=1}^k\textbf{NP}_1(A_i, X_i, Y_i).$  To obtain a contradiction, assume $K\geqslant 1$ is such that $$\textbf{NP}_1(\oplus_{i=1}^k A_i, \oplus_{i=1}^k X_i,\oplus_{i=1}^k Y_i, K/2)>\omega \omega^\xi.$$  By Proposition \ref{block index} applied with $\ee=K/2$, we can find $((x_{j,t})_{j=1}^k)_{t\in \ttt_{\omega^\xi}}$, $((y_{j,t})_{j=1}^k)_{t\in \ttt_{\omega^\xi}}$ satisfying $(i)$-$(iv)$ with $\ee_n=1/4kK$ for each $n\in \nn$.  Then since $((A_jx_{j,t|_i})_{j=1}^k)_{i=1}^{|t|}$ $K$-dominates the $\ell_1$ basis for each $t\in \ttt_{\omega^\xi}$, $((y_{j, t|_i})_{j=1}^k)_{i=1}^{|t|}$ must $2K$-dominate the $\ell_1$ basis for each $t\in \ttt_{\omega^\xi}$.  But since this is a disjointly supported sequence in a space with $1$-unconditional basis, this is simply equivalent to every convex combination of $((y_{j, t|_i})_{j=1}^k)$ having norm at least $1/2K$.  By the geometric version of the Hahn-Banach theorem, this is equivalent to the existence of a functional $(y^*_{j,t})_{j=1}^k\in \prod_{j=1}^k B_{Y^*_j}$ so that for each $1\leqslant i\leqslant |t|$, $$\sum_{j=1}^k y^*_{j,t}(y_{j, t|_i})\geqslant 1/2K.$$  Of course this means that for each $1\leqslant i\leqslant |t|$, $\vee_{j=1}^k y^*_{j,t}(y_{j, t|_i})\geqslant 1/2kK$.  For each $s\in \ttt_{\omega^\xi}$ and $j\in \{1, \ldots k\}$, let $$A_j(t)=\{s\in MAX(\ttt_{\omega^\xi}): t\preceq s,y^*_{j,s}(y_{j,t})\geqslant 1/2kK\}.$$  Then our previous remark guarantees that for each $t\in \ttt_{\omega^\xi}$, $$\cup_{j=1}^k A_j(t)=\{s\in MAX(\ttt_{\omega^\xi}): t\preceq s\}.$$  Then \cite{C2}[Lemma $3.7$] gives the existence of $j\in \{1,\ldots, k\}$ and maps $g:\ttt_{\omega^\xi}\to \ttt_{\omega^\xi}$ and $h:MAX(\ttt_{\omega^\xi})\to MAX(\ttt_{\omega^\xi})$ so that for each $s,t\in \ttt_{\omega^\xi}$ with $s\prec t$, $g(s)\prec g(t)$ and for each $t\in MAX(\ttt_{\omega^\xi})$, $y^*_{j, h(t)}(y_{j,g(t|_i)})\geqslant 1/2kK$ for each $1\leqslant i\leqslant |t|$.  Thus for each $t\in MAX(\ttt_{\omega^\xi})$, $(y_{j,g(t|_i)})_{i=1}^{|t|}$ is a disjointly supported sequence in $Y_j$ and $y^*_{j, h(t)}(y_{j, g(t|_i)})\geqslant 1/2kK$ for each $1\leqslant i\leqslant |t|$ witnesses the fact that $(y_{j, g(t|_i)})_{i=1}^{|t|}$, and therefore every branch of $(y_{j,g(s)})_{s\in \ttt_{\omega^\xi}}$, $2kK$-dominates the $\ell_1$ basis.  Since $\|A_j x_{j, g(t)}- y_{j,g(t)}\|\leqslant 1/4kK$ for every $t\in \ttt_{\omega^\xi}$, we deduce that every branch of $(A_jx_{j, g(t)})_{t\in \ttt_{\omega^\xi}}$ $4kK$-dominates the $\ell_1$ basis.  But the existence of $(x_{j, g(t)})_{t\in \ttt_{\omega^\xi}}$ implies that $$\textbf{NP}_1(A_j, X_j, Y_j, 4kK)>\omega^\xi,$$ a contradiction.  

\end{proof}

\begin{corollary} If $X_1, \ldots, X_k$ have unconditional bases, $\textbf{I}_1(\oplus_{i=1}^k X_i)\leqslant \omega \vee_{i=1}^k \textbf{I}_1(X_i).$  Moreover, the $\ell_1$ block index of the natural basis of the direct sum is exactly the maximum of the $\ell_1$ block indices of the individual spaces.

\end{corollary}

Next, recall that if $(e_i)_{i\in I}$ is a $1$-unconditional basis for the Banach space $E$ and if $(U_i)_{i\in I}$ is a collection of Banach spaces, $$ \bigl(\oplus U_i\bigr)_E = \Bigl\{(u_i)_{i\in I}: \sum_{i\in I}\|u_i\|e_i\in E\Bigr\}$$ is a Banach space when endowed with the norm $$\|(u_i)_{i\in I}\|=\|\sum_{i\in I}\|u_i\|e_i\|.$$  For convenience, we will denote $\bigl(\oplus U_i)_E$ by $U_E$.  For each $J\subset I$, we let $P^E_J$ be the projection in $U_E$ defined by $P^E_J(u_i)_{i\in I}=(1_J(i)u_i)$.  We let $\supp_U((u_i)_{i\in I})=\{i\in I: u_i\neq 0\}$.  

Suppose that we have two $1$-unconditional bases $(e_i)_{i\in I}, (f_i)_{i\in I}$ for $E, F$, respectively, indexed by the same set $I$.  Suppose also that we have a collection $(U_i, V_i)_{i\in I}$ of Banach spaces and $(A_i)_{i\in I}$ of operators $A_i:U_i\to V_i$ so that the the map $e_i\mapsto \|A_i\|f_i$ extends linearly to some $I_{E,F}\in \mathfrak{L}(E,F)$.  Then $A(u_i)_{i\in I}:= (A_iu_i)_{i\in I}$ defines a bounded operator from $U_E$ to $V_F$.  For each $J\subset I$, we let $A_J=P^F_JA$.  That is, $A_J(u_i)_{i\in I}=(1_J(i)A_iu_i)_{i\in I}$.

\begin{proposition} With $U_E, V_F$, $I_{E,F}$, $A$, and $A_J$ as above, $$\emph{\textbf{NP}}_1(A, U_E, V_F) \leqslant \Bigl(\sup\Bigl\{ \emph{\textbf{NP}}_1(A_J, U_E, V_F): J\subset I, |J|<\infty\Bigr\}\Bigr) \emph{\textbf{NP}}_1(I_{E,F}, E, F).$$

\label{infinite direct sum}
\end{proposition}

\begin{proof} The proof is similar to the proof of Proposition \ref{block index}, so we omit some details.  Let $\xi=\sup \bigl\{\textbf{NP}_1(A_J, U_E, V_F): J\subset I, |J|<\infty\}$ and let $\zeta=\textbf{NP}_1(I_{E,F}, E, F)$.  If either $\xi=\infty$ or $\zeta=\infty$, there is nothing to show, so suppose $\xi, \zeta\in \ord$.  Suppose also that there exists $K\geqslant 1$ so that $\textbf{NP}_1(A, U_E, V_F, K)>\xi\zeta$ and choose $(u_t)_{t\in \ttt_{\xi\zeta}}$ so that for each $t\in \ttt_{\xi\zeta}$, $(u_{t|_i})_{i=1}^{|t|}\in T_1(A, U_E, V_F,K)$.  For each $t\in \ttt_{\omega\xi}$, choose $N_t'$ finite so that $\|Au_t - P^F_{N'_t} Au_t\|<1/3K$.  Let $N_\varnothing=\varnothing$ and $N_t=\cup_{s\preceq t}N_s'$.  For $t\in \mt_{\xi\zeta}$ and $S\subset \ttt_{\xi\zeta}$ a chain, let $f_t(S)=1$ if $$1/3K\leqslant \min \Bigl\{\|P^F_{N_t}Ax\|:x\text{\ is a\ }1\text{-absolutely convex combination of\ }(u_s)_{s\in S}\Bigr\},$$ and $f_t(S)=0$ otherwise.   Note that there cannot exist $t\in \mt_{\xi\zeta}$ and a monotone $g:\ttt_\xi\to \ttt_{\xi\zeta}$ so that for each chain $S$ of $\ttt_\xi$, $f_t(\{g(s):s\in S\})=1,$ otherwise $(P^E_{N_t}x_{g(s)})_{t\in \ttt_\xi}$ witnesses the fact that $\textbf{NP}_1(A_{N_t}, U_E, V_F, 3K)>\xi$.  This is because $AP^E_N=P^F_NA=A_N$ for any $N\subset I$.  Therefore Lemma \ref{dichotomy} guarantees the existence of a block map $h$ mapping $\ttt_\zeta$ to the chains of $\ttt_{\xi\zeta}$ so that for each $s,t\in \ttt_\zeta$ with $s\prec t$, and for each $s'\in h(s)$, $f_{s'}(h(t))=0$.  As in the proof of Proposition \ref{block index}, we can find $(x_t)_{t\in \ttt_\zeta}$ each branch of which is a $1$-absolutely convex block of a branch of $(u_t)_{t\in \ttt_{\xi\zeta}}$ and so that for each $s,t\in \ttt_\zeta$ with $s\prec t$, $\|P^F_{N_{\max h(s)}}A x_t\|<1/3K$.  Observe that if $x_t=\sum_{t'\in h(t)} a_{t'}u_{t'}$, since $N_{t'}\subset N_{\max h(t)}$ and since $$\|Au_{t'}- P^F_{N_{\max h(t)}}Au_{t'} \| \leqslant \|Au_{t'}- P^F_{N_{t'}}Au_{t'}\|< 1/3K$$ for each $t'\in h(t)$, $$\|Ax_t- P^F_{N_{\max h(t)}} Ax_t\|\leqslant \sum_{t'\in h(t)}|a_{t'}|\|Au_{t'}- P^F_{N_{\max h(t)}}Au_{t'}\|<1/3K.$$ 
If $t$ is minimal in $\ttt_\zeta$, let $z_t=P^E_{N_{\max h(t)}} x_t$ and $y_t= P^F_{N_{\max h(t)}}Ax_t$.  If $\ttt_\zeta$ is not minimal, let $s$ denote the immediate predecessor of $t$ in $\ttt_\zeta$ and let $z_t= P^E_{N_{\max h(t)}\setminus N_{\max h(s)}}$ and $y_t=P^F_{N_{\max h(t)}\setminus N_{\max h(s)}} Ax_t$.  Note that $\|z_t\|\leqslant 1$ and $Az_t=y_t$.  Note also that $\|Ax_t-y_t\|\leqslant 2/3K$, so that $(y_{t|_i})_{i=1}^{|t|}$ $3K$-dominates the $\ell_1$ basis, since $(Ax_{t|_i})_{i=1}^{|t|}$ $K$-dominates the $\ell_1$ basis.  Moreover, $(z_{t|_i})_{i=1}^{|t|}$ (resp. $(y_{t|_i})_{i=1}^{|t|}$) have pairwise disjoint supports in $U_F$ (resp. $V_F$).

 Let $\Pi_E:U_E\to E$ denote the map $\Pi_E\bigl((u_i)_{i\in I}\bigr)= \sum_{i\in I} \|u_i\|e_i$ and let $\Pi_F:V_F\to F$ denote $\Pi_F\bigl((v_i)_{i\in I}\bigr)=\sum_{i\in I}\|v_i\|f_i$. Since $(z_{t|_i})_{i=1}^{|t|}$ have pairwise disjoint supports in $U_E$, this sequence is isometrically equivalent to $(\Pi_E(z_{t|i}))_{i=1}^{|t|}$, and the same holds for $(Az_{t|i})_{i=1}^{|t|}=(y_{t|_i})_{i=1}^{|t|}$ and $(\Pi_F(Az_{t|_i}))_{i=1}^{|t|}$. Therefore we deduce that $(\Pi_F y_{t|_i})_{i=1}^{|t|}$ $3K$-dominates the $\ell_1$ basis.     But if we write $z_t=(z_t(i))_{i\in I}$,  $$I_{E,F}\Pi_E z_t= I_{E,F}\sum_{i\in I}\|z_t(i)\|e_i = \sum_{i\in I}\|A_i\|\|z_t(i)\|f_i$$ and $$\Pi_FAz_t = \Pi_F y_t = \sum_{i\in I}\|A_iz_t(i)\|f_i.$$ Note that $\|A_i\|\|z_t(i)\|\geqslant \|A_iz_t(i)\|$ for all $i\in I$, so $I_{E,F}\Pi_E z_t$ dominates $\Pi_F y_t$ coordinate-wise.  Thus since $(I_{E,F}\Pi_Ez_{t|_i})_{i=1}^{|t|}$ is a disjointly supported sequence in $F$ which coordinate-wise dominates the disjointly supported sequence $(\Pi_F y_{t|_i})_{i=1}^{|t|}$, we deduce that $$(\Pi_F y_{t|i})_{i=1}^{|t|}\lesssim_1 (I_{E,F}\Pi_E z_{t|_i})_{i=1}^{|t|},$$ and $(I_{E,F}\Pi_E z_{t|_i})_{i=1}^{|t|}$ $3K$-dominates the $\ell_1$ basis.  Since $\|\Pi_E z_t\|\leqslant 1$ for each $t\in \ttt_\zeta$, $(\Pi_Ez_t)_{t\in \ttt_\zeta}$ implies that $\textbf{NP}_1(I_{E,F},E,F, 3K)>\zeta$, a contradiction.

\end{proof}

\begin{remark} We note that actually we have proved something slightly stronger than the claim.  Rather than using the value of $\textbf{NP}_1(I_{E,F}, E, F)$, we can use the value $$\sup_{K\geqslant 1} o\Bigl(\bigl\{(x_i)_{i=1}^n \in B_E: (x_i)_{i=1}^n\text{\ have disjoint supports}, (I_{E,F}x_i)_{i=1}^n K\text{-dominate the\ }\ell_1\text{\ basis}\bigr\}\Bigr).$$  The fact that the $(x_i)$ can be taken to have disjoint supports in $E$ follows from the proof.

\end{remark}

\begin{corollary} With $U_E$, $V_F$, $I_{E,F}$, $A$, and $A_J$ as in Proposition \ref{infinite direct sum}. \begin{enumerate}[(i)]\item If $\xi\in \ord$ is such that $A_i\in \mathfrak{NP}_1^{\omega^\xi}(U_i, V_i)$ for each $i\in I$, then $$\emph{\textbf{NP}}_1(A, U_E, V_F)\leqslant \omega^{\omega^\xi} \emph{\textbf{NP}}_1(I_{E,F}, E, F).$$ \item If every $V_i$ has a $1$-unconditional basis and if $A_i\in \mathfrak{NP}_1^\xi(U_i, V_i)$ for each $i\in I$, then $$\emph{\textbf{NP}}_1(A, U_E, V_F)\leqslant \omega\xi \emph{\textbf{NP}}_1(I_{E,F}, E, F).$$

\end{enumerate}

\end{corollary}

\begin{proof} Item $(i)$ follows from Proposition \ref{infinite direct sum} together with the fact that $\mathfrak{NP}_1^{\omega^\xi}$ is an ideal, and so $\textbf{NP}_1(A_J, U_E, V_F)\leqslant \omega^{\omega^\xi}$ for each finite $J$.  

$(ii)$ This follows from Proposition \ref{infinite direct sum} and Proposition \ref{block index}, which gives that $\textbf{NP}_1(A_J, U_E, V_F)\leqslant \omega \xi$ for each finite $J\subset I$.  

\end{proof}

\begin{proposition} Suppose $X$, $Y$ have $1$-unconditional bases $(e_i)_{i\in I}$, $(f_i)_{j\in J}$, respectively, and $A\in \mathfrak{L}(X,Y)$ is such that for each distinct members $e_{i_1}, e_{i_2}$ of the basis of $X$, $Ae_{i_2}$ and $Ae_{i_2}$ have disjoint supports in $Y$.  Then for $1\leqslant t<\infty$, the map $e_i\mapsto \sum_j |f_j^*(Ae_i)|^{1/t}f_j$ extends to an operator $A^t\in \mathfrak{L}(X^t, Y^t)$.  Moreover, for any $1\leqslant p, q<\infty$, $$\emph{\textbf{NP}}_p(A^p, X^p, Y^p)\leqslant \omega \emph{\textbf{NP}}_q(A^q, X^q, Y^q).$$  
\label{convexify}

\end{proposition}

\begin{proof} The first statement is clear.  For $s>0$ and for a vector $x$ in the span of $(e_i)$ (resp. $(f_j)$), let $x^s$ be the vector in the span of $(e_i)$ (resp. $(f_j)$) so that $e_i^*(x^s)= \sgn(e_i^*(x))|e^*(x)|^s$. Fix $1\leqslant q<\infty$ and assume $\textbf{NP}_q(A^q, X^q, Y^q, K)>\omega\xi$ for some $0<\xi$ and $K\geqslant 1$.  Fix $1\leqslant p<\infty$ and $\ee_n\downarrow 0$ so that $\sum \ee_n< 1/2K$ and $\sum \ee_n^{q/p} < 1/2(2K)^{1/p}$.  

Fix $(x_t)_{t\in \ttt_\xi}\subset X^q$ and $(y_t)_{t\in \ttt_\xi}\subset Y^q$ to satisfy $(i)$-$(iv)$ of Proposition \ref{block index}.  Recall from the proof of this proposition that there exist finite sets $N_t$ so that $y_t=P^F_{N_t}A^qx_t$. This means that $y_t^{q/p}=P^F_{N_t}A^px_t^{q/p}$ and $A^px_t^{q/p}- y_t^{q/p}= (A^qx_t- y_t)^{q/p}$, so that $\|A^px_t^{q/p}- y_t^{q/p}\|_{Y^p} < \ee_{|t|}^{q/p}.$  Then our choice of $(\ee_n)$, the disjointness of the supports of each branch of $(y_t)_{t\in \ttt_\xi}$, and the fact that each branch of $(Ax_t)_{t\in \ttt_\xi}$ $K$-dominates the $\ell_q$ basis gives that each branch of $(y_t)_{t\in \ttt_\xi}$ $2K$-dominates the $\ell_q$ basis, each branch of $(y_t^{q/p})_{t\in \ttt_\xi}$ $(2K)^{1/p}$-dominates the $\ell_p$ basis in $X^p$, and $(A^px^{q/p}_t)_{t\in \ttt_\xi}$ $2(2K)^{1/p}$-dominates the $\ell_p$ basis in $Y^p$.  Since the $(x_t)_{t\in \ttt_\xi}$ are disjointly supported and $1$-dominated by the $\ell_q$ basis in $X^q$, each branch of $(x_t^{q/p})_{t\in \ttt_\xi}$ is $1$-dominated by the $\ell_p$ basis in $X^p$.  Therefore $(x_t^{q/p})_{t\in \ttt_\xi}$ witnesses the fact that $\textbf{NP}_p(A^p, X^p, Y^p)>\xi$.  Since $0<\xi$ was arbitrary, we are done.

\end{proof}

\subsection{Sequential indices}

If $A:X\to Y$ is an operator between spaces with unconditional bases, and if $(x_n)\subset B_X$ is any sequence, then by passing to a subsequence we may of course assume that $(x_n)$ and $(Ax_n)$ are both coordinate-wise convergent.  If $(Ax_n)$ is an $\ell_1^\xi$ spreading model, an appropriate difference sequence will also be an $\ell_1^\xi$ spreading model.  This observation means that if $A$ preserves an $\ell_1^\xi$ spreading model, then there is a coordinate-wise null sequence $(x_n)\subset B_X$ so that $(Ax_n)$ is also coordinatewise-null and so that both $(x_n)$ and $(Ax_n)$ are $\ell_1^\xi$ spreading models.    A perturbation argument guarantees that the operator $A^p:X^p\to Y^p$ as defined in Proposition \ref{convexify} preserves an $\ell_p^\xi$ spreading model.  Conversely, if $A^p:X^p\to Y^p$ preserves an $\ell_p^\xi$ spreading model for $1<p$, say $(x_n)$ and $(A^px_n)$ are both $\ell_p^\xi$ spreading models, then both sequences are already weakly null.  Another perturbation argument yields that $A:X\to Y$ preserves an $\ell_1^\xi$ spreading model, and we arrive at the following: 

\begin{proposition} If $X,Y$ have unconditional bases and $A\in \mathfrak{L}(X,Y)$, then for any $0<\xi< \omega_1$ and any $1<p<\infty$,  $A\in \mathfrak{SM}_1^\xi(X,Y)$ if and only if $A^p\in \mathfrak{SM}_p^\xi(X^p,Y^p)$.  

\label{convexify spreading model}

\end{proposition}

Of course, since membership in $\mathfrak{SM}_p^\xi(X,Y)$ is determined by all separable subspaces of $X$, to deduce the analogue of Proposition \ref{infinite direct sum}, we may assume $E,F$ have countable, sequentially ordered unconditional bases.  We obtain the following.  

\begin{proposition} Fix Banach spaces $E, F$ with unconditional bases $(e_n)_{n\in\nn}$, $(f_n)_{n\in \nn}$, respectively, and a sequence $A_n:U_n\to V_n$ of operators so that $e_n\mapsto \|A_n\|f_n$ extends to an operator $I_{E,F}\in \mathfrak{L}(E,F)$.  \begin{enumerate}  [(i)] \item For any $0<\xi, \zeta< \omega_1$, if $A_n\in \mathfrak{SM}_1^\xi(U_n, V_n)$ for all $n\in \nn$, and if $I_{E,F}\in \mathfrak{SM}_1^\zeta(E,F)$, then $A\in \mathfrak{SM}_1^{\xi+\zeta}(U_E,V_F)$.  \item If $0<\zeta<\omega_1$ and $0\leqslant \xi< \omega_1$ are such that $A_n\in \mathfrak{SM}_1^{\xi+1}(U_n, V_n)$ for all $n\in \nn$, $I_{E,F}\in \mathfrak{SM}_1^\zeta(E,F)$, and if $F$ is reflexive, $A\in \mathfrak{SM}^{\xi+\zeta}_1(U_E,V_F)$.  \item If $0\leqslant \zeta<\omega$, $\xi<\omega_1$ is a limit ordinal, and $\eta_n\uparrow \xi$ are such that $A_n\in \mathfrak{SM}_1^{\eta_n}(U_n, V_n)$ for each $n\in \nn$, $I_{E,F}\in \mathfrak{SM}_1^{\zeta+1}$, and if $F$ is reflexive, then $A\in \mathfrak{SM}^{\xi+\zeta}_1(U_E,V_F)$.  

\end{enumerate}

\label{infinite direct sum spreading model}

\end{proposition}

\begin{proof}$(i)$ We know that $\mathfrak{SM}_1^\xi$ is an ideal, so for each $n\in \nn$, $A_{[1,n]}:\oplus_{i=1}^n U_i\to \oplus_{i=1}^n V_i$ can preserve no $\ell_1^\xi$ spreading model.  Assume $(x_i)\subset B_{U_E}$ is such that $(Ax_i)$ is an $\ell_1^{\xi+\zeta}$ spreading model.  Assume $K\|\sum_{i\in G} a_iAx_i\|_{U_E}\geqslant \sum_{i\in G}|a_i|$ for all $G\in \sss_{\xi+\zeta}$.  By replacing $K$ with any strictly larger value, we may assume $\supp_E(x_i)\subset [1, s_i]$ for some $s_i\in \nn$.  Choose $M\in \infin$ so that $\sss_\zeta[\sss_\xi](M)\subset \sss_{\xi+\zeta}$.  Since for any $n\in \nn$, no subsequence of $(A_{[1,n]}x_i)_{i\in M}$ can be an $\ell_1^\xi$ spreading model, we can choose $G_1<G_2<\ldots$, $G_i\in \sss_\xi$, and a $1$-absolutely convex block $(y_i)$ of $(x_i)$ so that $y_i=\sum_{j\in G_i}a_j x_{m_j}$ and so that with $t_0=0$ and $t_i= s_{m_{\max G_i}}$ for $i\in \nn$, $\|A_{[1, t_{i-1}]}y_i\|< 1/2K$ for all $i\in \nn$.  Then with $\Pi_E:U_E\to E$ and $\Pi_F:V_F\to F$ as in Proposition \ref{infinite direct sum}, we deduce that $(P^E_{(t_{i-1}, t_i]}\Pi_E y_i)$ and its image under $I_{E,F}$, which pointwise dominates $(P^F_{(t_{i-1}, t_i]}\Pi_F Ay_i)$,  are both $\ell_1^\zeta$ spreading models.  This is a contradiction and finishes $(i)$.

$(ii)$  Assume $(x_i)\subset B_{U_E}$, $s_i\in \nn$, and $K\geqslant 1$ are as in $(i)$.  Let $v_i=\Pi_F Ax_i$ and assume $v_i\underset{w}{\to} v\in F$.  Choose $l_0\in \nn$ so that $\|v- P^F_{[1,l_0]}v\|<1/3K$.  By passing to a further subsequence, we may assume there exist $l_0<l_1<l_2<\ldots$ so that with $I_0=[1,l_0]$ and $I_i=(l_{i-1}, l_i]$, $\|v_i-P^F_{I_0\cup I_i}\|<1/3K$.  Then noting that $A_{[1,l_0]}$ does not preserve an $\ell_1^{\xi+1}$ spreading model, by the claim following this proof, we can choose $k\in \nn$ and a subsequence $(x_i)_{i\in N}$ so that for any $M\in [N]$, there exist $G_1<G_2<\ldots$, $G_i\in \aaa_k[\sss_\xi]$, and a $1$-absolutely convex block $(y_i)$ of $(x_i)_{i\in M}$ so that $y_i=\sum_{j\in G_i}a_jx_{m_j}$ and $\|A_{[1,l_0]}y_i\|<1/3K$ for all $i\in \nn$.  In particular, we can choose $M\in [N]$ so that $\sss_\zeta[\aaa_k[\sss_\xi]](M)\subset \sss_{\xi+\zeta}$.  If $J_i=\cup_{j\in G_i}I_j$, $\| A y_i- P^F_{J_i} Ay_i\|< 2/3K.$  Reasoning as in $(i)$, this means that $(P^E_{J_i}y_i)$ and its image $(P^F_{J_i}Ay_i)$ under $A$ are both $\ell_1^\zeta$ spreading models, and the same holds for $(\Pi_E P^E_{J_i} y_i)$ and $(I_{E,F} \Pi_E P^E_{J_i}y_i)$, a contradiction.

$(iii)$ This is similar to $(ii)$.  With $l_0$ as in $(ii)$, we can take the $E_i$ used in the blocking $(y_i)$ to lie in $\sss_{\eta_{l_0}}$ and choose $M\in \infin$ so that $\sss_{\zeta+1}[\sss_{\eta_{l_0}}](M)\subset \sss_{\xi+\zeta}$ using Proposition \ref{Gasparis analogue}.  This is because the Cantor-Bendixson index of $\sss_{\zeta+1}[\sss_{\eta_{l_0}}]$ is $\omega^{\eta_{l_0}+\zeta+1}+1=\omega^{\eta_{l_0}+1+\zeta}+1<\omega^{\xi+\zeta}+1$, since we have assumed $\zeta$ is finite.

\end{proof}

\begin{claim} Fix $(x_n)\subset B_X$, $\xi<\omega_1$. \begin{enumerate}[(i)]\item If $\xi$ is a limit ordinal and no subsequence of $(x_n)$ is an $\ell_1^\xi$ spreading model, then for any $\ee>0$, there exists $\zeta=\zeta(\ee)<\xi$ and $N\in \infin$ so that for any $M\in [N]$, there exist $E_1<E_2<\ldots$, $E_i\in \sss_\zeta$, and a $1$-absolutely convex block $(y_i)$ of $(x_i)_{i\in M}$ with $y_i=\sum_{j\in E_i}a_jx_{m_j}$ and $\|y_i\|<\ee$.   

\item If no subsequence of $(x_n)$ is an $\ell_1^{\xi+1}$ spreading model, then for any $\ee>0$, there exist $k=k(\ee)\in \nn$ and $N\in \infin$ so that for any $M\in [N]$, there exist $E_1<E_2<\ldots$, $E_i\in \aaa_n[\sss_\xi]$, and a $1$-absolutely convex block $(y_i)$ of $(x_i)_{i\in M}$ with $y_i=\sum_{j\in E_i} a_jx_{m_j}$ and $\|y_i\|<\ee$.  

\end{enumerate}  

\end{claim}

\begin{proof}$(i)$ If it were not so, then there would exist $\ee>0$ so that for any $\zeta<\xi$ and $N\in \infin$, there exists $M\in [N]$ so that for any $E\in \sss_\zeta$ and scalars $(a_i)_{i\in E}$, $\|\sum_{i\in E}a_ix_{m_i}\|\geqslant \ee \sum_{i\in E}|a_i|$.  Let $\xi_k\uparrow \xi$ be the sequence used to define $\sss_\xi$.  Recursively choose $\nn=M_0\supset M_1\supset M_2\supset \ldots$, $M_k\in \infin$, so that with $M_k=(m^k_i)$, for any $E\in \sss_{\xi_k}$ and scalars $(a_i)_{i\in E}$, $\|\sum_{i\in E}a_i x_{m^k_i}\|\geqslant \ee \sum_{i\in E}|a_i|$.  Let $M=(m^k_k)$.  One easily checks that $(x_i)_{i\in M}$ is an $\ell_1^\xi$ spreading model.  

$(ii)$ This is essentially the same as $(i)$ with $\sss_{\xi_k}$ replaced by $\aaa_k[\sss_\xi]$, since $\sss_{\xi+1}=\{E\in [\nn]^{<\nn}: \exists k\leqslant E\in \aaa_k[\sss_\xi]\}$.

\end{proof}

\begin{corollary} For any $0<\xi<\omega_1$, if $I$ is any set and $(W_i)_{i\in I}$ is any family of Banach spaces so that $W_i$ does not admit an $\ell_1^\xi$ spreading model, then $(\oplus W_i)_{\ell_2(I)}$ does not admit an $\ell_1^\xi$ spreading model.    

\label{spreading model corollary}

\end{corollary}

\section{Distinction between classes}

The main goal of this section is to fully elucidate the relationship between the different classes of operators defined above in order to motivate the study of the distinct classes.   To that end, we have \begin{theorem} Fix $1\leqslant p\leqslant \infty$, $0<\xi\in \ord$. \begin{enumerate}[(i)]\item  $\mathfrak{NP}^\xi\subset \cup_{\zeta<\xi}\mathfrak{NP}^\zeta$ if and only if $\xi$ has uncountable cofinality.  \item  For $1<\zeta\in \ord$ and $1\leqslant q\leqslant \infty$, $\mathfrak{NP}^\zeta_q\subset \mathfrak{NP}_p^\xi$ if and only if $p=q$ and $\zeta\leqslant \xi$.\item For $1\leqslant q\leqslant \infty$,  $\mathfrak{NP}_q^1 \subset \mathfrak{NP}_p^\xi$ if and only if $p\leqslant q\leqslant 2$ and $1\leqslant \xi$.  
\end{enumerate}

\label{distinct1}
\end{theorem}

\begin{theorem} Fix $1\leqslant p\leqslant \infty$, $0<\xi< \omega_1$.  \begin{enumerate}[(i)]\item $\mathfrak{SM}_p^\xi\not\subset\cup_{0<\zeta<\xi}\mathfrak{SM}_p^\zeta$. \item $\mathfrak{WC}^\xi\not\subset \cup_{0<\zeta<\xi}\mathfrak{WC}^\zeta$.  \item For $\zeta\leqslant \omega_1$ and $1\leqslant q\leqslant \infty$, $\mathfrak{SM}_q^\zeta\subset \mathfrak{SM}_p^\xi$ if and only if $p=q$ and $\zeta\leqslant \xi$.   \end{enumerate}
\label{distinct2}

\end{theorem}

\begin{theorem} Fix $1\leqslant p\leqslant \infty$. \begin{enumerate}[(i)]\item For $0<\xi\leqslant \omega_1$, $\mathfrak{NP}_p^\xi\subsetneq \mathfrak{SM}_p^\xi$.  \item For $0<\xi\in \ord$, $\mathfrak{SM}_p^1\not\subset \mathfrak{NP}_p^\xi$. \end{enumerate}
\label{distinct3}
\end{theorem}

\begin{theorem}\begin{enumerate}[(i)]\item  For any $0<\xi<\omega_1$, $ \mathfrak{SS}^\xi\subset \sss\sss_\xi$.  \item For any $0<\xi\in \ord$, $\sss\sss_1\not\subset \mathfrak{SS}^\xi$. \item For any $0<\xi\in \ord$, $\mathfrak{SS}^\xi \subset \cup_{\zeta<\xi}\mathfrak{SS}^\zeta$ if and only if $\xi$ has uncountable cofinality.  \end{enumerate}

\label{distinct4}

\end{theorem}

In order to accomplish these results, we will provide a full characterization of which ordinals occur as the index of an operator. Every natural number occurs as the index of a finite rank operator, so we will consider only operators which are not finite rank.  Our argument will be similar in some regards to that given in \cite{Brooker}, where a similar result was shown for the Szlenk index.  We will inductively build up a transfinite sequence of spaces $W_\xi$ so that for each $0<\xi\in \ord$, the $\ell_1$ index of the space $W_\xi$ is exactly $\omega^{\xi+1}$.  We can deduce from this that every successor $\xi$ is such that $\omega^\xi$ is the $\ell_1$ index of some operator.  For limit ordinals $\xi$ of countable cofinality, we will take $\xi_n\uparrow \xi$ and take a diagonal operator on $(\oplus_n W_{\xi_n})_{\ell_2}$ to obtain an operator with $\ell_1$ index $\omega^\xi$.  Our argument differs from that of Brooker in that we must employ facts we have shown about how the $\ell_1$ sum behaves under direct sums.  We will also use the facts we have shown about dualization of $c_0$ indices and the behavior of $\ell_p$ indices under $p$-convexifications to simultaneously show that the dual $W_\xi$ of $W_\xi^*$ has $c_0$ index $\omega^{\xi+1}$ (when $0<\xi)$ and the $p$-convexification $W_\xi^p$ of $W_\xi$ has $\ell_p$ index $\omega^{\xi+1}$ (also when $0<\xi$).  As we build the spaces $W_\xi$, we will simultaneously build spaces $V_\xi$ and operators $A_\xi:V_\xi \to W_\xi$ so that the strictly singular index of $A_\xi$ is $\omega^{\xi+1}$.  In building the spaces this way, we will simultaneously exhibit for all successor ordinals $\xi$ operators with $\ell_p$, $c_0$, or strictly singular index equal to $\omega^{\xi+1}$ (the identity on $W_\xi$ for $p=1$, the identity on $W_\xi^p$ for $1<p<\infty$, the identity on $W_\xi^*$ for $p=\infty$, and the operator $A_\xi:V_\xi\to W_\xi$ for the strictly singular index).  We will also obtain, through diagonalizations similar to those in the $p=1$ case mentioned above, diagonal operators on direct sums of sequences of these spaces to obtain operators with $\ell_p$, $c_0$, or strictly singular index $\omega^\xi$ whenever $\xi$ is a limit ordinal of countable cofinality.

Recall that for $X,Y\in \Ban$ and $A\in \mathfrak{NP}_p(X,Y)$, $\textbf{NP}_p(A,X,Y)=\lim_n \textbf{NP}_p(A,X,Y,n)$.  By Proposition \ref{special form}, if $A$ is not finite rank, this supremum is not attained.  Thus if $\omega^\xi= \textbf{NP}_p(A,X,Y)$, $\omega^\xi$ must have countable cofinality, which happens if and only if $\xi$ has countable cofinality. This same restriction applies to the $\textbf{SS}$ index.  This means that the only infinite ordinals which may appear as the $\textbf{NP}_p$ or $\textbf{SS}$ index of an operator are those ordinals of the form $\omega^\xi$, where $\xi$ has countable cofinality.  As stated in the previous paragraph, we will show that for each $1\leqslant p\leqslant \infty$, each such ordinal occurs as the $\textbf{NP}_p$ of some operator, as well as the $\textbf{SS}$ index of some operator.

\begin{theorem} Fix $\xi\in \ord$. \begin{enumerate}[(i)]\item For $1\leqslant p\leqslant \infty$, there exists a Banach space $X$ and an operator $A:X\to X$ so that $\emph{\textbf{NP}}_p(A,X,X)=\omega^\xi$ if and only if $\xi$ has countable cofinality.  Moreover, if $\xi$ is a successor, $A$ can be taken to be the identity on $X$ except in the case that $p=2$ and $\xi=1$, and therefore every ordinal of the form $\omega^{\xi+1}$ occurs as the Bourgain $\ell_p$ index of a Banach space except in the case that $p=2$ and $\xi=0$.   

\item There exist Banach spaces $X,Y$ and an operator $A:X\to Y$ so that $\emph{\textbf{SS}}(A,X,Y)=\omega^\xi$ if and only if $\xi$ has countable cofinality.  Morever, $X$ can be taken to be $\ell_1(\Gamma)$ for some $\Gamma$.  

\item If $0<\xi\leqslant \omega_1$, then for any $1\leqslant p\leqslant \infty$ there exists a Banach space $X$ with $\emph{\textbf{SM}}_p(X)=\xi$. 

\item If $0<\xi\leqslant\omega_1$, then there exists a Banach space $X$ with $\emph{\textbf{SWC}}(X)=\xi$.  \end{enumerate}

\label{characterization}
\end{theorem}

The exception in $(i)$ in the case of $p=2$ and $\xi=1$ is due to Dvoretsky's theorem, which guarantees that $\textbf{I}_2(X)$ is either finite or at least $\omega^2$.

\begin{lemma} For every $0<\xi\in \ord$, there exist Banach spaces $V_\xi, W_\xi\in \Ban$ and $A_\xi\in \mathfrak{L}(V_\xi, W_\xi)$ so that $$\emph{\textbf{I}}_1(W_\xi)=\emph{\textbf{I}}_\infty(W_\xi^*)=\emph{\textbf{I}}_p(W_\xi^p)= \emph{\textbf{SS}}(A_\xi, V_\xi, W_\xi)=\omega^{\xi+1}.$$  Moreover, $W_\xi$ admits no $\ell_1^1$ spreading model, $W_\xi^p$ admits no $\ell_p^1$ spreading model, $W_\xi^*$ admits no $c_0^1$ spreading model, and $A_\xi\in \sss\sss_1(V_\xi, W_\xi)$.

\label{successors}
\end{lemma}

For this we will need the following, which uses the weakly null $\ell_1^+$ characterization of the Szlenk index established in \cite{AJO}.  In the following proposition, $Sz(X)$ denotes the Szlenk index of $X$.    

\begin{proposition}Let $X$ be a Banach space with countable, shrinking, $1$-unconditional basis.  \begin{enumerate}[(i)]\item For any operator $A:\ell_1\to X$, $\emph{\textbf{SS}}(A, \ell_1, X)\leqslant \omega Sz(X).$

\item For $1\leqslant p<\infty$, $\emph{\textbf{I}}_p(X^p)\leqslant \omega Sz(X).$  \end{enumerate}

\label{mixed}

\end{proposition}

\begin{proof}$(i)$ Let $(e_n)_{n\in \nn}$ be a $1$-unconditional basis for $X$.  Let $P_{[1,n]}^X$ denote the basis projections with respect to $(e_n)$ and $P_{[1,n]}^{\ell_1}$ the basis projections in $\ell_1$.  Fix $0<\xi<\omega_1$ and assume $\textbf{SS}(A, \ell_1, X)>\omega\xi$.  Fix $K\geqslant 1$ and $(x_t)_{t\in \ttt_{\omega\xi}}\subset S_{\ell_1}$ so that $(x_{t|_i})_{i=1}^{|t|}\in SS(A,\ell_1, X,K)$ for each $t\in \ttt_{\omega\xi}$. For each $n\in \nn$ and each chain $S$ of $\ttt_{\omega\xi}$, let $$f_n(S)=\min \{\|P^{\ell_1}_{[1,n]}x\|+\|P^X_{[1,n]}Ax\|: x\in [x_t:t\in S], \|x\|=1\}.$$  Note that for any any monotone $g:\ttt_\omega\to \ttt_{\omega\xi}$ and any $n\in \nn$, a dimension argument gives that there exists a segment $S$ of $\ttt_\omega$ so that $f_n(\{x_{g(t)}:t\in S\})=0$.  By Lemma $3.4$ of \cite{C2}, there exists a regular family $\fff$ with Cantor-Bendixson index $\xi+1$ and a tree $(y_E)_{E\in \fff\setminus \{\varnothing\}}$ so that $\|P^{\ell_1}_{[1,\max E]}y_E\|+ \|P^X_{[1,\max E]}Ay_E\|\leqslant 1/\max E$ so that every branch $(y_E)_{E\in \fff\setminus \{\varnothing\}}$ is a normalized block of a branch of $(x_t)_{t\in \ttt_{\omega\xi}}$.  Since $(y_{E\verb!^!n})$ is coordinate-wise null for every $E\in \fff\setminus MAX(\fff)$, we may prune and assume every branch of $(y_E)_{E\in \fff\setminus \{\varnothing\}}$ is $2$-equivalent to the unit vector basis of $\ell_1$.  But since each branch of this tree is in $SS(A, \ell_1, X, K)$, we deduce that each branch of $(Ay_E)_{E\in \fff\setminus \{\varnothing\}}$ $2K$-dominates the $\ell_1$ basis.  But since $(Ay_E)_{E\in \fff\setminus \{\varnothing\}}$ is such that $(y_{E\verb!^!n})$ is coordinate-wise null in $X$ for each $E\in \fff\setminus MAX(\fff)$, and since the basis of $X$ is shrinking, we deduce that $(Ay_E)_{E\in \fff\setminus \{\varnothing\}}$ is a weakly null $\ell_1$ tree.  By \cite{AJO}, $Sz(X)>\xi$.  Since $\xi$ was arbitrary, we are done.

$(ii)$ This is similar to $(i)$.  We assume that for $0<\xi<\omega_1$, $\textbf{I}_p(X^p)>\omega \xi$.  As in $(i)$, we arrive at a tree $(y_E)_{E\in \fff\setminus \{\varnothing\}}$ pointwise null so that each branch is $1$-dominated by and $K$-dominates the $\ell_p$ basis. The only difference is we replace normalized blocks with $p$-absolutely convex blocks.  By pruning, perturbing, and replacing $K$ with any strictly larger value, we may assume that each branch of $(y_E)_{E\in \fff\setminus \{\varnothing\}}$ is a block tree so that $\min \supp(y_{E\verb!^!n})\underset{n\to \infty}{\to} \infty$ for each $E\in \fff\setminus MAX(\fff)$.  Then $(y^p_E)_{E\in \fff\setminus \{\varnothing\}}$ is a weakly null $\ell_1$ tree in $X$, where $y^p$ is defined as in Proposition \ref{convexify}, and we finish again by \cite{AJO}.  
\end{proof}

\begin{proof}[Proof of Lemma \ref{successors}] Let $V_0=W_0$ be the scalar field. Let $A_0:V_0\to W_0$ be the identity.  If $V_\xi, W_\xi$, $A_\xi$ have been defined, let $Y_1=V_\xi$, $Z_1=W_\xi$, $Y_{n+1}=V_\xi\oplus_1 Y_n$, $Z_{n+1}=W_\xi\oplus_1 Z_n$ for $n\in \nn$.  Let $V_{\xi+1}=\bigl(\oplus Y_n\bigr)_{\ell_1}$, $W_{\xi+1}=\bigl(\oplus Z_n\bigr)_{\ell_2}$.  Define $A_{\xi+1}:V_{\xi+1}\to W_{\xi+1}$ by $A_{\xi+1}|_{Y_n}=\oplus_{i=1}^n A_\xi:Y_n\to Z_n$.  If $V_\zeta$, $W_\zeta$, and $Z_\zeta$ have been defined for each $\zeta$ less than the limit ordinal $\xi$, let $$V_\xi=\bigl(\oplus_{\zeta<\xi}V_\zeta\bigr)_{\ell_1([0,\xi))},$$ $$W_\xi=\bigl(\oplus_{\zeta<\xi} W_\zeta\bigr)_{\ell_2([0, \xi))},$$ $$A_\xi|_{V_\zeta}=A_\zeta.$$   It is obvious that $\|A_\xi\|=1$, $V_\xi$ is isometric to an $\ell_1(\Gamma_\xi)$ space for each $\xi$ and some set $\Gamma_\xi$, $V_\xi$ is isometric to $\ell_1$ when $0<\xi<\omega_1$, and that $W_\xi$ is separable when $\xi<\omega_1$.  Moreover, since we know the $\ell_2$ sum of Banach spaces not containing $\ell_1$ also does not contain $\ell_1$, $W_\xi$ fails to contain a copy of $\ell_1$ for each $\xi$, and $A_\xi$ is necessarily strictly singular. Since $W_\xi$ has an unconditional basis and contains no copy of $\ell_1$, the basis is shrinking.  

We next claim that for $0\leqslant \xi<\omega$, $Sz(W_\xi)\leqslant \omega^\xi$.  The base case $\xi=0$ is trivial, since any finite dimensional space has Szlenk index $1=\omega^0$. Assume $Sz(W_\xi)\leqslant\omega^\xi$ for some $0\leqslant \xi<\omega$.   Suppose $W_{\xi+1}=\bigl(\oplus Y_n\bigr)_{\ell_2}$, $Y_n=\oplus_{i=1}^n W_\xi$.  It is known that the Szlenk index of a finite direct sum of separable Banach spaces is simply the maximum of the Szlenk indices of the summands \cite{OSZ}, so $Sz(Y_n)\leqslant \omega^\xi$ for each $n\in \nn$. Moreover, $$Sz\Bigl(\bigl(\oplus_{n=1}^\infty Y_n\bigr)_{\ell_2}\Bigr)\leqslant Sz(W_\xi)Sz(\ell_2)\leqslant \omega^{\xi+1},$$ again by a result from \cite{OSZ}.

Last, we show by induction the following.    

\begin{enumerate}[(i)]  \item $\textbf{I}_1(W_\xi,1), \textbf{I}_p(W_\xi^p), \textbf{I}_\infty(W_\xi^*,1), \textbf{SS}(A_\xi, V_\xi, W_\xi,1)>\omega^\xi$, \item For $0<\xi$,  $\textbf{I}_1(W_\xi), \textbf{I}_p(W_\xi^p), \textbf{I}_\infty(W^*_\xi), \textbf{SS}(A_\xi, V_\xi, W_\xi)=\omega^{\xi+1}$. \end{enumerate}

For $\xi=0$, the assertions of $(i)$ are trivial, as they can be witnessed by any sequence of length $1$ consisting of a normalized vector. In this case, each index is exactly two, since each is bounded by $1+\dim W_0=2$.  

Next, recall that for any $1\leqslant p\leqslant \infty$, $X,Y\in \Ban$, and $K\geqslant 1$, if $\alpha<\textbf{I}_p(X,K)$ and $\beta<\textbf{I}_p(Y,K)$ for some $\alpha, \beta$, then $\textbf{I}_p(X\oplus_p Y, K)>\beta+\alpha$.  This is because if $$T_X=\{(x_i, 0)_{i=1}^n\in B_{X\oplus_p Y}: (x_i)_{i=1}^n\in T_p(X,K)$$ and $$T_Y= \{(0, y_i)_{i=1}^n\in B_{X\oplus_p Y}: (y_i)_{i=1}^n\in T_p(Y,K)\},$$ then $$\{s\verb!^! t:  s\in T_X, t\in T_Y\}\subset T_p(X\oplus_p Y, K).$$  An easy induction argument gives that for each $\eta\leqslant \beta$, $$\{s\verb!^! t:  s\in T_X, t\in T_Y^\eta\}\subset T_p(X\oplus_p Y, K)^\eta.$$ Since $o(T_Y)=\textbf{I}_p(Y,K)>\beta$, we deduce that $T_X\subset T_p(X\oplus_p Y, K)^\beta$.  Again, since $o(T_X)=\textbf{I}_p(X,K)>\alpha$, we deduce that $T_p(X\oplus_p Y,K)^{\beta+\eta}\neq \varnothing$ for each $\eta\leqslant \alpha$, which gives the result.  

Similarly, if $B_1:E_1\to F_1$ and $B_2:E_2\to F_2$, then for any $K\geqslant 1$, if $\textbf{SS}(B_i, E_i, F_i)>\alpha_i$ for $i=1,2$, $$\textbf{SS}(B_1\oplus B_2, E_1\oplus_1 E_2, F_1\oplus_1 E_2)>\alpha_1+\alpha_2.$$  The argument is essentially the same as in the previous paragraph.  

Using this, we prove the successor case of $(i)$.  We deduce from the fact that $\textbf{I}_1(W_\xi,1)>\omega^\xi$ that $\textbf{I}_1(Z_n, 1)>\omega^\xi n$ for each $n$, so that $$\textbf{I}_1(W_{\xi+1}, 1)\geqslant \sup_n \textbf{I}_1(Z_n, 1) \geqslant \sup_n \omega^\xi n=\omega^{\xi+1}.$$ Since $\textbf{I}_1(W_{\xi+1}, 1)$ must be a successor, this inequality must be strict.  The same argument gives the remaining claims of $(i)$ in the successor case.  

To prove $(ii)$ in the successor case, first assume $\xi<\omega$.  Then by Proposition \ref{mixed}, for $1\leqslant p<\infty$, $$\textbf{I}_p(W_{\xi+1}^p), \textbf{SS}(A_{\xi+1}, V_{\xi+1}, W_{\xi+1})\leqslant \omega Sz(W_{\xi+1})=\omega^{\xi+2},$$ and of course all of these inequalities must be equality by $(i)$ and Proposition \ref{special form}.   By Theorem \ref{dualize}, $\textbf{I}_\infty(W_{\xi+1})\leqslant \textbf{I}_1(W_{\xi+1})=\omega^{\xi+2}$, and this must also be equality.  

Next, assume $\xi\geqslant \omega$.  We deduce from Proposition \ref{convexify} that $\textbf{I}_p(W_\xi^p)=\textbf{I}_1(W_\xi)$ in this case. With $W_{\xi+1}=\bigl(\oplus Z_n\bigr)_{\ell_2}$, we deduce from Proposition \ref{finite direct sum} that $\textbf{I}_1(\oplus_{i=1}^n Z_i)=\omega^{\xi+1}$ and from Proposition \ref{infinite direct sum} that $$\textbf{I}_1\Bigl(\bigl(\oplus Z_n\bigr)_{\ell_2}\Bigr)\leqslant \omega^{\xi+1}\textbf{I}_1(\ell_2)=\omega^{\xi+2}.$$ We use Theorem \ref{dualize} to deduce $\textbf{I}_\infty(W^*_{\xi+1})\leqslant \omega^{\xi+2}$.  We deduce from Corollary \ref{mixing ss and ellp} that $$\textbf{SS}(A_{\xi+1}, V_{\xi+1}, W_{\xi+1})\leqslant \omega \textbf{I}_1(W_{\xi+1})=\omega \omega^{\xi+2}=\omega^{\xi+2}.$$ 

Last, suppose $\xi$ is a limit ordinal.  Then since $W_\xi=\bigl(\oplus_{\zeta<\xi} W_\zeta)_{\ell_2([0,\xi))}$, $$\textbf{I}_1(W_\xi, 1)\geqslant \sup_{\zeta<\xi} \textbf{I}_1(W_\zeta,1)\geqslant \omega^\xi.$$  Since $\textbf{I}_1(W_\xi,1)$ must be a successor, this inequality is strict.  The same argument provides the remainder of the estimates of $(i)$.   Since $\xi\geqslant \omega$, Proposition \ref{convexify} guarantees that $\textbf{I}_p(W_\xi^p)=\textbf{I}_1(W_\xi)$ for each $1\leqslant p<\infty$.  For each finite subset $I$ of $[0,\xi)$, $\textbf{I}_1(\oplus_{\zeta\in I} W_\zeta)\leqslant \omega \omega^{\max I+1} <\omega^\xi$ by Proposition \ref{finite direct sum}.  Then Proposition \ref{infinite direct sum} guarantees that $\textbf{I}_1(W_\xi)\leqslant \omega^\xi \textbf{I}_1(\ell_2)=\omega^{\xi+1}$.  By Theorem \ref{dualize}, $\textbf{I}_\infty(W_\xi^*)\leqslant \omega^{\xi+1}$.  By Corollary \ref{mixing ss and ellp}, $\textbf{SS}(A_\xi, V_\xi, W_\xi)\leqslant \omega \omega^{\xi+1}=\omega^{\xi+1}$.

Of course, $W_0$, $W_0^p$, $W_0^*$ can admit no $\ell_p^1$ or $c_0^1$ spreading models, and $A_0\in \sss\sss_1(V_0, W_0)$, since $A_0$ has rank $1$.  For $0<\xi$, The fact that $W_\xi$ does not admit an $\ell_1^1$ spreading model comes from Corollary \ref{spreading model corollary}. The fact that $W_\xi^p$ does not admit an $\ell_p^1$ spreading model then follows from Proposition \ref{convexify spreading model}, and the fact that $W_\xi^*$ does not admit a $c_0^1$ spreading model follows from Theorem \ref{dualize}.  We deduce that $A_\xi\in \sss\sss_1$ from Proposition \ref{strictly singular thing}.  
\end{proof}

\begin{proof}[Proof of Theorem \ref{characterization}]$(i),(ii)$ Note that $\textbf{I}_p(\ell_2)=\omega$ for any $p\in [1,\infty]\setminus \{2\}$.  If we fix $\theta_n\downarrow 0$, let $D:\ell_2\to \ell_2$ be defined by $D(\sum a_ne_n)=\sum \theta_n a_ne_n$ and $D_m\sum a_ne_n=\sum_{n=1}^m \theta_na_ne_n$.  Then $D_m\to D$.  Clearly $\textbf{NP}_p(D, W, W), \textbf{SS}(D,W,W)\geqslant \omega$, since $D$ is not finite rank.  But since $D_m$ is finite rank, $D_m\in \mathfrak{NP}^1_p, \mathfrak{SS}^1$, and since these classes are closed, $D\in \mathfrak{NP}^1_p, \mathfrak{SS}^1$.  Therefore $\textbf{NP}_p(D, \ell_2, \ell_2), \textbf{SS}(D, \ell_2, \ell_2)=\omega$.

Next, assume $\xi$ is any successor exceeding $1$.  Then $$\textbf{I}_1(W_{\xi-1})=\textbf{I}_p(W_{\xi-1}^p)=\textbf{I}_\infty(W_{\xi-1}^*)=\textbf{SS}(A_{\xi-1}, V_{\xi-1}, W_{\xi-1})=\omega^\xi.$$

Last, assume $\xi$ is a limit ordinal.  Note that $\xi$ has countable cofinality if and only if $\omega^\xi$ does.  If $\xi$ has uncountable cofinality, we have already explained why there can exist no operator with $\textbf{NP}_p$ or $\textbf{SS}$ index equal to $\omega^\xi$.  Suppose $\xi$ has countable cofinality and fix $\xi_n\uparrow \xi$, noting that $\xi_n+1\uparrow \xi$.   In the remainder of the proof, $W_\zeta$, $V_\zeta$, and $A_\zeta$ are as in Lemma \ref{successors}.    Let $W=\bigl(\oplus W_{\xi_n}\bigr)_{\ell_2}$ and let $D:W\to W$ be defined by $D|_{W_{\xi_n}}=\theta_n I_{W_{\xi_n}}$, $\theta_n\downarrow 0$. Let $D_m=\sum_{n=1}^m \theta_n I_{W_{\xi_n}}$.  Of course, $\textbf{NP}_1(D)\geqslant \sup_n \textbf{I}_1(W_{\xi_n})=\omega^\xi$.  But $\textbf{NP}_1(D_m, W, W)\leqslant \textbf{I}_1(\oplus_{n=1}^m W_{\xi_n}) \leqslant \omega \omega^{\xi_n+1}<\omega^\xi$ by Proposition \ref{finite direct sum}.  Since $D_m\to D$, $D_m\in \mathfrak{NP}^\xi_1$, and since this class is closed, $\textbf{NP}_1(D, W, W)\leqslant \omega^\xi$, and this must be equality.  We claim that similar diagonal operators $D^p:\bigl(\oplus W_{\xi_n}^p\bigr)_{\ell_2}\to \bigl(\oplus W_{\xi_n}^p\bigr)_{\ell_2}$, $D^*:\bigl(\oplus W_{\xi_n}^*\bigr)\to \bigl(\oplus W_{\xi_n}^*\bigr)$, and $D^{SS}:\bigl(\oplus V_{\xi_n}\bigr)\to \bigl(\oplus W_{\xi_n}\bigr)$ yield the $\textbf{NP}_p$, $\textbf{NP}_\infty$, and $\textbf{SS}$ cases.  The first two of these operators are vanishing multiples of the identities on $W_{\xi_n}^p$ and $W_{\xi_n}^*$, respectively, and the third consists of vanishing multiples of the operators $A_{\xi_n}$.  The estimates $$\textbf{NP}_p(D^p, (\oplus W_{\xi_n}^p), (\oplus W_{\xi_n}^p) ), \textbf{NP}_\infty(D^*, (\oplus W_{\xi_n}^*), (\oplus W_{\xi_n}^*)), \textbf{SS}(D^{SS}, (\oplus V_{\xi_n}), (\oplus W_{\xi_n}))\geqslant \omega^\xi$$ follow as in the $p=1$ case, and it remains to establish the upper estimate.  Let $D_m^p, D_m^*$, and $D_m^{SS}$ be the initial segments of the diagonal operator, as in the $p=1$ case.  It suffices to provide the desired upper estimate for $D^p_m, D^*_m$, and $D_m^{SS}$ for each $m\in \nn$.  We note that $$\textbf{NP}_p(D_m^p, (\oplus W_{\xi_n}^p), (\oplus W_{\xi_n}^p)) \leqslant \textbf{I}_p((\oplus_{n=1}^m W_{\xi_n}^p))\leqslant \omega \textbf{I}_1(\oplus_{n=1}^m W_{\xi_n}) \leqslant \omega^2 \omega^{\xi_m+1}<\omega^\xi$$ using Propositions \ref{convexify}, \ref{finite direct sum}.  By Theorem \ref{dualize}, since $D^*_m$ is the adjoint of $D_m$ as defined in the $p=1$ case, $\textbf{NP}_\infty(D^*_m, W^*, W^*)\leqslant \textbf{NP}_1(D_m,W, W)<\omega^\xi$.  By Propositions \ref{mixing ss and ellp} and \ref{finite direct sum}, $$\textbf{SS}(D_m^{SS}, (\oplus V_{\xi_n}), (\oplus W_{\xi_n}))\leqslant \omega \textbf{I}_1(\oplus_{n=1}^m W_{\xi_n})<\omega^\xi.$$

$(iii)$ and $(iv)$ have already been noted for successors using the identity on one of the spaces $X_\xi$, $X^p_\xi$, $X_\xi^*$, or $X_{\xi,2}$.  Again, appropriate diagonal operators give the limit ordinal cases.

\end{proof}

\begin{corollary} For any $0<\xi\in \ord$, $1\leqslant p\leqslant \infty$, $\cup_{\zeta<\xi}\mathfrak{NP}_p^\zeta$ fails to be closed if and only if the cofinality of $\xi$ is countably infinite, and the same is true of $\cup_{\zeta<\xi}\mathfrak{SS}^\zeta$.  Moreover, for $0<\xi\leqslant \omega_1$, $\cup_{\zeta<\xi} \mathfrak{SM}_p^\zeta$ or $\cup_{\zeta<\xi}\mathfrak{WC}^\zeta$ fails to be closed if and only if $\xi$ is a countable limit ordinal.

\end{corollary}

\begin{proof} If $\xi$ has cofinality $1$, $\xi$ is a successor, say $\xi=\eta+1$, so $\cup_{\zeta<\xi}\mathfrak{NP}_p^\zeta=\mathfrak{NP}_p^\eta$ is closed. The proof of each of the remaining statements for successor ordinals is similar.   

If $\xi$ has uncountable cofinality, fix any $X,Y\in \Ban$ and any sequence $(T_n)\subset \cup_{\zeta<\xi} \mathfrak{NP}_p^\zeta(X,Y)$ with $T_n\to T$ in norm. Then $(T_n)\subset \mathfrak{NP}_p^\eta(X,Y)$, where $\eta=\sup_n \mathbf{NP}_p(T_n, X, Y)<\xi$.  By closedness of $\mathfrak{NP}_p^\eta(X,Y)$, $T\in \mathfrak{NP}_p^\eta(X,Y)\subset \cup_{\zeta<\xi}\mathfrak{NP}_p^\zeta(X,Y)$.  The proof of the statement for $\cup_{\zeta<\xi}\mathfrak{SS}^\zeta$ is similar.  

If $\xi$ has countably infinite cofinality, the diagonal operators in the proof of Theorem \ref{characterization} give examples of operators having index strictly less than $\omega^\xi$ converging to an operator having index $\omega^\xi$.

\end{proof}

\begin{remark} Note that $\cup_{\zeta<\xi} \mathfrak{NP}_p^\zeta$ consists precisely of the operators $A:X\to Y$, $X,Y\in \Ban$, for which $\textbf{NP}_p(A,X,Y)<\omega^\xi$ except in the case that $\xi=1$.  In this case, $\cup_{\zeta<\xi} \mathfrak{NP}_p^\zeta$ consists simply of zero operators, while the the latter class consists of all finite rank operators.  Of course, the latter class also fails to be closed.

\end{remark}

\begin{proof}[Proof of Theorems \ref{distinct1}, \ref{distinct2}, \ref{distinct3}, \ref{distinct4}]

Theorem \ref{distinct1}$(i)$ follows from Theorem \ref{characterization}.  Part $(ii)$ follows from the fact that for any $1\leqslant p, q, \leqslant \infty$ with $p\neq q$, $I_q(\ell_p)\in \{\omega, \omega^2\}$ (or $I_q(c_0)$ if $p=\infty$).  This follows from Proposition \ref{block index}, since if $I_q(\ell_p)>\omega^2$, $\ell_q$ would be block finitely representable in $\ell_p$.  Thus for $1<\zeta$, $I_{\ell_p}\in \mathfrak{NP}_q^\zeta\setminus \mathfrak{NP}_p^\xi$.  For part $(iii)$, note that $I_q(\ell_p)>\omega$ if and only if $\ell_q$ is finitely representable in $\ell_p$, which happens if and only if $p\leqslant q\leqslant 2$.  Thus $I_{\ell_p}\in \mathfrak{NP}_q^1\setminus \mathfrak{NP}_p^\xi$ unless $p\leqslant q\leqslant 2$.  But if $p\leqslant q\leqslant 2$ and $A:X\to Y\notin \mathfrak{NP}^\xi_p$, then there exists $K\geqslant 1$ so that for each $n\in \nn$ we can find $(x_i^n)_{i=1}^n\subset X$ so that $(x_i^n)_{i=1}^n$ and $(Ax_i^n)_{i=1}^n$ are $K$-equivalent to $\ell_p^n$.   Because $\ell_q$ is finitely representable in $\ell_p$,  we can find natural numbers $k_n$ with $k_n\to \infty$ and sequences $(z_i^n)_{i=1}^{k_n}\subset [(x_i^n)]_{i=1}^n$ so that $(z_i^n)_{i=1}^{k_n}$ and its images under $A$ are $C$-equivalent to the $\ell_q^{k_n}$ basis for some $C$ independent of $n$.

For Theorem \ref{distinct2}$(i)$, we have already seen that $X_\xi$, $X_\xi^p$, and $X_\xi^*$ contain $\ell_1^\xi$, $\ell_p^\xi$, and $c_0^\xi$ spreading models and not $\ell_1^{\xi+1}$, $\ell_p^{\xi+1}$, or $c_0^{\xi+1}$ spreading models, respectively. It follows from Proposition \ref{infinite direct sum spreading model}, Proposition \ref{convexify spreading model}, and Theorem \ref{dualize} that if $0<\xi<\omega_1$, $1<r,p<\infty$, and $\xi_n\uparrow \xi$, $(\oplus_n X_{\xi_n,2})_{\ell_r}$, $(\oplus_n X_{\xi_n}^p)_{n\in \ell_{pr}}$, and $(\oplus_n X_{\xi_n}^*)_{\ell_r}$ do not admit $\ell_1^\xi$, $\ell_p^\xi$, or $c_0^\xi$ spreading models, respectively, but do admit all smaller spreading models.  Of course, $\ell_p$ contains all $\ell_p^\xi$ spreading models and no $\ell_q^\zeta$ spreading models, which means that if $\mathfrak{SM}_q^\zeta\subset \mathfrak{SM}_p^\xi$, $p=q$.  This together with Theorem \ref{distinct2}$(i)$ gives $(ii)$.  We also gave an example above, namely $X_{\xi, 2}$, of a space the identity of which lies in $\mathfrak{WC}^{\xi+1}\setminus \mathfrak{WC}^\xi$.  

For Theorem \ref{distinct3}$(ii)$, the examples $W_\xi$, $W_\xi^p$, and $W_\xi^*$ from Proposition \ref{successors} have identity operators lying in $\mathfrak{SM}^1_1\setminus \mathfrak{NP}_1^\xi$, $\mathfrak{SM}_p^1\setminus \mathfrak{NP}_p^\xi$, and $\mathfrak{SM}_\infty^1\setminus \mathfrak{NP}_\infty^\xi$, respectively. These examples also show that $\mathfrak{NP}_p^\xi\neq \mathfrak{SM}_p^\xi$, which is part of $(i)$.  To show the rest of  $(i)$, $\mathfrak{NP}_p^\xi\subset \mathfrak{SM}_p^\xi$, we note that if $A:X\to Y$ and $(x_n)\subset B_X$ is such that $(x_n)_{n\in E}\lesssim_1 (e_n)_{n\in E}$ and $(e_n)_{n\in E}\lesssim_K (Ax_n)_{n\in E}$ for each $E\in \sss_\xi$, where $(e_n)$ is the $\ell_p$ (resp. $c_0$) basis, then $(x_{\max E|_i})_{i=1}^{|E|}\in T_p(A,X,Y,K)$ for each $E\in \sss_\xi$.  One checks by induction on $\zeta$ that $(x_{\max E|_i})_{i=1}^{|E|}\in T_p(A,X,Y,K)^\zeta$ for each $E\in \sss_\xi^\zeta$.  Since $\sss_\xi^\zeta \neq \varnothing $ for each $\zeta\leqslant \omega^\xi+1$, we deduce $\textbf{NP}_p(A,X,Y,K)>\omega^\xi$, and $A\notin \mathfrak{NP}_p^\xi$.

Theorem \ref{distinct4} is similar to Theorem \ref{distinct3} using the examples $A_\xi:V_\xi\to W_\xi$ for $(ii)$.  

\end{proof}

\section{Descriptive set theoretic results}

\subsection{Property $(S')$}
In \cite{D}, a Schauder basis $(e_i)$ was said to have property $(S)$ if whenever $[(e_i)]$ does not embed into either $X$ or $Y$, $[(e_i)]$ does not embed into $X\oplus Y$.  In keeping with \cite{D}, we say that a Schauder basis has property $(S')$ provided that the class $\mathfrak{NP}_{(e_i)}$ is an ideal. Of course, any basis having property $(S')$ must have property $(S)$, since property $(S)$ may be restated as follows: If neither $P_X:X\oplus Y\to X$ nor $P_Y:X\oplus Y\to Y$ preserves a copy of $[(e_i)]$, then $P_X+P_Y$ does not preserve a copy of $[(e_i)]$.   It is obvious that $\mathfrak{NP}_{(e_i)}$ is an ideal if and only if it is closed under finite sums.  It is clear that having property $(S')$ is separably determined.  That is, $\mathfrak{NP}_{(e_i)}$ is an ideal if and only if whenever $A,B:X\to Y$ are operators between separable Banach spaces neither of which preserves a copy of $[(e_i)]$, then $A+B:X\to Y$ does not preserve a copy of $[(e_i)]$.  This is equivalent to the following: Whenever $\xi, \zeta<\omega_1$ and $A,B:X\to Y$ are operators between (not necessarily separable) Banach spaces such that $\textbf{NP}_{(e_i)}(A, X, Y)=\xi$, $\textbf{NP}_{(e_i)}(B,X,Y)=\zeta$, then $\textbf{NP}_{(e_i)}(A+B, X,Y)<\omega_1$.   

Of course, if $A+B:X\to Y$ fails to be strictly singular, then either $A$ or $B$ must fail to be strictly singular.  From this we deduce that if $(e_i)$ is a basis for a minimal Banach space, then $(e_i)$ has property $(S')$.  

\begin{example}\upshape Let $(s_i)$ denote the summing basis of $c_0$, $(f_i)$ the canonical $\ell_2$ basis.  Let $e_i=s_i+f_i\in c_0\oplus_\infty \ell_2=:X$.  Then $(e_i)$ is a normalized Schauder basis for its closed span.  Moreover, if $P_{c_0}:X\to c_0$ and $P_{\ell_2}:X\to \ell_2$ are the projections onto the summands, neither $P_{c_0}$ nor $P_{\ell_2}$ preserves a copy of $[(e_i)]$, while $I_X=P_{c_0}+P_{\ell_2}$ obviously does.  To see that neither projection preserves a copy of $[(e_i)]$, observe that if $(x_i)\subset c_0$ is a bounded sequence so that $\|x_i-x_j\|\geqslant \ee>0$ for all $i,j\in \nn$, $i\neq j$, then there exist $n_1<n_2<\ldots$ so that $(x_{n_{2i}}-x_{n_{2i-1}})$ is equivalent to the $c_0$ basis.  However, for any $n_1<n_2<\ldots$, $(e_{n_{2i}}-e_{n_{2i-1}})$ dominates the $\ell_2$ basis, and there can be no sequence in $c_0$ equivalent to $(e_i)$.  This means that $P_{c_0}$ cannot preserve a copy of $[(e_i)]$.  Next note that since $(e_i)$ is normalized and dominates the summing basis, it is a normalized basic sequence which is not weakly null.  This means $\ell_2$, since it is reflexive, can admit no sequence equivalent to $(e_i)$, and thus $P_{\ell_2}$ does not preserve a copy of $[(e_i)]$. 

\end{example}

We have already established explicit estimates on $\textbf{NP}_p(A+B,X,Y)$ in terms of $\textbf{NP}_p(A,X,Y)$, $\textbf{NP}_p(B,X,Y)$.  These estimates depended on the fact that the trees $T_p(A,X,Y, K)$ are $p$-absolutely convex.  If one defines $(e_i)$-block closed analogously to $p$-absolutely convex, and if one asks what property must possessed by the basis $(e_i)$ in order to guarantee that $T_{(e_i)}(A, X, Y, K)$ is $(e_i)$-block closed, or what property must be possessed so that the weaker but still sufficient condition that there exists $C\geqslant 1$ so that any $(e_i)$-block of $T_{(e_i)}(A, X, Y, K)$ lies in $T_{(e_i)}(A, X, Y, CK)$, one sees that the necessary condition on $(e_i)$ which yields this is perfect homogeneity.  Of course, by Zippin's result \cite{Z}, this means that the arguments we used work only for the $\ell_p$ and $c_0$ bases.  Therefore our combinatorial methods which yielded explicit estimates on $\textbf{NP}_p(A+B,X,Y)$ in terms of $\textbf{NP}_p(A,X,Y)$ and $\textbf{NP}_p(A,X,Y)$ do not yield estimates for other bases.  We will use descriptive set theoretic methods to prove that it is possible to provide a countable upper bound on $\textbf{NP}_{(e_i)}(A+B, X,Y)$ in terms of $\textbf{NP}_{(e_i)}(A,X,Y)$ and $\textbf{NP}_{(e_i)}(B,X,Y)$ when $X, Y$ are separable, $A, B\in \mathfrak{NP}_{(e_i)}$, and $(e_i)$ has property $(S')$.  We recall the following

\begin{theorem}\cite{D} If $(e_i)$ is a Schauder basis with property $(S)$, there exists a function $\psi_{(e_i)}:[1, \omega_1)\times [1, \omega_1)\to [1, \omega_1)$ so that if $X, Y$ are separable Banach spaces neither of which contains a copy of $[(e_i)]$, then $\emph{\textbf{I}}_{(e_i)}(X\oplus Y) \leqslant \psi_{(e_i)}(\emph{\textbf{I}}_{(e_i)}(X), \emph{\textbf{I}}_{(e_i)}(Y))$.  

\end{theorem}

Generalizing this result, the main result of this section is the following 

\begin{theorem} If $(e_i)$ is a Schauder basis with property $(S')$, there exists a function $\phi_{(e_i)}:[1, \omega_1)\times [1, \omega_1)\to [1, \omega_1)$ so that if $X,Y$ are Banach spaces and $A,B\in \mathfrak{NP}_{(e_i)}(X,Y)$ with $\emph{\textbf{NP}}_{(e_i)}(A,X,Y)=\xi<\omega_1$ and $\emph{\textbf{NP}}_{(e_i)}(B,X,Y)=\zeta<\omega_1$, $\emph{\textbf{NP}}_{(e_i)}(A+B, X, Y)\leqslant \phi_{(e_i)}(\xi, \zeta).$ 

\label{our theorem}
\end{theorem}

Note that we do not assume the spaces $X$ and $Y$ are separable.  This is because the property $\textbf{NP}_{(e_i)}(A+B, X, Y)>\phi_{(e_i)}(\xi, \zeta)$ is separably determined.  The fact that we do not need to assume $X$ and $Y$ are separable allows us to deduce the following result immediately from the discussion above and Theorem \ref{our theorem}.  The result is non-trivial, since, as we have seen with our examples $W_\xi$, there may be operators $A:X\to Y$ with $\omega_1<\textbf{NP}_{(e_i)}(A,X,Y)< \infty$.  

\begin{theorem} The class of operators $A:X\to Y$ such that $\emph{\textbf{NP}}_{(e_i)}(A,X,Y)<\omega_1$ is an ideal if and only if $(e_i)$ has property $(S')$.  

\end{theorem}

By our discussion above, if such a function $\phi_{(e_i)}$ exists, then $(e_i)$ must have property $(S')$.  Thus property $(S')$ characterizes the existence of such a function.  In order to prove Theorem \ref{our theorem}, we must establish a few basic facts concerning the coding of operators between separable Banach spaces.

\subsection{The standard space $\mathfrak{L}$}

We first undertake a coding of the operators between separable Banach spaces in the spirit of Bossard's coding \cite{Bos,Bos1} of all separable Banach spaces.  Recall that for any Polish (separable, completely metrizable topological) space $P$, we let $F(P)$ denote the closed subsets of $P$.  We let $E(P)$ be the $\sigma$-algebra generated by sets of the form $\{F\in F(P): F\cap U\neq \varnothing\}$, where $U$ ranges over the open subsets of $P$.  It is known \cite{Ke} that there exists a Polish topology on $F(P)$ so that the Borel $\sigma$-algebra generated by this topology is $E(P)$. We recall the Kuratowski and Ryll-Nardzewski result concerning the existence of Borel selectors: There exists a sequence $d_n:F(P)\setminus \{\varnothing\}\to P$ of Borel functions so that for all $F\in F(P)\setminus \{\varnothing\}$, $d_n(F)\in F$ and the sequence $(d_n(F))_n$ is dense in $F$ \cite{KRN}. We will apply this with $P=\cts$,  the Banach space of continuous functions on the Cantor set $2^\nn$. It is well-known that the set of closed subsets of $\cts$ which are closed subspaces, which we denote \textbf{SB}, is Borel in $F(\cts)$. Therefore there exists a Polish topology on \textbf{SB} so that the Borel $\sigma$-algebra generated by this topology is the relative Effros-Borel structure $E(\cts)|_{\textbf{SB}}$. Through the remainder of this work, \textbf{SB} will be topologized by such a topology to which we omit direct reference.  We let $\textbf{S}=\textbf{SB}\times \textbf{SB}\times \cts^\nn$, endowed with the product topology.  As mentioned above, we may fix a sequence of Borel selectors $d_n:\textbf{SB}\to \cts$.  For $X\in \textbf{SB}$, we let $D_X=\{d_n(X): n\in \nn\}$.

For $(q, n)\in \qn\setminus \{\varnothing\}$ and $K\in \nn$, write $(q,n)=(q_i, n_i)_{i=1}^{|q|}$ and let $$\aaa_K(q,n)= \Bigl\{(X,Y, \hat{A})\in \textbf{S}: K\|\sum_{i=1}^{|q|} q_i d_{n_i}(X)\|\geqslant \| \sum_{i=1}^{|q|} q_i \hat{A}(n_i)\|\Bigr\}.$$  Let $$\aaa_K=\bigcap_{(q,n)\in \qn\setminus \{\varnothing\}} \aaa_K(q,n)$$ and $\aaa=\cup_{K\in \nn}\aaa_K$.  

The map $(X, Y, \hat{A})\underset{M}{\mapsto} K\|\sum_{i=1}^{|q|} q_i d_{n_i}(X)\| - \|\sum_{i=1}^{|q|} q_i \hat{A}(n_i)\|$ is a Borel function.  Then $\aaa_K(q,n)= M^{-1}([0, \infty))$ is Borel, and therefore $\aaa_K$ and $\aaa$ are Borel.

Let $\mathcal{J}= \bigl\{(X,Y, \hat{A})\in \textbf{S}: \hat{A}(n)\in Y \text{\ }\forall n\in \nn\bigr\}.$  Recall \cite{D} that $\mathcal{I}=\{(Z,z)\in \textbf{SB}\times \cts: z\in Z\}$ is Borel. Since $(X,Y, \hat{A})\underset{M_n}{\mapsto} (Y, \hat{A}(n))$ is continuous for each $n\in \nn$, $\mathcal{J}=\cap_{n\in \nn}M^{-1}_n(\mathcal{I})$ is Borel.  

We therefore deduce that $\mathfrak{L}:=\aaa\cap \mathcal{J}$ and $\mathfrak{L}_1:=\aaa_1\cap \mathcal{J}$ are Borel.  We have the following result.  

\begin{claim} For $(X, Y, \hat{A})\in \emph{\textbf{S}}$, the relation $\{(d_n(X), \hat{A}(n)):n\in \nn\}\subset D_X\times \cts$ is the restriction to $D_X$ of an operator (resp. an operator with norm not exceeding $1$) $A:X\to Y$ if and only if $(X, Y, \hat{A})\in \mathfrak{L}$ (resp. $(X,Y, \hat{A})\in \mathfrak{L}_1$).  

\end{claim}

\begin{proof} Assume $(X, Y, \hat{A})$ is such that $Ad_n(X)= \hat{A}(n)$ for all $n\in \nn$, where $A:X\to Y$ is an operator.  Then for any $(q,n)\in \qn\setminus \{\varnothing\}$ and $K\in \nn$, $K\geqslant \|A\|$,  $$\|\sum_{i=1}^{|q|} q_i \hat{A}(n_i)\| = \|\sum_{i=1}^{|q|} q_i A d_{n_i}(X)\| \leqslant K \|\sum_{i=1}^{|q|} q_id_{n_i}(X)\|,$$ and we deduce $(X, Y, \hat{A})\in \aaa_K\subset \aaa$.  Of course, $\hat{A}(n)=Ad_n(X)\in Y$, so $(X, Y, \hat{A})\in \mathcal{J}$, and $(X,Y, \hat{A})\in \mathfrak{L}$.  Moreover, if $\|A\|\leqslant 1$, we obtain the result with $K=1$, so $(X, Y, \hat{A})\in \mathfrak{L}_1$.

Next, assume $(X, Y, \hat{A})\in \mathfrak{L}$.  Let $K\in \nn$ be minimal such that $(X, Y, \hat{A})\in \aaa_K$.   We first show that $f:D_X\to Y$ defined by $f(d_n(X))=\hat{A}(n)\in Y$ is well-defined.  If $d_n(X)=d_m(X)$, then $d_n(X)-d_m(X)=0$.  Since $(X, Y, \hat{A})\in \aaa_K$, $$\|\hat{A}(n)- \hat{A}(m)\|\leqslant K \|d_n(X)- d_m(X)\| =0,$$ and $\hat{A}(n)=\hat{A}(m)$.  This shows that $f:D_X\to Y$ is well-defined.  Also, noting that $\|\hat{A}(n)- \hat{A}(m)\|\leqslant K\|d_n(X)- d_m(X)\|$ shows that $f:D_X\to Y$ is $K$-Lipschitz continuous.  This means that $f$ extends uniquely to a continuous $A:X\to Y$, since $D_X$ is dense in $X$.  Moreover, for any $x\in X$ and any $(n_i)\in \nn^\nn$ such that $d_{n_i}(X)\to x$, $Ax= \lim_i Ad_{n_i}(X)=\lim_i \hat{A}(n_i)$.  It remains to show that $A$ is linear.    Fix $p,q\in \rr$ and sequences of rationals $p_i, q_i$ with $p_i\to p$, $q_i\to q$.  Fix $x,y\in X$.  Choose $(n_i), (m_i), (r_i)\in \nn^\nn$ so that $d_{n_i}(X)\to x$, $d_{m_i}(X)\to y$, and $d_{r_i}(X)\to p x + qy$. This means $pAx + qAy = \lim_i p_i\hat{A}(n_i)+q_i\hat{A}(m_i)$ and $A(px+qy)=\lim_i \hat{A}(r_i)$.  Since $p_id_{n_i}(X)  +q_id_{m_i}(X)- d_{r_i}(X)\to 0$, we deduce \begin{align*} \|p Ax + qAy - A(px+qy) \| & = \lim_i \|p_i \hat{A}(n_i)+ q_i \hat{A}(m_i) - \hat{A}(r_i)\| \\ & \leqslant \lim_i K \|p_id_{n_i}(X)+ q_i d_{m_i}(X) - d_{r_i}(X)\|= 0.\end{align*}

\end{proof}

We will identify triples $(X, Y, \hat{A})\in \mathfrak{L}$ with operators $A:X\to Y$ between separable Banach spaces in the remainder of this work.  

\subsection{$\Pi_1^1$ ranks}

Recall the following facts concerning $\Pi_1^1$ ranks.  These facts can be found in \cite{D}.  

\begin{fact} Let $P$ be a Polish space, $B$ a $\Pi_1^1$ subset of $P$, and $\phi:B\to [0, \omega_1)$ a $\Pi_1^1$ rank on $B$.  Then the following hold: \begin{enumerate}[(i)]\item For every $\xi<\omega_1$, $\{x\in B: \phi(x)\leqslant \xi\}$ is Borel. \item For every analytic subset $A$ of $B$, $\sup\{\phi(x): x\in A\}<\omega_1$.  \end{enumerate}\label{fact1}\end{fact}

\begin{remark} Property $(ii)$ of a (not necessarily $\Pi_1^1$) rank is called \emph{boundedness}.  We will see later that $\textbf{NP}_p$ is a $\Pi_1^1$ rank on $\mathfrak{NP}_p$, while $\textbf{SM}_p$ is not $\Pi_1^1$.  However, since $\textbf{SM}_p\leqslant \textbf{NP}_p$, $\textbf{SM}_p$ will satisfy boundedness.  

\end{remark}

We recall also the following results about Borel reductions.  In what follows, $\textbf{Tr}$ denotes the non-empty trees on $\nn$, considered as a subspace of $2^{\nn^{<\nn}}$, and $\textbf{WF}\subset \textbf{Tr}$ denotes the well-founded members of $\textbf{Tr}$.    

\begin{fact} Suppose $P$ is a Polish space, $A\subset P$, $f:P\to \emph{\textbf{Tr}}$ is Borel, and $f^{-1}(\emph{\textbf{WF}})=A$.  Then $A$ is $\Pi_1^1$ and $\phi(x)=o(f(x))$ defines a $\Pi_1^1$ rank on $A$.  

\end{fact}

This combines Theorem $A.4$ with Fact $A.8$ of \cite{D}.  We will use this to prove that a number of ranks are $\Pi_1^1$ ranks, including $\textbf{NP}_{(e_i)}$ and $\textbf{SS}$.  In what follows, for $t\in \nn^{<\nn}$ and $X\in \textbf{SB}$, let $d(t,X)=(d_{n_1}(X), \ldots, d_{n_l}(X))$ if $t=(n_1, \ldots, n_l)$ and $d(\varnothing, X)=\varnothing$.  

The final fact that we recall concerns $\Pi_1^1$ complete sets.  

\begin{fact}\cite{Ke} If $B\subset P$ is a $\Pi_1^1$ subset of $P$, $P$ a Polish space, then $B$ is $\Pi_1^1$ complete if there exists a Borel function $f:\emph{\textbf{Tr}}\to X$ so that $f^{-1}(B)=\emph{\textbf{WF}}$.  

\label{fact3}
\end{fact}

We recall that any $\Pi_1^1$ complete set is necessarily non-Borel.

\begin{proposition} Define $f:\mathfrak{L}\to \emph{\textbf{Tr}}$ by letting $$f(X, Y, \hat{A})=\{\varnothing\}\cup \Bigl\{k\verb!^!t:k\in \nn, d(t,X)\in T_{(e_i)}(A,X,Y,k)\Bigr\}.$$  Then $f$ is Borel, $f^{-1}(\emph{\textbf{WF}})=\mathfrak{NP}_{(e_i)}\cap \mathfrak{L}$, and $o(f(X,Y, \hat{A}))=\emph{\textbf{NP}}_{(e_i)}(A,X,Y)+1$.  Moreover, if we replace $T_{(e_i)}(A,X, Y,k)$ with $SS(A,X,Y, k)$ in the definition of $f$, then the resulting $f$ is also Borel, $f^{-1}(\emph{\textbf{WF}})=\mathfrak{SS}\cap \mathfrak{L}$, and $o(f(X,Y, \hat{A}))=\emph{\textbf{SS}}(A,X,Y)+1$.  

\label{coanalytic1}

\end{proposition}

\begin{proof} First note that to show that $f$ is Borel, it is sufficient to show that for each $t\in \nn^{<\nn}$, $f^{-1}(t)$ is Borel.  This is because $$\nnn=\Bigl\{ \{T\in \textbf{Tr}: T\cap F=E\}: E\subset F \in [\nn^{<\nn}]^{<\nn}\Bigr\}$$ is a countable neighborhood basis for $\textbf{Tr}$ and for each $E\subset F\in [\nn^{<\nn}]^{<\nn}$, $$f^{-1}(\{T\in \textbf{Tr}: T\cap F=E\})= \bigcap_{t\in E}f^{-1}(t)\setminus \bigcup_{t\in F\setminus E}f^{-1}(t).$$  

Fix $t\in \nn^{<\nn}$.  If $|t|\leqslant 1$, $f^{-1}(t)=\mathfrak{L}$.  Assume $t=(k, n_1, \ldots, n_l)$.  Let $S_1=\{(x,y)\in \rr^2: x\leqslant k y\}$, $S_2=\{(x,y)\in \rr^2: x\leqslant y\}$.  Then for $q=(q_i)\in c_{00}\cap \mathbb{Q}^\nn=:Q$, $$(X, Y, \hat{A})\underset{M_q}{\mapsto} \Bigl(\|\sum_{i=1}^l q_i d_{n_i}(X)\|, \|\sum_{i=1}^l q_i e_i\|\Bigr)$$ and $$(X, Y, \hat{A})\underset{N_q}{\mapsto} \Bigl(\|\sum_{i=1}^l q_i e_i\|, \|\sum_{i=1}^l q_i \hat{A}(n_i)\|\Bigr)$$ are both Borel.  Then $f^{-1}(t)= \cap_{q\in Q} [M_q^{-1}(S_2)\cap N_q^{-1}(S_1)]$ is Borel.  This shows that $f$ is Borel in the first case.  For the strictly singular trees, for $(k, n_1, \ldots, n_l)$ fixed, and for each $q\in Q$ and $1\leqslant m<n$, we consider $$(X, Y, \hat{A})\underset{M_{q,m}}{\mapsto} \Bigl(\|\sum_{i=1}^m q_i d_{n_i}(X)\|, \|\sum_{i=1}^l q_i d_{n_i}(X)\|\Bigr),$$ $$(X, Y, \hat{A})\underset{N_q}{\mapsto} \Bigl(\|\sum_{i=1}^l q_i d_{n_i}(X)\|, \|\sum_{i=1}^l q_i \hat{A}(n_i)\|\Bigr).$$  In this case, $f^{-1}(t)= \cap_{q\in Q} [N_q^{-1}(S_1)\cap \cap_{m=1}^{l-1} M_{q,m}^{-1}(S_1)]$.

If $(X, Y, \hat{A})\in \mathfrak{L}$ is such that $A:X\to Y$ preserves a copy of $[(e_i)]$, we may choose $k\in \nn$ and $(n_i)\in \nn^\nn$ so that $(d_{n_i}(X))$ is $1$-dominated by $(e_i)$ and $(Ad_{n_i}(X))$ $k$-dominates $(e_i)$.  This means $k\verb!^!(n_i)_{i=1}^l\in f(X, Y, \hat{A})$ for all $l\in \nn$, and $o(f(X,Y, \hat{A}))=\infty$.  This means that $f^{-1}(\textbf{WF})\subset \mathfrak{NP}_{(e_i)}$.  Similarly, we deduce in the strictly singular case that $f^{-1}(\textbf{WF})\subset \mathfrak{SS}$.  We next show that $o(f(X,Y, \hat{A}))= \textbf{NP}_{(e_i)}(A,X,Y)+1$, which will yield that $\mathfrak{NP}_{(e_i)}\cap \mathfrak{L}\subset f^{-1}(\textbf{WF})$.  For this, we first observe that for any $T\in \textbf{Tr}$, $o(T)=(\sup_{k\in \nn} o(T(k)))+1$.  This is well-known, and easy to see.  Thus in order to reach the conclusion, we only need to show that $\sup_k o(f(X, Y, \hat{A})(k))=\textbf{NP}_{(e_i)}(A,X,Y)=\sup_k o(T_{(e_i)}(A,X,Y, k))$.  Note that for $(n_1,\ldots, n_l)\mapsto (d_{n_1}(X), \ldots, d_{n_l}(X))$ is a monotone map from $f(X,Y, \hat{A})(k)$ into $T_{(e_i)}(A,X,Y, k)$, whence $o(f(X,Y, \hat{A})(k))\leqslant o(T_{(e_i)}(A,X,Y, k))$.  If $\xi=o(f(X,Y, \hat{A})(k+1))< o(T_{(e_i)}(A,X,Y, k))$, we can choose $(x_t)_{t\in \ttt_\xi}\subset B_X$ so that $(x_{t|_i})_{i=1}^{|t|}\in T(A,X,Y,k)$ for all $t\in \ttt_\xi$.  By scaling $(x_t)_{t\in \ttt_\xi}$ by some $c<1$, $c\approx 1$, we can assume that for every $t\in \ttt_\xi$, $(x_{t|_i})_{i=1}^{|t|}$ is $c$-dominated by $(e_i)_{i=1}^{|t|}$ and $(Ax_{t|_i})_{i=1}^{|t|}$ $(k+1/2)$-dominates $(e_i)_{i=1}^{|t|}$.  Then if $\ee_n\downarrow 0$ rapidly (depending on $c$, $k$, and $\|A\|$), we can choose for each $t\in \ttt_\xi$ some $n_t\in \nn$ so that $\|x_t- d_{n_t}(X)\|<\ee_{|t|}$ and that $(d_{n_{t|_i}}(X))_{i=1}^{|t|}\in T_{(e_i)}(A,X,Y, k+1)$ for each $t\in \ttt_\xi$.  Then $(n_{t|_i})_{i=1}^{|t|}\in f(X, Y, \hat{A})(k+1)$ for each $t\in \ttt_\xi$, yielding that $o(f(X,Y, \hat{A})(k+1))>\xi$, a contradiction.  This yields that $o(f(X,Y, \hat{A})(k))\leqslant o(T_{(e_i)}(A,X,Y, k))\leqslant o(f(X,Y, \hat{A})(k+1))$ for all $k\in \nn$, which finishes the proof.

\end{proof}

\begin{remark} Note that in the proof that $f$ is Borel, we deduced $f^{-1}(t)=\cap_{q\in Q}[M_q^{-1}(S_2)\cap N_q^{-1}(S_1)]$. Fix $0<\xi<\omega_1$. Let $g(X,Y, \hat{A})$ be the tree consisting of $\varnothing$, $(k)$ such that $k\in \nn$, and $(k, n_1, \ldots, n_l)$ so that $(d_{n_i}(X))_{i\in E}$ is $1$-dominated by the $\ell_p^{|E|}$ basis and $(\hat{A}(n_i))_{i\in E}$ $k$-dominates the $\ell_p^{|E|}$ basis for each $E\subset \{1,\ldots, l\}$ such that $E\in \sss_\xi$.  If we let $Q_\xi=\{q\in Q: \supp(q)\in\sss_\xi\}$ and if we fix $t=(k, n_1, \ldots, n_l)$, arguing as in the previous proof, $g^{-1}(t)= \cap_{q\in Q_\xi}[M_q^{-1}(S_2)\cap N_q^{-1}(S_1)]$.  We therefore deduce that $g$ is Borel.  Moreover, $g^{-1}(\textbf{WF})$ consists of all of those $(X, Y, \hat{A})\in \mathfrak{L}$ so that $A$ does not preserve an $\ell_p^\xi$ ($c_0^\xi$ spreading model if $p=\infty$) spreading model.  Thus the sets $\mathfrak{SM}_p^\xi$ are also $\Pi_1^1$.  

Similarly, we can define a map from $\mathfrak{L}\mapsto \textbf{Tr}$ so that $(k, n_1, \ldots, n_l)\in f(X, Y, \hat{A})$ if $(d_{n_i}(X))_{i=1}^l$ is $k$-basic and $(\hat{A}{n_i})_{i\in E}$ $k$-dominates the summing basis for each $E\subset \{1, \ldots, l\}$ with $E\in \sss_\xi$, and deduce that the set $\mathfrak{WC}^\xi \cap \mathfrak{L}$ is $\Pi_1^1$.  We will see later that $\mathfrak{SM}_p^\xi$ and $\mathfrak{WC}^\xi$ are actually $\Pi_1^1$ complete.  

\end{remark}

\begin{proof}[Proof of Theorem \ref{our theorem}] For $\eta<\omega_1$, let $\bbb_\eta=\{(X,Y, \hat{A})\in \mathfrak{L}: \textbf{NP}_{(e_i)}(X,Y, A)\leqslant \eta\}$ and recall that this is a Borel subset of $\mathfrak{L}$.  Observe that $$\textbf{A}:=\{(X_i, Y_i, \hat{A}_i)_{i=1}^3\in \textbf{S}^3: X_1=X_2=X_3, Y_1=Y_2=Y_3, \hat{A}_3(n)=\hat{A}_1(n)+\hat{A}_2(n)\text{\ }\forall n\in \nn\}$$ is closed in $\textbf{S}^3$.  We therefore deduce that $$\bbb:=\mathfrak{L}^3\cap \textbf{A}\cap (\bbb_\xi\times \bbb_\zeta \times \textbf{S})$$ is Borel in $\textbf{S}^3$.  Therefore if $\pi$ is the projection onto the third coordinate of $\textbf{S}^3$, $\aaa:=\pi(\bbb)$ is analytic.  Then $\aaa$ is simply collection of all sums of pairs of operators $A,B:X\to Y$ so that $A\in \bbb_\xi$ and $B\in \bbb_\zeta$, $X,Y\in \textbf{SB}$.  Because $(e_i)$ has property $(S')$, $\aaa\subset \mathfrak{NP}_{(e_i)}$.  By boundedness, $$\phi_{(e_i)}(\xi, \zeta):=\sup \{\textbf{NP}_{(e_i)}(A,X,Y): (X,Y, \hat{A})\in \aaa\}<\omega_1.$$  This implies the conclusion if we only consider operators between separable spaces.  

Next, suppose $X,Y\in \Ban$ are (not necessarily separable) Banach spaces, $\xi, \zeta<\omega_1$, and $\textbf{NP}_{(e_i)}(A,X,Y)=\xi$, $\textbf{NP}_{(e_i)}(B,X,Y)=\zeta$.  If $\textbf{NP}_{(e_i)}(A+B, X, Y, K)>\phi_{(e_i)}(\xi, \zeta)=:\eta$ for some $K\geqslant 1$, choose $(x_t)_{t\in \ttt_\eta}\subset B_X$ so that $(x_{t|_i})_{i=1}^{|t|}\in T_{(e_i)}(A+B,X,Y, K)$ for all $t\in \ttt_\eta$.  Since $\eta<\omega_1$, $\ttt_\eta$ is countable, $W:=[x_t:t\in \ttt_\eta]$ is separable.  Let $Z$ be a closed, separable subspace of $Y$ so that $A$ and $B$ map $W$ into $Z$.  Then the collection $(x_t)_{t\in \ttt_\eta}\subset B_W$ implies that $\eta<\textbf{NP}_{(e_i)}((A+B)|_W, W, Z,K)$.  But using the previous paragraph together with the fact that $\textbf{NP}_{(e_i)}(A|_W, W, Z)\leqslant \textbf{NP}_{(e_i)}(A,X,Y)=\xi$ and $\textbf{NP}_{(e_i)}(B|_W, W, Z)\leqslant \textbf{NP}_{(e_i)}(B,X,Y)=\zeta$, we deduce that $\textbf{NP}_{(e_i)}((A+B)|_W, W, Z)\leqslant \phi_{(e_i)}(\xi, \zeta)$.  This contradiction yields the non-separable case.

\end{proof}

\begin{proposition} For each $0<\xi<\omega_1$ and $1\leqslant p\leqslant \infty$, the classes $\mathfrak{WC}^\xi$ and $\mathfrak{SM}_p^\xi$ are $\Pi_1^1$ complete and therefore non-Borel.  

\label{complete}

\end{proposition}

For this, we will use modifications of the examples considered in \cite{BeD}, which are themselves modifications of the James tree space.  Let $(e_t)_{t\in \nn^{<\nn}}$ denote the canonical basis for $c_{00}(\nn^{<\nn})$.  A finite subset $\mathfrak{s}\subset \nn^{<\nn}$ is called a \emph{segment} if there exist $s,t\in \nn^{<\nn}$ so that $\mathfrak{s}=\{u\in \nn^{<\nn}: s\preceq u \preceq t\}$.   For $1\leqslant p,q< \infty$, let $Z_{p,q}$ be the completion of $c_{00}(\nn^{<\nn})$ under the norm $$\|x\|= \sup \Bigl\{\Bigl(\sum_{i=1}^n \bigl(\sum_{t\in \mathfrak{s}_i} |x(t)|^p\bigr)^{q/p}\Bigr)^q: (\mathfrak{s}_i)_{i=1}^n \text{\ are disjoint segments}\Bigr\}.$$  Note that the norm of a vector $x\in Z_{p,q}$ is at least its norm in $\ell_q(\nn^{<\nn})$, since coordinate projections are projections onto segments of length $1$.  Therefore the basis of $Z_{p,q}$ is boundedly complete.  This means that $Z:=Z_{1,2}$ is therefore naturally the dual of a Banach space $Z_*$ having a shrinking basis the biorthogonal functionals of which are the basis of $Z$.  Given a subset $T$ of $\nn^{<\nn}$, let $Z^T=[e_t:t\in T]$ and let $P^T:Z\to Z^T$ denote the basis projection onto $Z^T$, which has norm $1$ if $T\neq \varnothing$ since $(e_t)$ is a $1$-unconditional basis for $Z$. We let $Z^\varnothing=\{0\}$.  Let $S^T:\ell_1(\nn^{<\nn})\to Z$ be the composition of the formal identity from $\ell_1(\nn^{<\nn})$ to $Z$ with the projection $P^T$. 

\begin{proposition} If $T\in \emph{\textbf{WF}}$, then $S^T$ fails to preserve an $\ell_1^1$ spreading model.  
\label{well founded 1}
\end{proposition}

\begin{proof} We will show by induction on $\xi$ that if $T\in \textbf{WF}$ is such that $o(T)\leqslant \xi+1$, then $Z^T$ does not admit an $\ell_1^1$ spreading model, which yields the result. We first recall, as we have already mentioned, that $o(T)=(\sup_k o(T(k)))+1$.  In particular, $o(T)>o(T(k))$ for all $k\in \nn$. We also recall that if $T$ is a non-empty, well-founded tree, $o(T)$ is a successor, since all non-empty trees contain the empty sequence.   Note that $T(k)$ may be empty, but this will cause no problems.  First, if $o(T)\leqslant 1$, then $T=\{\varnothing\}$, and $Z^T$ is one-dimensional.  Thus the result is trivial in this case.  

Next, assume $0<\xi<\omega_1$ and for every $0\leqslant \zeta<\xi$, the result holds for every well-founded tree $T$ on $\nn$ with order not exceeding $\zeta+1$.  Suppose $T\in \textbf{WF}$ is such that $o(T)\leqslant \xi+1$.  Note that $\ker(e_\varnothing^*)\cap Z^T=(\oplus_k Z^{T(k)})_{\ell_2}$ isometrically. To see that these spaces are isometrically isomorphic, first note that we can partition $(e_t:t\in T\setminus \{\varnothing\})$ into the sets $(e_t: t\in T, (k)\preceq t)$, and for distinct $k,l\in \nn$, any vectors $x$ and $y$ supported in $[e_t: t\in T, (k)\preceq t]$ and $[e_t:t\in T, (l)\preceq t]$, respectively, the members of the supports of $x$ and $y$ are incomparable, which means $\|x+y\|^2= \|x\|^2+\|y\|^2$.  Moreover, the identification $e_{k \verb!^!t}\leftrightarrow e_t$ between $(e_t: t\in T, (k)\preceq t)$ and $(e_t: t\in T(k))=(e_t: t\in \nn^{<\nn}, k\verb!^!t\in T)$ extends to an isometric isomorphism between $[e_t:t\in T, (k)\preceq t]$ and $Z^{T(k)}$.  Thus $\ker(e^*_\varnothing)\cap Z^T$ is isometrically isomorphic to $(\oplus_k Z^{T(k)})_{\ell_2}$.  For each $k\in \nn$, $o(T(k))\leqslant \zeta+1$ for some $\zeta<\xi$, whence $Z^{T(k)}$ does not admit an $\ell_1^1$ spreading model by the inductive hypothesis.  Thus $\ker(e_\varnothing^*)\cap Z^T$ is the $\ell_2$ sum of Banach spaces none of which admits an $\ell_1^1$ spreading model, whence $\ker(e_\varnothing^*)\cap Z^T$, and therefore $Z^T$, does not admit an $\ell_1^1$ spreading model.

\end{proof}

\begin{remark} Note that if $T\in \textbf{Tr}$, $S^T$ clearly preserves a copy of $\ell_1$ if $T$ is ill-founded.  That is, if $(n_i)\in \nn^\nn$ is such that $(n_i)_{i=1}^l\in T$ for all $l\in \nn$, then $(e_{(n_1, \ldots, n_l)})_{l=1}^\infty \subset \ell_1(\nn^{<\nn})$ is isometrically equivalent to the $\ell_1$ basis, and so is its image under $S^T$.  But since $S^T$ is a diagonal operator between spaces with unconditional bases, $S^T$ fails to preserve a copy of $\ell_1$ if and only if $ (S^T)^*\subset [e_t^*:t\in \nn^{<\nn}]\subset \ell_1(\nn^{<\nn})^*$, which happens if and only if $T$ is well-founded.  It is easy to see that in the case that $T$ is well-founded, $S^T$ must actually be weakly compact.  Therefore $S^T$ is weakly compact if and only if $T$ is well-founded if and only if $S^T$ fails to preserve a copy of $\ell_1$ if and only if $S^T$ fails to preserve an $\ell_1^1$ spreading model.  Thus $\{S^T: T\in \textbf{WF}\}\subset \mathfrak{WC}^1\cap \mathfrak{L}$.

\end{remark}

Note that $S^T$ takes disjointly supported vectors in $\ell_1(\nn^{<\nn})$ to disjointly supported vectors in $Z$. Note also that the $p$-convexification of $Z^T$ is $Z_{p, 2p}^T$ for any $T\in \textbf{Tr}$.  Therefore $S^T$ has a $p$-convexification $S^T_p:\ell_p(\nn^{<\nn})\to Z_{p, 2p}$.  Note also that $S^T$ is the adjoint of a map $S_*^T:Z_*\to c_0(\nn^{<\nn})$, where $S^T_*$ is the composition of the projection $P_*^T:Z_*\to Z_*^T$ with the formal identity from $Z_*$ into $c_0(\nn^{<\nn})$.

\begin{corollary} If $T\in \emph{\textbf{WF}}$, then $S^T_p:\ell_p(\nn^{<\nn})\to Z_{p,2p}^T$ fails to preserve an $\ell_p^1$ spreading model.  The preadjoint $S^T_*:Z_*^T\to c_0(\nn^{<\nn})$ of $S^T:\ell_1(\nn^{<\nn})\to Z$ fails to preserve a $c_0^1$ spreading model. 
\label{well founded 2}
\end{corollary}

\begin{proof} The space $Z_{p, 2p}^T$ is just the $p$-convexification of $Z^T$, and so $Z_{p, 2p}^T$ cannot admit an $\ell_p^1$ spreading model unless $Z^T$ admits an $\ell_1^1$ spreading model, which it does not.  Similarly, $Z_*^T$ is a predual of $Z^T$.  If $Z_*^T$ were to admit a $c_0^1$ spreading model, $Z^T$ would admit an $\ell_1^1$ spreading model.

\end{proof}

\begin{proof}[Proof of Proposition \ref{complete}]
We have already seen that each of these classes is $\Pi_1^1$.  To see that these sets are not analytic, one can simply observe that if $\mathfrak{SM}_p^\xi\cap \mathfrak{L}$ is analytic, then it is an analytic subset of $\mathfrak{NP}_p\cap \mathfrak{L}$, and boundedness of $\textbf{NP}_p$ on analytic subsets of $\mathfrak{NP}_p\cap \mathfrak{L}$ would yield that $\sup\{\textbf{NP}_p(A,X,Y):(X, Y, \hat{A})\in \mathfrak{SM}_p^\xi \cap \mathfrak{L}\}$ must be countable.  But we have already seen that for $0<\zeta<\omega_1$, the identity on one of the spaces $W_\zeta$, $W_\zeta^p$, or $W_\zeta^*$ (depending on if $p=1$, $1<p<\infty$, or $p=\infty$) has $\textbf{NP}_p$ index exceeding $\omega^\zeta$, but lies in $\mathfrak{SM}_p^1\cap \mathfrak{L}$.  Moreover, since $W_\xi$ is reflexive for all $\xi$, this yields the result for $\mathfrak{WC}^\xi\cap \mathfrak{L}$.  But we will see the formally stronger statement that these sets are $\Pi_1^1$ complete.  

For the remainder of the proof, we will endow each space $\mathfrak{L}(X,Y)$ with the strong operator topology.  First we note that for $X, Y\in \textbf{SB}$, the map from $\mathfrak{L}(X,Y)$ into $\mathfrak{L}$ given by $A\mapsto (X, Y, \hat{A})$ is continuous. To verify this, since the first two components $X$ and $Y$ are fixed, it is sufficient to show that any net $(S_\lambda)\subset \mathfrak{L}(X,Y)$ converging SOT to $S$ has $(S_\lambda d_n(X))_n$ converging to $(S d_n(X))_n$ in $\cts^\nn$.  But this is simply $S_\lambda d_n(X)\to S d_n(X)$ for each $n\in \nn$, which is implied by SOT convergence.  

Let $X\in \textbf{SB}$ be isometrically isomorphic to $\ell_1(\nn^{<\nn})$ and $Y\in \textbf{SB}$ be isometrically isomorphic to $Z$.  Note that we can identify $\mathfrak{L}(\ell_1(\nn^{<\nn}), Z)$ and $\mathfrak{L}(X, Y)$, and this identification forms a (SOT-SOT) homeomorphism between these spaces. Let $\Phi:\mathfrak{L}(\ell_1(\nn^{<\nn}), Z)\to \mathfrak{L}(X,Y)$ be this identification.  Then the map from $\textbf{Tr}$ to $\mathfrak{L}$ defined by $$\underset{\in \textbf{Tr}}{T}\mapsto \underset{\in\mathfrak{L}(\ell_1(\nn^{<\nn}), Z)}{S^T} \mapsto \underset{\in\mathfrak{L}(X,Y)}{\Phi(S^T)} \mapsto \underset{\in\mathfrak{L}}{(X, Y, \widehat{\Phi(S^T)})}$$ is continuous, once we show that $T\mapsto S^T$ is continuous.  Similar arguments will yield that $T\mapsto S^T_p\in \mathfrak{L}(\ell_p(\nn^{<\nn}), Z_{p, 2p})$ and $T\mapsto S_*^T\in \mathfrak{L}(Z_*, c_0(\nn^{<\nn}))$ are also continuous.  We first show how this finishes the proof, and then return to proving continuity.

We first complete the $p=1$ case, with the $1<p<\infty$ case following by the analogous steps with the $p$-convexifications of the operators, and the $p=\infty$ case following by taking the preadjoints.  Note that we have defined a continuous function $f:\textbf{Tr}\mapsto \mathfrak{L}$.  Moreover,for each $0<\xi<\omega_1$, our previous remarks yield that $f^{-1}(\mathfrak{SM}_1^\xi\cap \mathfrak{L})=f^{-1}(\mathfrak{WC}^1\cap \mathfrak{L})=\textbf{WF}$.  Thus by Fact \ref{fact3} and our above sketch that $\mathfrak{SM}_1^\xi\cap \mathfrak{L}$ and $\mathfrak{WC}^\xi\cap \mathfrak{L}$ are $\Pi_1^1$, we deduce that these classes are $\Pi_1^1$ complete.

We return to the proof of continuity of $T\mapsto S^T$.  This will follow from the following: If $P_\lambda\underset{SOT}{\to}P\in \mathfrak{L}(X,X)$ and $S\in \mathfrak{L}(X,Y)$, then $SP_\lambda\underset{SOT}{\to}SP$.  Similarly, if $P_\lambda\underset{SOT}{\to}P\in \mathfrak{L}(Y,Y)$ and $S\in \mathfrak{L}(X,Y)$, then $P_\lambda S\underset{SOT}{\to}PS$.  Finally, if $(e_i)_{i\in \Lambda}$ is an unconditional basis for $X$, then the map from $2^\Lambda$ to $\mathfrak{L}(X,X)$ given by $J\mapsto P_J$ is continuous.  To see this, suppose $J_\lambda \to J$.  By unconditionality, $(P_{J_\lambda})$ is uniformly bounded, and it is sufficient to check pointwise convergence $P_{J_\lambda}x\to P_Jx$ for all $x$ in a dense subset to conclude that $P_{J_\lambda}\underset{SOT}{\to}P_J$.   To that end, we check that this is true for all finitely supported vectors in $X$.  Fix $x$ with finite support and for each $i\in \supp(x)$, note that $1_{J_\lambda}(i)=1_J(i)$ eventually by definition of convergence in $2^\Lambda$.  Thus $P_{J_\lambda}x=P_J x$ eventually.

\end{proof}

\section{Open questions and discussion}

\subsection{Ideals}

We begin with the most natural question. 

\begin{question} For which ordinals $\xi$, $1\leqslant p\leqslant \infty$, and normalized Schauder bases $(e_i)$ are the following classes ideals? \begin{enumerate}[(i)]\item $\mathfrak{NP}_p^\xi$ \item $\mathfrak{SM}_p^\xi$ \item $\{A:X\to Y: \emph{\textbf{NP}}_{(e_i)}(A,X,Y)\leqslant \xi\}$ \item $\mathfrak{SS}^\xi$ \end{enumerate}

\end{question}

A natural step to showing classes above are ideals is to improve the product estimates provided in this work.  Recall the index $\iota$ defined on non-empty regular families by $$\iota(\fff)=\min \{\xi: \fff^\xi=\{\varnothing\}\}.$$  Recall that if $A\in \mathfrak{SM}_p^\xi(X,Y)$ and $B\in \mathfrak{SM}_p^\zeta(X,Y)$ for some $X,Y\in \Ban$ and $1\leqslant p\leqslant \infty$ and $0<\xi, \zeta<\omega_1$, then $A+B\in \mathfrak{SM}_p^{\xi+\zeta}(X,Y)$.  Since the quantified complexity of $\sss_\xi$ is $\iota(\sss_\xi)=\omega^\xi$, we see that this estimate essentially multiplies complexity.  That is, estimates of complexity $\omega^\xi$ and $\omega^\zeta$ on $A$ and $B$, respectively, yield an estimate on the complexity of the sum $A+B$ of $\omega^{\xi+\zeta}=\omega^\xi\omega^\zeta$.  This is in complete analogy to the local case, where $\textbf{NP}_p(A+B, X, Y)\leqslant \textbf{NP}_p(A,X,Y)\textbf{NP}_p(B,X,Y)$.  

\begin{question} Are the product estimates optimal? 

\end{question}

We have already seen that if $p=1$ or $p=\infty$, better estimates are possible for the spreading model indices. We have seen that a better estimate is possible for $\textbf{NP}_\infty$, and for the $\textbf{NP}_1$ index when the spaces involved have unconditional bases.    

\subsection{Weak compactness}

Let $(s_i)$ be the summing basis for $c_0$.  It follows from standard techniques modifying James's characterization of reflexivity that an operator $A:X\to Y$ fails to be weakly compact if and only if there exists $(x_i)\subset B_X$ so that $(Ax_i)$ dominates $(s_i)$.  Therefore for any operator $A:X\to Y$ and $K\geqslant 1$, we define $$\textbf{WC}(A,X,Y,K)=\{\varnothing\}\cup \Bigl\{(x_i)_{i=1}^n \in B_X^{<\nn}: (s_i)_{i=1}^n \lesssim_K (Ax_i)_{i=1}^n\Bigr\}.$$  We then let $$\textbf{WC}(A,X,Y)=\sup_{K\geqslant 1} o(\textbf{WC}(A,X,Y,K)).$$  Then $A$ is weakly compact if and only if $\textbf{WC}(A,X,Y)<\infty$, which follows immediately from the definition.  We make the following easy observations: 

\begin{proposition}\begin{enumerate}[(i)]\item The class of operators $A:X\to Y$ such that $\emph{\textbf{WC}}(A,X,Y)\leqslant \omega$ is the ideal of super weakly compact operators.  \item The class of operators $A:X\to Y$ such that $\emph{\textbf{WC}}(A,X,Y)<\omega_1$ is an ideal.  \end{enumerate}\end{proposition}

The proof of part $(i)$ is essentially the same as the proof that when $(e_i)$ has property $(S')$, the class of operators $A:X\to Y$ so that $\textbf{NP}_{(e_i)}(A,X,Y)\leqslant \omega$ is the ideal of all operators all ultrapowers of which fail to preserve a copy of $(e_i)$.

Of course, part $(ii)$ would be trivial if we restricted our attention to separable domains, since if $X$ is separable and $A:X\to Y$ is an operator, $\textbf{WC}(A,X,Y)$ is either countable or $\infty$.   But since the $\textbf{NP}_1$ index of an operator cannot be larger than the $\textbf{WC}$ index, the identity operators on the reflexive examples $W_\xi$ yield weakly compact operators having uncountable $\textbf{WC}$ index.  Thus part $(ii)$ is non-trivial.  Our proof of part $(ii)$ follows by another descriptive set theoretic argument.  Define the function $f:\mathfrak{L}\to \textbf{Tr}$ by $$(X,Y, \hat{A})\mapsto \{\varnothing\}\cup \Bigl\{(k), k\verb!^!(n_i)_{i=1}^l: (d_{n_i}(X))_{i=1}^l \in \textbf{WC}(A,X,Y, k)\Bigr\}.$$ Then by the same methods as in Lemma \ref{coanalytic1}, $(X,Y, \hat{A})\mapsto o(f(X, Y, \hat{A}))$ is a $\Pi_1^1$ rank on the $\Pi_1^1$ subset $\mathfrak{WC}\cap \mathfrak{L}=f^{-1}(\textbf{WF})$ such that $o(f(X,Y, \hat{A}))= \textbf{WC}(A,X,Y)+1$.  With this we establish the analogue of Theorem \ref{our theorem}: There exists a function $\varphi_{WC}:[1, \omega_1)\times [1, \omega_1)\to [1, \omega_1)$ so that if $X,Y\in \Ban$ (not necessarily seaparable), $\xi, \zeta<\omega_1$ are such that $\textbf{WC}(A,X,Y)\leqslant \xi$ and $\textbf{WC}(B,X,Y)\leqslant \zeta$, then $\textbf{WC}(A+B,X,Y)\leqslant \varphi_{WC}(\xi, \zeta)$.  The proof is an inessential modification of the proof of Theorem \ref{our theorem}.  

\begin{question} For which ordinals $\xi$ is the class of operators $A:X\to Y$ such that $\emph{\textbf{WC}}(A,X,Y)\leqslant \xi$ an ideal? 

\end{question}

\subsection{Other applications of $\mathfrak{L}$ and $\mathfrak{L}_1$} In some cases, it is perhaps more convenient to code only the operators having norm not exceeding $1$.  One convenience of $\mathfrak{L}_1$ is the following concerning the Szlenk index $Sz(A)$ of an operator.  Recall that if $X$ is a separable Banach space, $A:X\to Y$ is an Asplund operator if and only if $A^*B_{Y^*}$ is norm separable.  Let $\mathfrak{A}$ denote the ideal of Asplund operators.  

\begin{proposition} The class $\mathfrak{A}\cap \mathfrak{L}_1$ is a $\Pi_1^1$ subset of $\mathfrak{L}_1$ and the Szlenk index $(X, Y, \hat{A})\mapsto Sz(A)$ is a $\Pi_1^1$ rank on $\mathfrak{A}\cap \mathfrak{L}_1$.  
\end{proposition}

\begin{proof} We follow the argument from \cite{D}, the ideas of which have their origins in \cite{Bos}, where it was shown that the class \textbf{SD} of Banach spaces having separable dual is a $\Pi_1^1$ subset of \textbf{SB} and the Szlenk index is a $\Pi_1^1$ rank on \textbf{SD}. Of course, \textbf{SD} is simply the class of separable Banach spaces whose identity operators lie in $\mathfrak{A}$.  

Let $H=(B_{\ell_\infty}, \sigma(\ell_\infty, \ell_1))$, and note that this set is compact metrizable.  Then $\Omega=\{F\in F(H): F \text{\ is norm separable}\}$ is a $\Pi_1^1$ subset of $H$ and the index $\sup_n |F|_{D_n}$ is a $\Pi_1^1$ rank on $\Omega$.  We do not define the indices $|\cdot|_{D_n}$, only state the relevant properties as necessary.  

For each $n\in \nn$, the map $s_n:\textbf{SB}\to \cts$ defined by $s_n(X) = d_n(X)/\|d_n(X)\|$ if $d_n(X)\neq 0$ and $s_n(X)=0$ is Borel and $\{s_n(X): n\in \nn\}$ is dense in $S_X$ for all $X\in \textbf{SB}$.  For $A^*y^*\in A^*B_{Y^*}$, we let $f_{A^*y^*}=(A^*y^* s_n(X))$.  Then one easily observes that $A^*y^*\leftrightarrow f_{A^*y^*}$ is a homeomorphism between $(A^*B_{Y^*}, \sigma(X^*, X))$ and its image, call it $F_{(X,Y, \hat{A})}\in F(H)$, which preserves norm distances.  Then $(X,Y,\hat{A})\in \mathfrak{A}\cap \mathfrak{L}_1$ if and only if $F_{(X,Y, \hat{A})}\in \Omega$ and $Sz(A^*B_{Y^*})= \sup_n |F_{(X,Y, \hat{A})}|_{D_n}$.  Let $D=\cup_{(X,Y, \hat{A})\in \mathfrak{L}_1} \{(X,Y, \hat{A})\}\times F_{(X,Y, \hat{A})} \subset \mathfrak{L}_1\times H$.  Note that $D$ is Borel.  Since each section $D_{(X,Y, \hat{A})}= F_{(X, Y, \hat{A})}$ is compact, the map $(X, Y, \hat{A})\mapsto F_{(X,Y, \hat{A})}$ is Borel \cite{Ke}.  Thus the map $(X,Y, \hat{A})\underset{\Phi}{\mapsto} F_{(X,Y, \hat{A})}$ is a Borel reduction, $\Phi^{-1}(\Omega)=\mathfrak{A}\cap \mathfrak{L}_1$, and $Sz(A)=\sup_n |F_{(X,Y, \hat{A})}|_{D_n}$ is a $\Pi_1^1$ rank on $\mathfrak{A}\cap \mathfrak{L}_1$.

\end{proof}

\end{document}